\documentclass[letterpaper,11 pt]{article}
\newcommand{\f}{{\tilde{f}}}
\newcommand{\F}{{\tilde{F}}}
\usepackage{graphicx}
\usepackage{amsfonts,amsmath,fullpage,bbm}
\usepackage{amssymb,amsthm,multirow,verbatim}
\usepackage{acronym,wrapfig,plain,mathrsfs,enumerate,relsize,color}
\newtheorem{algorithm}{Algorithm}
\usepackage{algorithm}
\newtheorem{theorem}{Theorem}
\newtheorem{corollary}{Corollary}
\newtheorem{lemma}{Lemma}

\newtheorem{remark}{Remark}
\newtheorem{proposition}{Proposition}
\newtheorem{assumption}{Assumption}
\newtheorem{definition}{Definition}
\usepackage{subfig}
\usepackage{algorithmic}
\usepackage{pifont}
\usepackage{footmisc}

\newcommand{\tabincell}[2]{\begin{tabular}{@{}#1@{}}#2\end{tabular}}

\newcommand{\us}[1]{{\color{black}#1}}
\newcommand{\uss}[1]{{\color{black}#1}}

\newcommand{\blue}[1]{{\color{black}#1}}
\newcommand{\afo}[1]{{\color{black}#1}}

\newcommand{\usv}[1]{{\color{black}#1}}
\newcommand{\uvv}[1]{{\color{black}#1}}
\newcommand{\redd}[1]{{\color{black}#1}}
\newcommand{\magg}[1]{{\color{black}#1}}
\newcommand{\jb}[1]{{\color{black}#1}}
\newcommand{\blu}[1]{{\color{black}#1}}
\begin{document}
\allowdisplaybreaks
\title{Smoothed Variable Sample-size Accelerated Proximal Methods for Nonsmooth Stochastic Convex Programs
}


\author{A.~Jalilzadeh\footnote{University of Arizona, Tucson, AZ 85721 E-mail: afrooz@arizona.edu}~\and
        U.~V.~Shanbhag\footnotemark[\value{footnote}, \footnote{Pennsylvania State University, University Park, PA 16803 E-mail: udaybag@psu.edu}]~\and
                J.~Blanchet\footnotemark[\value{footnote}, \footnote{Stanford University, Stanford, CA 94305 E-mail: jblanche@stanford.edu}]~\and
	P.~W.~Glynn\footnotemark[\value{footnote}, \footnote{Stanford University, Stanford, CA 94305 Email:glynn@stanford.edu}]}




\date{}

\maketitle
\begin{abstract}
We consider the unconstrained minimization of the function $F$, where $F=f+g$,
$f$ is an expectation-valued nonsmooth convex or strongly convex function, and
$g$ is a closed, convex, and proper function. (I) {\bf Strongly convex
$f$.} When ${f}$ is  $\mu$-strongly
convex in $x$, traditional stochastic subgradient schemes ({\bf SSG})
 often display poor behavior, arising in part from noisy subgradients and
diminishing steplengths.  Instead, we apply a variable sample-size accelerated
proximal scheme ({\bf VS-APM}) on $F_{\eta}$, the Moreau envelope of $F$; we
term such a scheme as ({\bf mVS-APM}) and in contrast with ({\bf SSG}) schemes,
({\bf mVS-APM}) utilizes {\em constant} steplengths and {\em increasingly
exact} gradients. We consider two settings. (a) {\em Bounded domains.} In this
setting,  {\bf (mVS-APM)}  displays linear convergence in inexact gradient
steps, each of which requires utilizing an inner {\bf (prox-SSG)} scheme.
Specifically, {\bf (mVS-APM)}  achieves an optimal oracle complexity in {{\bf
prox-SSG} steps} of $\mathcal{O}({1}/{\epsilon})$ with an iteration complexity
of $\mathcal{O}( \log(1/\epsilon))$ in inexact (outer) gradients of $F_{\eta}$
to achieve  an $\epsilon$-accurate solution in mean-squared error, computed via
an increasing number of inner (stochastic) subgradient steps;  (b) {\em
Unbounded domains.} In this regime,  under an assumption of state-dependent
bounds on subgradients, an unaccelerated variant ({\bf mVS-PM}) is linearly
convergent where increasingly exact gradients $\nabla_x F_{\eta}(x)$ are
approximated with increasing accuracy via {\bf (SSG)} schemes.  Notably, {\bf
(mVS-PM)} also displays an optimal oracle complexity of
$\mathcal{O}(1/\epsilon)$; (II) {\bf Convex $f$.} When $f$ is merely convex but
smoothable, by suitable choices of the smoothing, steplength, and batch-size
sequences, smoothed ({\bf VS-APM}) (or {\bf sVS-APM}) achieves an optimal
oracle complexity of $\mathcal{O}(1/\epsilon^2)$ to obtain an
$\epsilon$-optimal solution. Our results can be specialized to two important
cases: (a) {\bf Smooth $f$.} Since smoothing is no longer required,  we observe
that ({\bf VS-APM}) admits the optimal rate and oracle complexity, matching
prior findings; (b) {\bf Deterministic nonsmooth $f$.} In the nonsmooth
deterministic regime, ({\bf sVS-APM}) reduces to a smoothed accelerated
proximal method  ({\bf s-APM}) that is both asymptotically convergent and
optimal in that it displays a  complexity of
$\mathcal{O}(1/\epsilon)$, matching the bound provided by
Nesterov in 2005 for producing $\epsilon$-optimal solutions. Finally, ({\bf
sVS-APM}) and ({\bf VS-APM}) produce sequences that converge almost surely to a
solution of the original problem.  
\end{abstract}
\section{Introduction}
\label{intro}
We consider the following stochastic nonsmooth convex optimization problem 
\begin{align}\label{main problem} \min_{x\in
\mathbb{R}^n} \ F(x), \mbox{ where } { F(x) \triangleq f(x) + g(x)},  \end{align}  
$f(x)\triangleq\mathbb{E}[{\f}(x,\xi(\omega))]$, $\xi : \Omega
\rightarrow \mathbb{R}^o$, ${\f}: \mathbb{R}^n \times \mathbb{R}^o \rightarrow
\mathbb{R}$, {${g}$ is a closed, convex, and proper deterministic function
with an efficient proximal evaluation,} $(\Omega,\mathcal{H},\mathbb{P})$
denotes the associated probability space, and $\mathbb{E}[\bullet]$ denotes the
expectation with respect to the probability measure $\mathbb{P}$. Throughout,
we refer to ${\f}(x,\xi(\omega))$ by ${\f}(x,\omega)$, whereas ${\F}(x,\omega) \triangleq {\f}(x,\omega)+g(x)$. We consider settings where
${{\f}(\cdot,\omega)}$ is nonsmooth strongly convex/convex in $x$ for every
$\omega$, {\bf generalizing the focus} beyond the {\em structured
nonsmooth} setting where the ``stochastic part'' is smooth.  Specifically,
structured nonsmooth problems require minimizing ${f(x)} + g(x)$ where $f$
is smooth and $g$ is nonsmooth with an efficient prox evaluation (allows for capturing constrained problems over closed and convex
sets). 

Amongst the earliest avenues for resolving~\eqref{main problem} is stochastic
approximation~\cite{robbins51sa,kushner03stochastic} and has proven to be effective on a
breadth of stochastic computational problems including convex optimization problems. 
\blue{\cite{Polyak92acceleration} developed {an} averaging
scheme in  convex differentiable settings, deriving the optimal convergence
rate of $\mathcal O(1/\sqrt K)$ under classical assumptions, where $k$ is the
number of iterations. {Amongst the cleanest of early complexity
requirements for the minimization of expectation-valued $\mu$-strongly convex
and convex functions over a closed and convex set $X$ were given by 
$\left(\max
\left\{  {M^2}/{\mu^2},\|x_0-x^*\|^2 \right\} ({1}/{\epsilon})\right)$
(to ensure that $\mathbb{E}[\|x_k-x^*\|^2] \leq \epsilon$) and  $\mathcal
O({M D_X}/\epsilon^2)$ (to ensure that {the expected optimality gap is less than $\epsilon$}), respectively  where $S(x,\omega)$ denotes a measurable selection from $\partial_x \tilde{f}(x,\omega)$, $\sup_{x \in X}
\mathbb{E}[\|S(x,\omega)\|^2] \leq M^2$ and $D_X \triangleq \displaystyle
\max_{x \in X} \|x_0-x\|$.  Of these, the former was presented by
\cite{shapiro09lectures} whereas the latter is the result of an optimal robust
constant steplength SA scheme suggested by~\cite{nemirovski_robust_2009}.} When
$f$ is both $L$-smooth and $\mu$-strongly convex, an improved complexity requirement (from a constant factor standpoint) of
$\mathcal{O}(
\sqrt{{(L\|x_0-x^*\|^2}/{\epsilon)}}+{\nu^2}/(\mu \epsilon))$ was
provided by \cite{2013optimal}.    } 
This contrasts sharply with the deterministic regime where
$\mathcal{O}(\log(1/\epsilon))$ and $\mathcal{O}(1/\sqrt{\epsilon})$ steps are
required in smooth strongly convex and \jb{smooth convex regimes to compute an
\blu{$\epsilon$-accurate solution ($\epsilon$-solution in terms of mean-squared
error) and $\epsilon$-optimal solution ($\epsilon$-solution in terms of
expected sub-optimality)},} respectively. In structured nonsmooth regimes,
there has been an effort to employ the stochastic generalization of an
accelerated proximal gradient method to minimize $\uss{f+g}$ when $f$ is
smooth.  Reliant on a first-order oracle that produces a sampled gradient
$\nabla_x \f(x,\omega)$ and given an {$x_0$}, our proposed
variable sample-size accelerated proximal gradient scheme ({\bf VS-APM}) (also
see~\cite{ghadimi2016accelerated} and~\cite{jofre2017variance}) is stated as
follows where the true gradient is replaced by a sample average $\left(\nabla_x
f(x_k) + \bar{w}_{k,N_k}\right)$ with batch size $N_k$.  
\begin{align} \label{t-VS-APM}
    \begin{aligned}
	y_{k+1} & := {\bf P}_{\gamma_k g} \left(x_k - \gamma_k  \left(\nabla_x f(x_k) + \bar{w}_{k,N_k}\right)\right) \\
	x_{k+1} & := y_{k+1} + \beta_k (y_{k+1}-y_k), 
    \end{aligned}
\end{align} 
where $\bar{w}_{k,N_k} \triangleq \frac{\sum_{j=1}^{N_k} {\left(\nabla_x
\uss{\f}(x_k,\omega_{j,k})-\nabla_x f(x_k)\right)}}{N_k}$, ${\bf P}_{\eta g}(y) \triangleq \arg
\min_x\{{1\over 2} \|x-y\|^2 + \tfrac{1}{2\eta}g(x)\}$,  $\gamma_k$, and $\beta_k$ are suitably
defined steplengths. {Our approach} produces linearly convergent iterates in
strongly convex {regimes} {and achieves {an} iteration complexity of} $\mathcal{O}(1/K^2)$ in
merely convex \jb{and smooth regimes}, \usv{where $K$ is the total number of iterations}, \jb{matching the deterministic results} seen in the work by ~\cite{beck2009fast} and
~\cite{nesterov83}, .
\jb{The avenue represented by \eqref{t-VS-APM}} {has two key distinctions: (i) {\em Increasingly exact  gradients}
through increasing batch-sizes $N_k$ of sampled gradients, allowing for
progressive variance reduction; (ii) {\em Larger (non-diminishing) step-sizes}
in accordance with deterministic accelerated schemes. Collectively, (i) and
(ii) allow  for recovering fast (i.e. deterministic) convergence rates (in an
expected value sense) when $N_k$ grows sufficiently fast}. {Additionally,
such schemes have a more muted reliance on the condition number $\kappa =
L/\mu$ (in $\mu$-strongly convex and $L$-smooth regimes); \blue{specifically,  {in accelerated schemes}, such {dependence} reduces to $\sqrt \kappa$ in comparison with $\kappa$ in unaccelerated {counterparts} (cf.~\cite{nesterov14}). }}
\subsection{Prior \uvv{R}esearch}  ({\bf a}) {\em Stochastic {gradient schemes}.}  In nonsmooth
convex stochastic optimization problems, \cite{nemirovski_robust_2009}
derived an optimal rate of $\mathcal{O}(1/\sqrt{K})$ \usv{in terms of expected sub-optimality} via  an optimal constant
steplength (also see \cite{shamir2013stochastic}) whereas in strongly convex regimes, they derived a rate of
$\mathcal{O}(1/K)$ \usv{in a mean-squared sense}.  Structured nonsmooth problems (or {composite problems}) \jb{as defined by \eqref{main problem}})
have been examined extensively (cf.~\cite{lan2012optimal},\cite{ghadimi2012optimal})
and rates of $\mathcal{O}(\jb{L}/K^2 + 1/\sqrt{K})$  and $\mathcal{O}(\jb{L}/K +
1/\sqrt{K})$  were developed by~\cite{dang2015stochastic} 
{via} a mirror-descent framework for strongly convex and convex problems \jb{with $L$-smooth objectives}, respectively. 
 {In related work,~\cite{devolder14} derive oracle complexities with a deterministic oracle
of fixed inexactness, which was extended to a stochastic oracle by
\cite{dvurechensky2016stochastic}. Randomized smoothing techniques have also been employed by~\cite{yousefian2012stochastic} together with recursive steplengths (see~\cite{newton18recent} for a review) ({\bf b}) {\em Variance reduction.} In strongly
convex regimes (without acceleration), a linear rate of convergence in expected
error was first shown for variance-reduced gradient methods by ~\cite{shanbhag15budget} and revisited by \cite{jofre2017variance}, whereas similar rates were provided for extragradient methods by~\cite{jalilzadeh16egvssa}; the accelerated counterpart ({\bf VS-APM}) mutes the {dependence on $\kappa$, improving the bound} to $\mathcal{O}(\sqrt{L/\mu}\log(1/\epsilon))$.  In  smooth 
regimes, an accelerated scheme was first presented by
~\cite{ghadimi2016accelerated} where every {iteration} requires two prox
evaluations, admitting the optimal \jb{iteration complexity} and oracle complexity of
$\mathcal{O}(1/\jb{\sqrt{\epsilon}})$ and $\mathcal{O}(1/\epsilon^2)$,
respectively.~\cite{jofre2017variance} extended this scheme to allow for
state-dependent noise. 
An extragradient-based variable sample-size framework was suggested
by~\cite{jalilzadeh16egvssa} with a  rate of $\mathcal{O}(1/K)$. ({\bf c}) {\em
Smoothing techniques for nonsmooth problems.} For a subclass of deterministic nonsmooth problems, \cite{nesterov2005smooth} proved that an \redd{$\epsilon$-optimal solution} is
computable in $\mathcal{O}(1/\epsilon)$ gradient steps by applying an
accelerated method to a smoothed problem (primal smoothing with fixed smoothing
parameter). Subsequently, \cite{nesterov2005excessive}
considered primal-dual smoothing \jb{in deterministic regimes} (extended to composite problems by \cite{tran2018smooth})   with a diminishing smoothing parameter,
leading to rates of $\mathcal O(1/K^2)$ and $\mathcal O(1/K)$ for
strongly convex and convex \jb{deterministic} problems, respectively (also see \cite{boct2013double}, \cite{devolder2012double}). Adaptive
smoothing, considered by \cite{tran2017adaptive}, was shown to have an  iteration complexity of $\mathcal O(1/\epsilon)$ while Ouyang and Gray~\cite{ouyang2012stochastic} showed that smoothing-based minimization of $f+g$ where  $f(x) \triangleq \mathbb E[\f(x,\omega)]$ and $g(x) \triangleq \mathbb
E[\jb{\tilde g}(x,\omega)]$ leads to rates $\mathcal O(1/K)$ and  $\mathcal O(1/\sqrt K)$ when $\jb{\tilde g(\cdot, \omega)}$ is nonsmooth \jb{for a.e. $\omega$} whereas $\jb{\tilde f(\cdot,\omega)}$ is either strongly convex or merely convex  \jb{for a.e. $\omega$} (extended by \cite{zhong2014accelerated})\footnote{We would like to thank P.
Dvurechensky for alerting us to 
~\cite{tran2018smooth} and ~\cite{van2017smoothing}. }. 

\subsection{Gaps and \uvv{C}ontributions.} Unfortunately when  $\uss{\f(\cdot,\omega)}$ is a
nonsmooth strongly convex/convex function, \jb{stochastic subgradient schemes, subsequently defined in} ({\bf SSG}), while a
de-facto standard,  generally display poor empirical behavior, {since they
utilize  diminishing steplengths and} noisy gradients. We develop two distinct
avenues for combining {smoothing with acceleration and variance-reduction}
in strongly convex and convex regimes that ameliorate these concerns {while
achieving optimal rates}.   

{\bf(I)  ({\bf mVS-APM}) for strongly convex nonsmooth $f$.}  In Section 2, our
smoothing framework is reliant on a variable sample-size accelerated proximal
method {\bf (VS-APM)} which \usv{requires smoothness of}  $f$ while displaying linear
convergence and optimal oracle complexity. In two distinct settings, we propose applying {\bf (VS-APM)}
(or an unaccelerated variant) on the Moreau envelope of  {$F$}, denoted by
${F}_{\eta}$, where $F_{\eta}$ is $\tfrac{1}{\eta}$-smooth and
retains the minimizers of $F$. 
{{\bf(a) {Compact domains}.} {Under the assumption that \afo{ the domain of $g$ is bounded and $\mathbb{E}[\|\uss{S}(x,\omega)\|^2] \leq M^2$ for all $x \in \mathbb{R}^n$ where  \jb{$S(x,\omega)$ is a measurable selection from $\partial \f(x,\omega)$, i.e.} $\uss{S}(x,\omega) \in \partial \uss{\f}(x,\omega)$}, we show that ({\bf mVS-APM}) produces a linearly convergent sequence with \uss{an} iteration complexity \usv{of $\mathcal{O}(\log(1/\epsilon))$} in inexact gradient steps $\nabla_x
{F_{\eta}(x_k)}$, where increasingly exact gradients $\nabla_{x} F_{\eta}(x)$ are obtained by employing an ({\bf prox-SSG}) scheme. \blue{In
particular, our variance-reduced scheme endeavors to get increasingly exact
gradients {by progressively reducing the bias in the gradients} (since we {utilize  an}
increasing number of SSG steps); such a benefit does not appear in a naive implementation of SSG. Moreover,} the overall complexity in subgradient \jb{evaluations} (and consequently \jb{sample or} oracle
complexity) is $\mathcal{O}(1/\epsilon)$, matching the optimal complexity in subgradient steps achieved by ({\bf SSG})
schemes. {\bf (b) Unbounded domains.} {When domains are possibly unbounded, \jb{assuming that} $\mathbb{E}[\|S(x,\omega)\|^2] \leq
    \bar{M}^2 \|x\|^2 + M^2$, where $S(x,\omega)\in \partial \uss{\F}(x,\omega)$, the proposed (unaccelerated)  variable sample-size
proximal method ({\bf mVS-PM}) achieves an iteration complexity of
$\mathcal{O}(\log(1/\epsilon))$ (in \usv{gradient steps with $\nabla_x F_{\eta}$}) and
overall complexity in subgradient steps of $\mathcal O(1/\epsilon)$. }

\noindent {\bf (II) ({\bf sVS-APM}) for convex nonsmooth $f$.} In this setting, in Section 3, we develop an
iterative smoothing-based extension of ({\bf VS-APM}), denoted by ({\bf sVS-APM}).
By reducing the smoothing and steplength parameters at a suitable rate,  $\mathbb{E}[F(y_K)-F(x^*)] \leq
\mathcal{O}(1/K)$. Notably ({\bf sVS-APM}) produces asymptotically accurate
solutions (unlike the scheme by~\cite{nesterov2005smooth} \jb{which produces approximate solutions via a fixed smoothing parameter}) and is characterized
by the optimal oracle complexity of $\mathcal{O}(1/\epsilon^2)$. \usv{When
$f$ is convex and smooth, we} may
specialize these results to obtain an optimal rate of $\mathcal{O}(1/K^2)$  and displays an optimal \jb{sample} complexity of
$\mathcal{O}(1/\epsilon^2)$. 
When $f$ is deterministic but nonsmooth, ({\bf
s-APM}) matches the rate by ~\cite{nesterov2005smooth} but produces
asymptotically exact solutions.  Additionally, we prove that for suitable (but distinct) choices of steplength and smoothing sequences, ({\bf sVS-APM}) and ({\bf VS-APM}) produce sequences that converge
a.s. to a solution of \eqref{main problem}, a convergence statement that was
unavailable thus far, matching deterministic results by
\cite{orabona2012prisma} and \cite{boct2015variable} which leverage Moreau
smoothing; we provide a result for $(\alpha,\beta)$-smoothable functions
(see \cite{beck17fom}). 

\textbf{Notation:} A vector $x$ is assumed to be
a column vector while 
$\|x\|$ denotes the Euclidean vector norm, i.e., $\|x\|=\sqrt{x^Tx}$.
${\bf P}_{\eta g}(x)$ denotes the prox with respect to $g$ with {prox parameter} $\tfrac{1}{2\eta}$ at $x$. We  abbreviate ``almost
surely'' \usv{by} {\em a.s.} {and} $\mathbb{E}[z]$ denotes the expectation of a random
variable~$z$. {We let $X^*$ denote the set of optimal solutions of the \eqref{main problem}.} 

\begin{table}[htb]
\centering
\tiny
\begin{tabular}{|c|c|c|c|} \hline
{\sc Smooth} &   \tabincell{l}{\sc Conv. Rate \\ \sc Iter. comp.}& \tabincell{l}{\sc Prox. eval. \\ \sc Oracle comp.} & {\sc Comments} \\ \hline
\tabincell{l}{{\bf VS-APM} (2.1)\\
$f$ is $L$-smooth}&  \tabincell{c}{$\mathcal O(\rho^k)$ \\ {$\mathcal O( \sqrt \kappa\log(1/\epsilon))$}} & \tabincell{c}{$\mathcal O({\sqrt{\kappa}}\log(1/\epsilon))$\\
$\mathcal O(\kappa /\epsilon)$}& \tabincell{l}{Optimal rate and complexity} 
\\ \hline \hline
\usv{\sc Nonsmooth} &  \tabincell{l}{\sc Conv. Rate \\ Iter. comp.}& \tabincell{l}{\sc Oracle comp.}& \usv{\sc Comments} \\ \hline
{\tabincell{c}{{\bf mVS-APM} (2.3)\\
\afo{${\rm dom}(g)$ is bounded;} \\
$\mathbb{E}[\|R(x,\omega)\|^2] \leq M^2 $  \\
$\forall R(x,\omega) \in \partial \uss{\f}(x,\omega)$}}
&{\tabincell{c}{$\mathcal O(\rho^k)$ \\ $\mathcal O( \log(1/\epsilon))$}} & $\mathcal O( 1/\epsilon)$  & {\tabincell{l}{Minimize Moreau env. $F_{\eta}(x)$ via {\bf (VS-APM)} \\
Non-diminishing outer steps; \\ \afo{Approx. $\nabla_x F_{\eta}$ by {\bf (prox-SSG)} with increasing exactness;}}}  \\ \hline 
{\tabincell{c}{ {\bf mVS-PM} (2.4)\\
$\mathbb{E}[\|S(x,\omega)\|^2] \leq \bar{M}^2\|x\|^2 + M^2 $  \\
$\forall S(x,\omega) \in \partial \uss{\f}(x,\omega)$}}
&{\tabincell{c}{$\mathcal O(\rho^k)$  \\ $\mathcal O(\log( 1/\epsilon))$}} &{$\mathcal O( 1/\epsilon)$}& {\tabincell{l}{Minimize Moreau env. $F_{\eta}(x)$ via {\bf (VS-PM)}  \\
Non-dminishing outer steps;\\ 
Approx. $\nabla_x F_{\eta}(x)$ by {\bf (SSG)} with increasing exactness;
}}  \\ 
\hline
\end{tabular}

\caption{{Comparison of schemes in nonsmooth (NS) and strongly convex regimes, $\kappa=L/\mu$ and $\tilde \kappa={\mu\eta+1\over \mu\eta}$}}
\label{sc_table_nonsmooth}
\vspace{-0.4in}
\end{table}

\section{Nonsmooth \uvv{S}trongly \uvv{C}onvex \uvv{P}roblems}\label{sec. strong} In this
section,   we develop rate and complexity analysis for nonsmooth strongly
convex optimization problems via techniques that combine smoothing,
acceleration, and variance reduction. In Section~\ref{sec:2.1}, we review a
linearly convergent variance-reduced accelerated proximal scheme ({\bf
VS-APM}) for smooth stochastic convex optimization; this scheme will
serve as our subproblem solver.  In Section~\ref{sec:2.2}, we present a
Moreau-smoothed variant of ({\bf VS-APM}), referred to as ({\bf
mVS-APM}), {which relies on minimizing the Moreau envelope $F_{\eta}(x)$ of the strongly convex nonsmooth function $F(x)$} by {\bf (VS-APM)}. \blu{In Section~\ref{sec:2.3}, we then derive rate and complexity guarantees 
for ({\bf mVS-APM}) 
, where $\nabla_x
F_{\eta}(x)$ is approximated with increasing accuracy by a stochastic subgradient
({\bf SSG}) scheme. Finally, in Section~\ref{sec:2.4}, we derive analogous statements when applying an unaccelerated variable
sample-size proximal method ({\bf mVS-PM}) under possibly non-compact domains and under a (weaker) state-dependent bound on the subgradient (See Table~\ref{sc_table_nonsmooth} for a summary of findings).}

\subsection{Background on ({\bf VS-APM})} \label{sec:2.1}  

Consider \eqref{main problem} where $f, g$, and the initial point \redd{$x_0$} satisfy the following {assumption.} 
\begin{assumption}\label{fistaass2}
(i) $\afo{f}$ is a $\mu$-strongly convex function {and $g$ is a closed, convex, and proper deterministic function.} 
(ii) There exist $C,D>0$ such that $\mathbb{E} [\Vert \redd{x_0}-x^{\ast }\Vert ^{2}]\leq C$ and $\mathbb{E}[\|{F}(\redd{x_0})-{F}(x^*)\|]\leq D$, where {$F(x) \triangleq f(x)+g(x)$ and} {$x^*$ solves \eqref{main problem}}.
\end{assumption}
{In a subset of regimes, we impose an $L$-smoothness assumption on $f$.}  
{\begin{assumption}\label{smooth-f}
 \magg{The function} $\afo{f}$ is 
continuously differentiable   with
Lipschitz continuous gradient with constant $L$ i.e.  $ \|\nabla_x f(x)
-\nabla_x f(y) \| \leq L\|x-y\|$ for all $x,y \ \in \ \mathbb{R}^n.$
\end{assumption}}
 We utilize {a} variable sample-size
accelerated proximal scheme ({\bf VS-APM}), as defined in
Algorithm~\ref{nonsmooth sc scheme}, which can process such problems and
differs from a standard accelerated proximal method in that we employ an
inexact gradient $\nabla_{x} f(x_k) + \bar{w}_{k,N_k}$ where the bound on the second moment of $\bar{w}_{k,N_k}\triangleq \nabla_x f(x_k) - \frac{\sum_{k=0}^{N_k}\nabla_x f(x_k,\omega_k)}{N_k}$ is {diminishing} with $k$, a consequence of using variance reduction.
\begin{algorithm}[htbp]
\caption{\bf {Variable} sample-size accelerated proximal method (VS-APM)}
\em
\label{nonsmooth sc scheme}
(0) Given $\redd{x_0}$, \redd{$y_0=x_0$}, {$\kappa$}, and  positive
sequences $\{\gamma_k, N_k\}${\em ;} Set $\lambda_1\in \left(1,{\sqrt \kappa}\ \right]$;  $k := 1$ {\em ;}  \\
(1) $y_{k+1}:= {{\bf P}_{\gamma_k g}\left( x_k - \gamma_k \left(\nabla_x f(x_k) + \bar{w}_{k,N_k}\right)\right)}$ {\em ;} \\
(2) $\lambda_{k+1}:={{1\over 2} \left(1- \frac{\lambda_k^2}{\kappa}+\sqrt{\left(1-\frac{\lambda_k^2}{\kappa}\right)^2+4\lambda_k^2}\right)}$ {\em ;} \\
(3) $x_{k+1}:=y_{k+1}+\left(\frac{(\lambda_{k}-1)\left(1-\frac{1}{4\kappa}\lambda_{k+1}\right)}{\left(1-\frac{{1}}{4\kappa}\right)\lambda_{k+1}}\right)\left(y_{k+1}-y_{k}\right)$ {\em ;} \\
(4) If $k > K$, then stop{\em ;} else $k := k+1${\em ;} return
to (1).
\end{algorithm}

We outline the assumptions on the first and second moments of $\bar{w}_k$. 
\begin{assumption}\label{ass_error2}
(i) {\bf (Conditional boundedness of second moments)} There exists $\nu>0$ such that {$\mathbb{E}[\|\bar w_k\|^2\mid \mathcal{H}_k] \leq {\nu^2\over N_k}$} holds a.s.  for all $k$ and  $\mathcal{H}_k
	\triangleq \sigma\{x_0, x_1, \hdots, x_{k-1}\}$.  (ii) {\bf (Conditional unbiasedness of first moments)} $\mathbb{E}[w_k \mid \mathcal{H}_k] = 0$ holds a.s., where $w_k \triangleq \nabla_x f(x_k,\omega_k) - \nabla_x f(x_k)$. 
	\end{assumption}
{\bf (VS-APM)} can be shown to achieve linear convergence
akin to that by~\cite{nesterov14} by combining inexact gradients where the
inexactness is driven to zero by increasing the sample-size in estimating the
gradients. This avenue also allows for achieving the optimal oracle complexity
to obtain an \redd{$\epsilon$-accurate} solution. These differences lead to a slightly modified
set of update rules in contrast with that developed by~\cite{nesterov14} and
requires that $\gamma_k = 1/2L$ rather than $1/L$.  This scheme serves as a
subproblem solver in subsequent sections and we now state  a lemma and the associated 
complexity statement of ({\bf VS-APM}). The proof is similar to that by
\cite{nesterov14} and is in the Appendix.  Importantly, this scheme
allows for a possibly {\bf biased} estimate of the gradient.
\begin{lemma} \label{prop-err-bd}
Suppose Assumptions \ref{fistaass2}, \ref{smooth-f} {and \ref{ass_error2}(i)} hold.  Consider the iterates
	generated by {\bf (VS-APM)}, where $\gamma_k = \tfrac{1}{2L}$  for all $k \geq
	0$, $\kappa=\tfrac{L}{\mu}$, and $\bar \alpha=\tfrac{1}{{2}\sqrt \kappa}$
Then
	the following holds for all $K$.
 \begin{align}\label{bound_ac-VSSA_strong}
&  \mathbb{E}[F(y_{K})-F^*]\leq \left(D+\tfrac{\mu}{2}C^2\right)\left(1-\bar \alpha\right)^{K-1} + \sum_{i=0}^{{K}-1}{{\tfrac{\left(1-\bar \alpha\right)^{i}\left({2\over L}+{1\over \mu}\right)\nu^2}{N_{k-i}}}} 
  + \sum_{i=0}^{{K}-2}\tfrac{\left(1-\bar \alpha\right)^{i+1}\left({2\over L}+{1\over \mu}\right)\nu^2}{N_{k-i-1}}.
\end{align}
\end{lemma}
The following theorem characterizes the iteration and oracle complexity of {\bf (VS-APM)}.
\begin{theorem}[{\bf Rate and oracle  complexity of {\bf (VS-APM)} {under biased oracles}}]\label{th-rate-itercomp-sc}
Suppose Assumptions~\ref{fistaass2},~\ref{smooth-f}, and \ref{ass_error2}(i)   hold. Consider
the iterates generated by {\bf (VS-APM)}, where $\gamma_k \triangleq
\tfrac{1}{2L}$, $N_k \triangleq \lfloor \rho^{-k}\rfloor$, $\theta \triangleq
\left(1-\tfrac{1}{{2}\sqrt{\kappa}}\right)$,  $\rho \triangleq
\left(1-{1\over {2}a\sqrt{\kappa}}\right)$  for all $k \geq 0$ and $a > 2$.
 \begin{align}\label{bd-K} \hspace{-0.4in}
\mbox{(i) For all $K$, we have that }
\mathbb{E}[{F}(y_K)-{F}^*]
 \leq \tilde C \rho^{K-1}
\mbox{ where } \tilde C   \triangleq \left(D+\tfrac{\mu}{2}C^2\right) & +\tfrac{4 \nu^2}{\mu}+\tfrac{2\nu^2\sqrt\kappa}{\mu}.  \end{align}
{In addition,{\bf(VS-APM)} needs $\mathcal{O}(\sqrt{\kappa} \log (\tfrac{1}{\epsilon}))$ steps to obtain an \redd{$\epsilon$-accurate} solution, i.e. $\mathbb{E}[{F}(y_{{K+1}})-{F}^*]\leq \epsilon$.} \\ 
(ii) To compute an \redd{$\epsilon$-accurate} solution, {$\sum_{k=1}^K N_k \leq  {\left(\left(D+ \frac{\mu C^2}{2} \right) + \frac{4 \nu^2 }{ \mu}+{2\nu^2\sqrt\kappa\over \mu}\right)} \mathcal O\left({\sqrt \kappa\over \epsilon}\right).$}  
\end{theorem}
We know of no other result for variance-reduced accelerated proximal schemes in strongly convex (or even convex) smooth regimes that allows for biased oracles. For instance, {\cite{scholschmidt2011convergence}  impose unbiasedness in strongly convex regimes.}  
Next, we show that by adding the unbiasedness requirement, i.e. $\mathbb{E}[w_k \mid {\mathcal H_k}] = 0$ a.s. for all $k$, improves the constants in these bounds.

\begin{corollary}[{\bf Rate and oracle  complexity of {\bf (VS-APM)} under unbiased {oracles}}]
Suppose Assumptions \ref{fistaass2}, \ref{smooth-f}, and \ref{ass_error2}(i,ii)  hold. Consider the iterates
generated by {\bf (VS-APM)}, {where $\gamma_k \triangleq \tfrac{1}{2L}$, $N_k
\triangleq \lfloor \rho^{-k}\rfloor$, $\theta \triangleq
\left(1-\tfrac{1}{{2}\sqrt{\kappa}}\right)$,  $\rho \triangleq
\left(1-{1\over {2}a\sqrt{\kappa}}\right)$  for all $k \geq 0$ and $a > 2$.}
 \begin{align}\label{bd-K} \hspace{-0.4in}
\mbox{(i) For all $K$, we have that }
\mathbb{E}[{F}(y_K)-{F}^*]
 \leq \tilde C \rho^{K-1}
\mbox{ where } \tilde C   \triangleq \left(D+\tfrac{\mu}{2}C^2\right) +\tfrac{4 \nu^2}{\mu}.  \end{align}
{In addition, {\bf (VS-APM)} needs $\mathcal{O}(\sqrt{\kappa} \log (1/\epsilon))$ steps to obtain an \redd{$\epsilon$-accurate} solution.}\\ 
(ii) To compute an \redd{$\epsilon$-accurate} solution, $\sum_{k=1}^K N_k \leq  \left(\left(D+ \tfrac{\mu C^2}{2} \right) + \tfrac{4 \nu^2 }{ \mu}\right) \mathcal O\left({\sqrt \kappa\over \epsilon}\right).$  
\end{corollary} 

\blue{The application of ({\bf VS-APM}) is afflicted by the need for the $L$-smoothness of $f$ as well as the availability of $L$, the Lipschitz constant. Naturally, in many settings, the problem may not be smooth and even if $L$-smoothness holds, an estimate of $L$ may be unavailable. Consequently to broaden the reach of the scheme, an approach that obviates the need for $L$ or the imposition of the smoothness assumption is necessitated. 
This prompts the subsequent smoothed scheme ({\bf mVS-APM}). \blu{This scheme can
    always be implemented if  the strong convexity modulus (denoted by $\mu$) is known but the function is either nonsmooth or smooth with an unknown Lipschitz constant $L$. It is worth noting that estimating $\mu$ is challenging and if $\mu$ is indeed unknown, then in Section \ref{sec. convex}, we introduce an iteratively smoothed VS-APM (sVS-APM) method which {necessitates neither} the knowledge of the Lipschitz constant $L$, {nor} the smoothness of $f$,  {nor} the strong convexity modulus $\mu$. 
}}    
\subsection{A Moreau-smoothed \uvv{I}nexact \uvv{A}ccelerated \uvv{F}ramework ({\bf mVS-APM})}\label{sec:2.2}
When $\afo{\uss{\f}(\cdot,\omega)}$ is a nonsmooth {strongly} convex function {for
almost every $\omega$}, then the standard approach lies in utilizing stochastic
subgradient schemes {\bf (SSG)} where convergence relies on choosing square-summable
but non-summable steplength sequences. The choice of the parameters in such
sequences can have debilitating impact on performance in some settings
(cf.~\cite{shapiro09lectures}).  Specifically, {while choosing
$\gamma_k$ as $\tfrac{1}{\mu k}$ minimizes the mean-squared error but
over-estimating $\mu$ can have catastophic impact as seen in~\cite[Sec 5.9,
Ex.~5.36]{shapiro09lectures}. More generally, such choices
are often characterized by poor asymptotic behavior, a consequence that arises
in part from the diminishing nature of steplength sequences and the noisy subgradients.} We consider a
{\bf distinct avenue} {reliant} on {minimizing  the
Moreau envelope of a closed, convex, and proper function $F$ (cf.~\cite{moreau1965proximite}), denoted by $F_{\eta}(x)$ and defined next.} 
\begin{align}\label{prox moreau}
{F}_\eta(x) \triangleq \min_{u} \ \left\{{F}(u)+{1\over 2\eta}\|u-x\|^2 \right\}.
\end{align}
Notably, this smoothing {\bf retains} the {minimizer} {of $F(x)$} when $F$ is strongly convex. 
\begin{lemma}\cite[Lemma~2.19]{planiden2016strongly}\label{feta}
Consider a convex, closed, and proper function ${F}$ and its Moreau envelope
${F}_{\eta}(x)$. Then the following hold: (i) $x^*$ is a minimizer of ${F}$ over
$\mathbb{R}^n$ if and only if $x^*$ is a minimizer of ${F}_{\eta}(x)$; (ii) $F$
is $\mu$-strongly convex on $\mathbb{R}^n$ if and only if ${F}_{\eta}$ is
$\bar{\mu}$-strongly convex on $\mathbb{R}^n$ where $\bar{\mu} \triangleq \tfrac{\mu}{\eta\mu+1}$. 
\end{lemma}
{ Consequently, we minimize the $\bar{\mu}$-strongly convex
and ${1\over \eta}$-smooth function ${F}_{\eta}$, which is not
necessarily an easy task since computing $\nabla_x {F}_{\eta}(x)$
necessitates solving nonsmooth stochastic optimization problems.} We adopt
an inexact accelerated proximal scheme for minimizing
${F}_{\eta}$. But in contrast with ({\bf SSG}) schemes {applied to minimizing $F$}, we 
control the {smoothness of the outer problem by choosing $\eta$ and utilize
{\bf (i) larger non-diminishing steplengths}, {\bf (ii) acceleration}, and {\bf (iii)  increasingly exact
gradients}}, all of which are distinct from ({\bf SSG}), as shown next.  
\begin{align}\notag
\overbrace{\left[\begin{aligned} 
x_{k+1} &:= x_k - \gamma_k u_k \\
u_k & \in \afo{\partial {\tilde F}(x_k,\omega_k).}
\end{aligned} \qquad (\mbox{\bf SSG}) \right]}^{\scriptsize \gamma_k \to 0, \quad u_k \mbox{ is noisy subgradient.} } \qquad 
\overbrace{\left[\begin{aligned} 
y_{k+1} & := x_k - \gamma_k(\nabla_x {{F}_{\eta}} (x_k) + {\bar w}_{k,N_k}),  \\
x_{k+1} & := y_{k+1} + \beta_k (y_{k+1}-y_k).
\end{aligned} \, \quad  (\mbox{\bf {mVS-APM}})\right]}^{\scriptsize \mbox{Non-diminishing $\gamma_k$ + increasingly exact gradients + Acceleration}}
\end{align}
Importantly, $\nabla_x {F}_{\eta}(x_k) + {\bar w}_{k,N_k}$ represents an {\em approximation} of the gradient of the Moreau envelope. The true gradient of the Moreau envelope ${F}_{\eta}(x)$ is defined as $\nabla_x {F}_{\eta}(x) = \tfrac{1}{\eta}(x-\mbox{prox}_{\eta {F}}(x))$, where 
\begin{align}\label{def-prox} \mbox{prox}_{\eta {F}}(x) \triangleq \mbox{arg} \min_{u} \left\{ {F}(u) + {1\over 2\eta} \|x-u\|^2\right\}. \end{align}  
But $\mbox{prox}_{\eta {F}}(x)$ cannot be computed in finite time since ${F}$
is a nonsmooth \uss{expectation-valued} convex function. Instead, via stochastic
approximation, we compute an approximate solution of $\mbox{prox}_{\eta
{F}}(x)$, denoted by $\widehat{\mbox{prox}}_{\eta {F}}(x)$, implying the inexact
gradient of ${F}_{\eta}(x)$ is given by $\tfrac{1}{\eta}(x-\widehat{\mbox{prox}}_{\eta {F}}(x))$. In Algorithm~\ref{nonsmooth sc scheme}, the inexact gradient $\nabla_x {F}_{\eta}(x_k) + {\bar w}_{k,N_k}$ is defined as 
\begin{align}
\nabla_x {F}_{\eta}(x_k) + {\bar w}_{k,N_k}  =  \tfrac{1}{\eta}(x_k - \mbox{prox}_{\eta {F}}(x_k)) + \overbrace{\tfrac{1}{\eta}(\mbox{prox}_{\eta {F}}(x_k)-\widehat{\mbox{prox}}_{\eta {F}}(x_k))}^{\triangleq {\bar w}_{k,N_k}}.
\end{align}

{We now proceed to develop {\bf (mVS-APM)} for compact domains in Section~\ref{sec:2.3} and then weaken compactness requirements in Section~\ref{sec:2.4} for an unaccelerated variant.}
\subsection{{Linear \uvv{C}onvergence of ({\bf mVS-APM}): \uvv{C}ompact \uvv{D}omains}} \label{sec:2.3}
When $F(x) = \mathbb{E}[\uss{\f}(x,\omega)]+g(x)$, 
$\uss{\mbox{prox}_{\eta F}}(x)$, defined as \eqref{def-prox},  is
generally unavailable in closed-form and requires solving a strongly convex
nonsmooth stochastic optimization problem exactly.   Instead, one may
solve \eqref{prox moreau} {\bf inexactly} using {({\bf prox-SSG}), a slightly extended variant of  ({\bf SSG}) scheme~\cite{shapiro09lectures}. \blue{In particular, we propose ({\bf mVS-APM}) with the following update rules for $k\geq 1$,
\begin{subequations}\label{alg:mVS-SSG}
\begin{align}
y_{k+1} & := x_k - \frac{\gamma_k}{\eta}(x_k-\widehat{\mbox{prox}}_{\eta {F}}(x_k)),  \\
x_{k+1} & := y_{k+1} + \beta_k (y_{k+1}-y_k), 
\end{align}
\end{subequations}
\blu{where $\widehat{\mbox{prox}}_{\eta {F}}(x_k)$ is obtained by taking finite number of steps of ({\bf prox-SSG}) with  a sample size of one at each step} and having the following update rule for $j=0,\hdots,N_k-1$,
\begin{align}
\tag{{\bf prox-SSG}}
z_{k,j+1} := {\bf P}_{\eta/j, g}(z_{k,j} - \tfrac{\eta}{j} u_j),\quad u_j  \in \partial {\uss{\f}}(z_{k,j},\omega_j).
\end{align}}}%
\afo{Next, we state our assumptions and present the main result of this section. The constant in the rate and complexity bounds is dependent on $\tilde \kappa$; \uss{unlike, the condition number $\kappa$ in smooth regimes, $\tilde \kappa$ is user-specified and can be relatively small. For instance,  $\tilde \kappa = 2$ when $\eta = 1/\mu$}.} \jb{We employ a measurable selection from $\partial \f(x,\omega)$ as a stochastic subgradient in ({\bf SSG}) and impose the following assumption.} 

\afo{ \begin{assumption}\label{assum_sgd_general_1}
    \jb{For any $x \in \mathbb{R}^n$, consider a measurable selection} $R(x,\omega) \in \partial \uss{\f}(x,\omega)$.
        {(Unbiasedness).  We have that $\mathbb{E}[R(x,\omega)] = 
R(x) \in \partial f(x).$}
{(Subgradient boundedness).} There exists $M>0$  {such that for any $x$}, $\mathbb E[{\|R(x,\omega)\|^2}]\leq  M^2$. (Compact domain). The function $g$ has a compact domain, i.e., there exists $\Delta>0$ such that $\|x\|\leq \Delta$ for any $x\in {\rm dom}(g)$.
\end{assumption}}
 

\begin{theorem}[{\bf Rate and oracle  complexity of ({\bf mVS-APM})}]\label{th-rate-itercomp-sc-sgd}
{Suppose Assumptions~\ref{fistaass2} and \ref{assum_sgd_general_1} hold. 
Consider the iterates generated by {\bf (VS-APM)} applied
on $F_{\eta}(x)$ defined as \eqref{prox moreau} where $\theta \triangleq \left(1-{1
\over {2}\sqrt{\tilde \kappa}}\right)$,  $\rho \triangleq \left(1-{1\over
{2}a\sqrt{\tilde \kappa}}\right)$,  $\tilde
\kappa={\mu\eta+1\over \mu\eta}$, $a > 2$, and $\gamma_k =
\eta/{2}$, $N_k = \lfloor \rho^{-k}\rfloor$  for all $k \geq 0$. Then the following hold for $Q
\triangleq \max\left\{{\eta^2 M^2},4\Delta^2\right\}$.}\\
\noindent(i) {\bf (Rate).}  For all $K\geq 1$, we have that 
 \begin{align}\label{bd-K-sgd}
\mathbb{E}[\|y_K-x^*\|^2]
 \leq \widehat C \rho^{K-1}
\mbox{ where } \widehat C  \triangleq {2D \eta \tilde{\kappa}+ C^2 + {8 \tilde{\kappa}^{5/2}Qa}}. \end{align}
\noindent (ii) {\bf (\afo{Outer iteration complexity}).}  { The  iteration complexity of {\bf (mVS-APM)} in gradient steps {$($of $\nabla_x f_{\eta}(x_k))$}  to obtain an \redd{$\epsilon$-accurate} solution is $\mathcal O(\sqrt{\tilde \kappa} \log(\widehat C/\epsilon))$.} \\
\noindent (iii) {\bf (\afo{Oracle  complexity})}. {To compute $y_{K}$ such that $\mathbb{E}[\|y_K-x^*\|^2]\leq \epsilon$, the complexity of SSG steps is bounded as follows: $\sum_{k = 1}^K N_k \leq {2a^2\sqrt{\tilde \kappa} \widehat C\over (a-1)\epsilon}= \mathcal O(1/\epsilon).$}
\end{theorem}
\begin{proof}   {\bf (i)}  Recall that {${F}_\eta$} is $\tfrac{\mu}{
\mu\eta+1}$-strongly convex with $\tfrac{1}{\eta}$-Lipschitz continuous
gradients. At iteration $k$ of Algorithm~\ref{nonsmooth sc scheme}, ({\bf
prox-SSG}) \redd{with single sampling} can be used to inexactly solve \afo{$\displaystyle \min_{u}\left\{\mathbb
E[\uss{\f}(u,\omega)]+g(u)+\tfrac{1}{2\eta}\|u-x_k\|^2 \right\}$}. In particular, let
$\{z_{k,j}\}_{j=1}^{N_k}$ be the sequence generated by {({\bf prox-SSG})}
starting from  $z_{k,0}=x_k$ and {let} $z^*_k$ {denote the unique} optimal
solution of the subproblem.  Therefore, at step (1) of Algorithm
\ref{nonsmooth sc scheme}, $\bar w_{k,N_k}=\tfrac{1}{\eta} (z^*_k- z_{k,N_k})$ {and} by
the convergence rate of ({\bf prox-SSG})\cite{shapiro09lectures}, $\mathbb E[\|\bar w_{k,N_k}\|^2]\leq
{\bar Q_k\over {\eta^2} N_k}$, where $\bar Q_k \triangleq \max\left\{{\eta^2 M^2},\|
z_{k,0}- z^*_k\|^2\right\}\leq Q$, since $\|
z_{k,0}- z^*_k\|^2\leq 4\Delta^2$.  The results in Lemma \ref{prop-err-bd} hold
when $F(x)$ is replaced by $F_\eta(x)$, by letting $L=\tfrac{1}{\eta}$, replacing $\mu$ by $\tfrac{\mu}{
\mu\eta+1}$, $\nu^2$ by $\tfrac{Q}{\eta^2}$, and setting $\bar \alpha=1/({2}\sqrt {\tilde \kappa})$, where $\tilde
\kappa={\mu\eta+1\over \eta\mu}$:
\begin{align}\label{bound_ac-VSSA_strong-nonsmooth}
  \mathbb{E}[F_\eta(y_{K})-{F}_\eta^*]\leq \left(D+\tfrac{\mu}{2(\mu\eta+1)}C^2\right)\Big(1-\bar \alpha\Big)^{K-1} + \sum_{i=0}^{K-1}\tfrac{\left(1-\bar \alpha\right)^{i}{\left(2\eta+\tfrac{1}{\mu}\right)Q}}{{\eta^2N_{\afo{K}-i}}} 
  + \sum_{i=0}^{K-2}\tfrac{\left(1-\bar \alpha\right)^{i+1}{\left(2\eta+\tfrac{1}{\mu}\right)Q}}{{\eta^2N_{\afo{K}-i-1}}}.
\end{align} 
From Lemma \ref{feta}, $x^*$ is minimizer of function {${F}$} if and only if $x^*$ is a minimizer of function {${F}_{\eta}$}. Since {${F}_\eta$} is ${\mu\over \mu\eta+1}$-strongly convex, ${\mu\over 2(\mu\eta+1)}\|y_K-x^*\|^2\leq {F}_{\eta}(y_K)-{F}_{\eta}(x^*)$, implying \eqref{bound_ac-VSSA_strong-nonsmooth} can be written as 
  \begin{align}\label{bound_ac-VSSA_strong-nonsmooth2}
 \tfrac{\mu\mathbb{E}[\|y_K-x^*\|^2]}{2(\mu\eta+1)} \leq \left(D+\tfrac{\mu}{2(\mu\eta+1)}C^2\right)\left(1-\bar \alpha\right)^{K-1} + \sum_{i=0}^{K-1}\tfrac{\left(1-\bar \alpha\right)^{i}{\left(2\eta+\tfrac{1}{\mu}\right)Q}}{{\eta^2N_{\afo{K}-i}}} 
  + \sum_{i=0}^{K-2}\tfrac{\left(1-\bar \alpha\right)^{i+1}\left(2\eta+\tfrac{1}{\mu}\right)Q}{\eta^2N_{\afo{K}-i-1}}.
\end{align} 
 From \eqref{bound_ac-VSSA_strong-nonsmooth}, by definition of $\theta$ {and recalling the increasing nature of $\{N_k\}$}, we may claim the following:
\begin{align}\label{bound N_k nonsmooth}
 \tfrac{\mu\mathbb{E}[\|y_K-x^*\|^2]}{2(\mu\eta+1)}& \leq (D+\tfrac{\mu}{2(\mu\eta+1)}C^2)\theta^{K-1}+ \sum_{j=0}^{K-1} \theta^{j} \tfrac{\left(2\eta+\tfrac{1}{\mu}\right)Q}{\eta^2N_{K-j-1}}+\sum_{j=0}^{K-1}\theta^{j+1}\tfrac{\left(2\eta+\tfrac{1}{\mu}\right)Q}{{\eta^2N_{K-j-1}}} \notag \\ 
&= (D+\tfrac{\mu}{2(\mu\eta+1)}C^2)\theta^{K-1}+ \sum_{j=0}^{K-1} \tfrac{\theta^j (1+\theta)\left(2\eta+\tfrac{1}{\mu}\right)Q}{
	{\eta^2N_{K-j-1}}} \notag \\
& \overset{\scriptsize (1+\theta) \leq 2}{\leq} (D+\tfrac{\mu}{2(\mu\eta+1)}C^2)\theta^{K-1}+ \sum_{j=0}^{K-1}  \tfrac{2\theta^{j}\left(2\eta+\tfrac{1}{\mu}\right)Q}{{{\eta^2N_{K-j-1}}}}.
 \end{align}
If $ N_{K-j-1}=\lfloor\rho^{-(K-j-1)}\rfloor$, by using Lemma \ref{fistabound for floor}, we have the following:
\begin{align}\label{b_ac2 nonsmooth}
& \quad  \sum_{i=0}^{K-1}{\tfrac{2\theta^{j} (2\eta+1/\mu)Q}{\eta^2\lfloor \rho^{-(K-j-1)}\rfloor}}
 \leq \sum_{i=0}^{K-1}\tfrac{\theta^{j} \left(2\eta+\tfrac{1}{\mu}\right)Q}{{\eta^2 \rho^{-(K-j-1)}}}  \leq \tfrac{\left(2\eta+\tfrac{1}{\mu}\right)Q \rho^{K-1}}{\eta^2} \sum_{i=0}^{K-1}{ \left(\tfrac{\theta}{\rho}\right)^{i} } 
 \leq \left(\tfrac{{(2\eta+\tfrac{1}{\mu})Q}\rho}{\eta^2( \rho-\theta)}\right){\rho^{K-1}}.
\end{align}
By substituting \eqref{b_ac2 nonsmooth} in \eqref{bound N_k nonsmooth} and {using $\tfrac{\rho}{\rho-\theta}  = {\tfrac{1-\tfrac{1}{2a \sqrt{\tilde \kappa}}}{\tfrac{1}{2 \sqrt{\tilde \kappa}}-{\tfrac{1}{2a \sqrt{\tilde \kappa}}}} = \tfrac{(2a \sqrt{\tilde \kappa}-1)}{a-1}}  \leq 2a\sqrt{\tilde\kappa}$}, \eqref{bound N_k nonsmooth} becomes 
\begin{align}\label{bound for K nonsmooth}
\mathbb{E}[\|y_K-x^*\|^2]
  & \leq\tfrac{2(\mu\eta+1)}{\mu} \left(D+\tfrac{\mu}{2(\mu\eta+1)}C^2\right)\theta^{K-1}  + \left( \tfrac{2(\mu\eta+1)}{\mu}\right)
 \tfrac{2}{\eta^2}\left(2\eta+\tfrac{1}{\mu}\right)Qa\sqrt{\tilde \kappa}{\rho^{K-1}}  \notag \\
& \leq 	{\left(\left(D \tfrac{2 (\eta\mu+1)}{
\mu}\right)+C^2  + \left(8\left(\tfrac{1+\eta \mu }{\eta\mu}\right)^2Qa\right)\sqrt{\tilde \kappa} \right)} \rho^{K-1} \notag \\
& = \widehat{C} \rho^{K-1}, \mbox{ where } \widehat{C} \triangleq  {2D \eta \tilde{\kappa}+ C^2 + {8 \tilde{\kappa}^{5/2}Qa}}.
\end{align}
\noindent {\bf (ii)} {We may derive the number of gradient steps ${K}$ {(of $\nabla_x f_{\mu}$)} to obtain an \redd{$\epsilon$-accurate} solution:} 
{\begin{align*}
	\tfrac{1}{\rho} & = \tfrac{1}{(1-\tfrac{1}{{2}a \sqrt{{\tilde \kappa}}})} 
			 = \tfrac{{2}a \sqrt{{\tilde \kappa}}}{({2}a\sqrt{{\tilde \kappa}}-1)} \implies \tfrac{ \log(\widehat C) - \log(\epsilon)} {\log(1/\rho)} \leq \tfrac{ \log(\widehat C) - \log(\epsilon)} {(1-\rho)} = {{(2a\sqrt{{\tilde \kappa}})} \log({\widehat C}/\epsilon) \leq K}.
\end{align*}}
\noindent {\bf (iii)} {To compute a vector $y_{K}$} satisfying $\mathbb{E}[\|y_K-x^*\|^2]\leq \epsilon$, we have $\widehat C{\rho}^{K}\leq \epsilon$ implying that  
{$K = \lceil \log_{(1/  {\rho})}(\widehat C/\epsilon)\rceil \leq 1+\log_{(1/  {\rho})}(\widehat C/\epsilon)$.}
To obtain the oracle complexity, we require $\sum_{k=1}^{K} N_k$ gradients. If $N_k=\lfloor \rho^{-k}\rfloor\leq \rho^{-k}$, we obtain the following since $(1-\rho) = (1 \slash ({2}a \sqrt{{ {\tilde \kappa}}}))$.
{\begin{align}\label{rho tag}
 \quad \sum_{k=1}^{K} \rho^{-k} 
 &\leq \tfrac{\left(\tfrac{1}{\rho}\right)^{2+K}}{\left(\tfrac{1}{{\rho}} -1\right)} \leq \tfrac{\left(\tfrac{1}{\rho}\right)^{3+\log_{1/\rho}(\widehat{C}/\epsilon)}}{\left(\tfrac{1}{{\rho}} -1\right)}  \leq \tfrac{\widehat C}{\rho^2(1-{\rho})\epsilon}
 = \tfrac{ 2a \sqrt{{\tilde \kappa}} \widehat C}{\rho^2\epsilon}. \end{align}}
Note that $\rho=1-{1\over {2}a\sqrt{ {\tilde \kappa}}}$, implying that
\begin{align*}
 \rho^2 & = 1-2/({2}a\sqrt {\tilde \kappa})+1/({4}a^2{\tilde \kappa})= \tfrac{{4}a^2{\tilde \kappa}-{4}a\sqrt {\tilde \kappa}+1}{{4}a^2{\tilde \kappa}}\geq \tfrac{ {4}a^2{\tilde \kappa}-{4}a{\tilde \kappa}}{{4}a^2{\tilde \kappa}}=\tfrac{(a^2-a)}{a^2}\\
  \implies & \tfrac{\sqrt{\tilde \kappa}}{\rho^2}\leq \tfrac{a^2 \sqrt {\tilde \kappa}}{(a^2-a)}=\tfrac{a}{a-1}\sqrt{\tilde \kappa}
 \implies  {\mbox{by \eqref{rho tag}},}\ \sum_{k=1}^{\log_{(1/{\rho})}\left(\widehat C/\epsilon\right)+1} \rho^{-k} \leq \tfrac{{2}a^2\sqrt{\tilde \kappa} \widehat C}{(a-1)\epsilon}. 
\end{align*}
 \end{proof}
%

\begin{remark}
In Theorem~\ref{th-rate-itercomp-sc-sgd}, choosing $\eta={1}/{\mu}$ leads to 
 $\mathbb{E}[\|y_K-x^*\|^2] \leq \left({4D\over \mu}+C^2+12\sqrt 2aQ\right)\rho^{K-1},$
 and {an} oracle complexity {of} $\mathcal O\left(\tfrac{\max\left\{{M^2}/{\mu^2},{\|\redd{\tilde x_0}-\tilde x^*\|^2}\right\}}{\epsilon}\right)$, matching the result by \cite{shapiro09lectures}.
 \end{remark}
\begin{figure}
\begin{center}
	\includegraphics[width=0.5\textwidth]{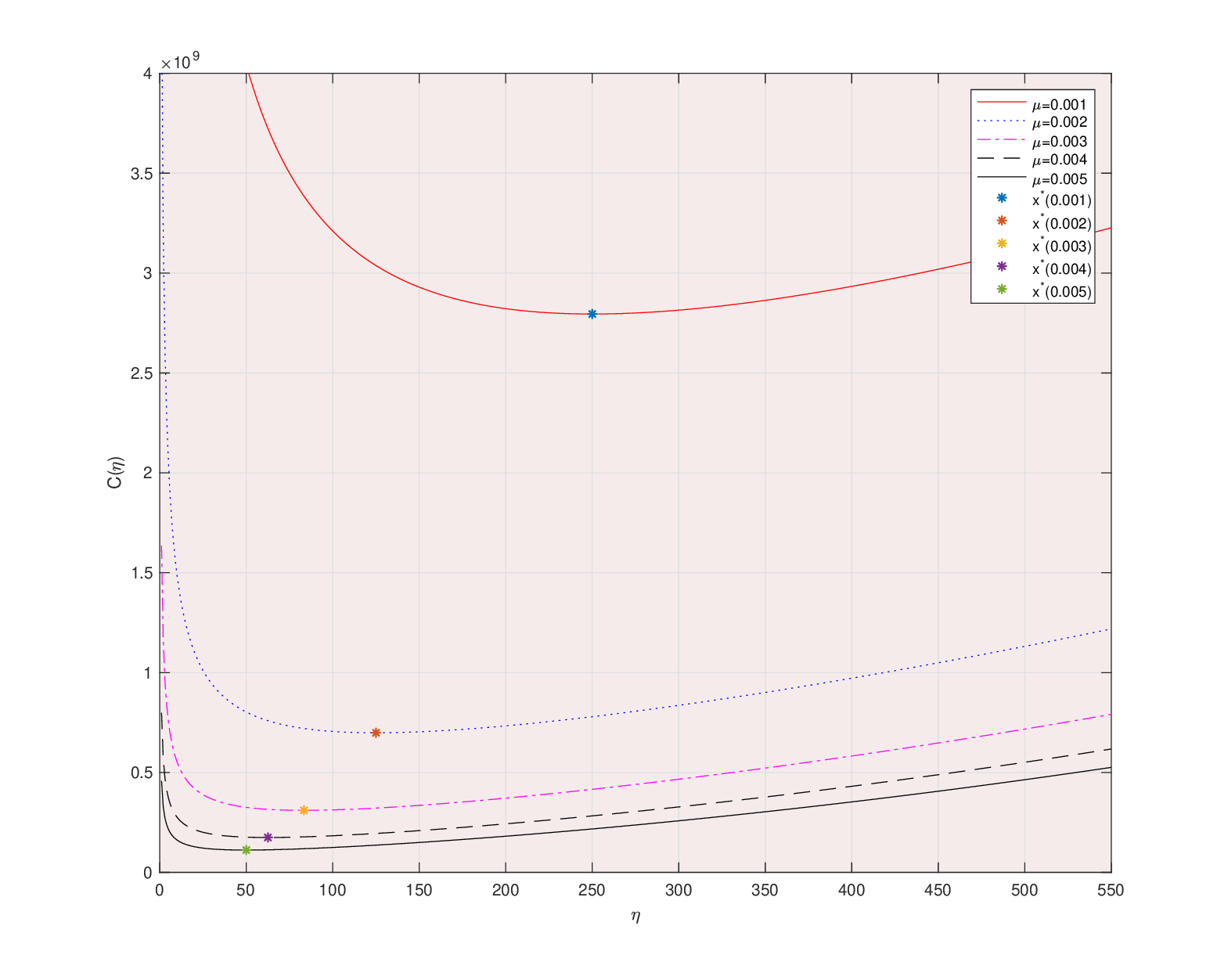}
\end{center}
	\caption{Schematic of $\widehat{C}(\eta)$ when $D = 10, M = 10, C = 100, a = 2.1, \Delta = 1$ for $\mu \in \{0.001, \cdots, 0.005\}$}
	\label{fig-eta}
	\end{figure}
 {Minimizing the convergence bound {in \eqref{bound for K nonsmooth}} in $\eta$ is possible via a less obvious coercivity and strict convexity claim for the nonsmooth function $\widehat{C}(\eta)$ (See Appendix for proof).}
{\begin{lemma}\label{char-eta}
Consider $\widehat{C}(\eta)$ defined as 
$\widehat C(\eta)  \triangleq  2D \eta \tilde{\kappa}(\eta) + C^2 + {8 \tilde{\kappa}(\eta)^{5/2}Q(\eta)a}$,  where $Q \triangleq \max\{\eta^2 M^2,4 \Delta^2\}$. 
Then the following hold. 

(i) $\widehat{C}(\eta)$ is a coercive function on $\{\eta \mid \eta \geq 0\}$. 

(ii) $\widehat{C}(\eta)$ is a strictly convex function on $\{\eta \mid \eta \geq 0\}$. 

(iii) The minimizer of $\widehat{C}(\eta)$ on $\{\eta \mid \eta \geq 0\}$ is unique. 
\end{lemma}}

{\begin{remark}\label{remark2}
{Lemma~\ref{char-eta} allows for claiming that $\widehat{C}(\eta)$ has a unique minimizer $\eta^*$; in fact, such a minimizer can be computed
by a standard semismooth Newton method~\cite{facchinei02finite}.
Fig.~\ref{fig-eta} provides a schematic of $\widehat{C}(\eta)$ for different
values of $\mu$ while $\eta^*$ is computed by  semismooth Newton method. We
note that when $\mu$ is larger, $\eta^*(\mu)$ tends to be smaller. In such
cases, obtaining an optimal $\eta^*$ is particularly useful.  However, when
$\mu \ll 1$, we observe that $\eta^*(\mu) \gg 1$; consequently, this leads to
rescaling of the step $\gamma_k$ to $\tfrac{\gamma_k}{\eta}$, resulting in
poorer behavior. Therefore, if $\mu \ll 1$, we employ $\eta = 1$ and this has
far better empirical behavior as seen in the numerics. }  
\end{remark}}

\subsection{{Linear \uvv{C}onvergence of ({\bf mVS-PM}): \uvv{N}on-compact \uvv{D}omains}} \label{sec:2.4}
In this subsection, we derive rate and complexity guarantees when {\bf (VS-PM)}, an unaccelerated variant of {\bf (VS-APM)}, is 
applied on a Moreau-smoothed problem under {possibly non-compact domains and under a (weaker)} state-dependent bound on 
the \afo{subgradient} (Assumption~\ref{assum_sgd_general}). \afo{When the subgradient of $g$ \uss{is characterized by} a state-dependent bound, the \uss{bound on the} cumulative error in the accelerated method \uss{builds up} due to a recursive relation, see \eqref{def-pk} in the Appendix. 
Hence, in this section,  we \uss{consider} a more general case \usv{in which Assumption~\ref{assum_sgd_general} imposes a state-dependent bound, weakening Assumption~\ref{assum_sgd_general_1}.}  \uss{By employing} an unaccelerated method, we \us{derive} a similar oracle complexity as in section \ref{sec:2.3}.} To obtain rate results, we {apply ({\bf VS-PM})} with the following update rule:
\begin{align}\tag{\bf VS-PM}x_{k+1}  := x_k - \gamma(\nabla_x {{F}_{\eta}} (x_k) + {\bar w}_{k,N_k}),\end{align}
where $\nabla_x {{F}_{\eta}} (x_k) + {\bar w}_{k,N_k}$ can be obtained by solving $\displaystyle \min_{u \in \mathbb{R}^n }\left[\mathbb E[\afo{\tilde F(u,\omega)}]+\tfrac{1}{2\eta}\|u-x_k\|^2 \right]$ inexactly taking $N_k$ (stochastic)
subgradient steps. {Consider the sequence of iterates $\{x_k\}$ generated by applying an inexact gradient scheme on the following strongly convex smooth optimization problem. 
 \begin{align*}
\min_{x \in \mathbb{R}^n} \ F_{\eta}(x),\ \mbox{where}\   F_{\eta}(x)\triangleq \min_{u \in \mathbb{R}^n} \ \left[\mathbb{E}[\uss{\f}(u,\uss{\omega})] + \uss{g(u)}+ \tfrac{1}{2\eta} \|x-u\|^2\right].  
\end{align*}
In effect, given an $x_0 \in \mathbb{R}^n$, the inexact gradient scheme generates a sequence $\{x_k\}$ such that 
\begin{align}\label{IG}
x_{k+1} := x_k - \gamma \left(\nabla_x F_{\eta}(x_k) + \bar{w}_k\right). \tag{IG}
\end{align}
Given an $x_k$, we denote the update with the exact gradient by $\bar{x}_{k+1}$, which is defined as follows.
\begin{align*}
\bar{x}_{k+1} := x_k - \gamma \nabla_x F_{\eta}(x_k). 
\end{align*}
Recall that $\nabla_x F_{\eta}(x_k)$ is defined as $\nabla_x F_{\eta}(x_k) = \tfrac{1}{\eta}(x_k - z^*_{k})$ where $z^*_{k}$ is the unique minimizer of the following problem, i.e. 
\begin{align}\label{prox-prob}
	z^*_{k} \ \triangleq \  \mbox{arg}\hspace{-0.02in}\min_{u \in \mathbb{R}^n} \   \left[\mathbb{E}[\uss{\F}(u,\uss{\omega})]+\tfrac{1}{2\eta} \|x_k-u\|^2\right]. 
\end{align}
In other words, $z^*_{k}$ is defined as 
\begin{align*}
	z^*_{k} \ \triangleq \ \mbox{prox}_{\eta F}(x_k) \mbox{ while } x^* = \mbox{prox}_{\eta F}(x^*). 
\end{align*}
{Since $\mbox{prox}_{\eta F}(x_k)$ is unavailable in closed form, we may compute increasingly  exact analogs;  given $z_{k,0} = x_k$, we construct the sequence $\{z_{k,j}\}_{j=1}^{N_k}$ based on ({\bf SSG}).} 
\begin{align} \tag{\bf SSG}\label{SA}
	z_{k,j+1} & = z_{k,j} - \sigma_j G(z_{k,j},\uss{\omega}_{k,j}), \quad j \geq 0, \mbox{ where } G(z_{k,j},\uss{\omega}_{k,j}) \in {\partial \uss{\F}(z_{k,j}, \uss{\omega}_{k,j}) + \tfrac{1}{\eta} {(z_{k,j}-x_k)}}. 
\end{align} 
Consequently, at major iteration $k$,  the inexact gradient of $F_{\eta}(x)$ is given by $\tfrac{1}{\eta}(x_k-z_{k,N_k})$ implying that $\bar{w}_k$ is defined as $\tfrac{1}{\eta}(z_{k}^*-z_{k,N_k}).$ Consequently, we have that
\begin{align*}x_{k+1} = x_k - \gamma (\tfrac{1}{\eta}(x_k-z_{k,N_k})) = (1-\tfrac{\gamma}{\eta})x_k + \tfrac{\gamma}{\eta} z_{k,N_k}. \end{align*}
We proceed to derive a bound on \uss{the conditional second moment of} $G(z_{k,j},\omega_{k,j}) \jb{= S(z_{k,j},\omega_{k,j})+ \tfrac{1}{\eta} {(z_{k,j}-x_k)}}$ where ${S(z_{k,j},\omega_{k,j})}\in
\partial \uss{\F}(z_{k,j}, \uss{\omega}_{k,j}) $, $M_1^2\triangleq 2 \bar{M}^2 + \tfrac{4}{\eta^2}$, $M_2^2\triangleq \tfrac{4}{\eta^2}$, and $M_3^2\triangleq2M^2$. \uss{This requires defining the history upto iteration $j$ at outer iteration $k$ by $\mathcal{F}_{k,j}$ as follows.
\begin{align}
    \mathcal{F}_{0} & = \{x_0\}, \mathcal{F}_{0,j} = \mathcal{F}_0 \cup \left\{S(z_{0,0},\omega_{0,0}), \cdots, S(z_{0,j-1}, \omega_{k,j-1})\right\}, \ \qquad \qquad j = 1, \cdots, N_0 \\
    \label{def-Fk}
    \mathcal{F}_k & = \mathcal{F}_{k-1, N_{k-1}} \cup \{x_k\}, \mathcal{F}_{k,j} = \afo{\mathcal{F}_{k}} \cup  \left\{  S(z_{k,0},\omega_{k,0}), \cdots, S(z_{k,j-1}, \omega_{k,j-1})\right\}, \ j = 1, \cdots, N_k, \ k \geq 1. 
\end{align}}
We now outline an
assumption on the bound on the stochastic subgradient  that scales with the size of $x$ allowing for non-compact domains.
\blue{ \begin{assumption}\label{assum_sgd_general} 
        \redd{Let} $\{x_k\}$ be a sequence generated by $(${\bf VS-PM}$)$ where $\nabla_{x} F_{\eta}(x_k)+\bar{w}_{k,N_k}$ is computed by taking $N_k$ steps of $(${\bf SSG}$)$ leading to a set of iterates $\{z_{k,1}, \cdots, z_{k,N_k}\}$. \redd{Let} $\mathcal{F}_{k,j}$ \usv{be} defined as \eqref{def-Fk} for $k \geq 1$ and $j = 1, \cdots, N_k$. {For any $z_{k,j}$, \redd{let} $S(z_{k,j},\omega_{k,j})$ denote a measurable selection $S(z_{k,j},\omega_{k,j})\in {\partial \tilde F(z_{k,j},\omega_{k,j})}$.} With these \uvv{constructs}, the following are assumed to hold.

\noindent (a) {$($Unbiasedness$)$. We have that $\mathbb{E}[S(z_{k,j},\omega_{k,j}) \mid \mathcal{F}_{k,j}] = S(z_{k,j}) \in \partial F(z_{k,j})$ {almost surely}.}

\noindent (b) {$($Subgradient boundedness$)$.} There exists $M,\bar M>0$  {such that for any $x$}, $\mathbb E[{\|S(z_{k,j},\omega_{k,j})\|^2}\mid \mathcal{F}_{k,j}]\leq \bar M^2\|z_{k,j}\|^2+M^2$ {almost surely}. 
\end{assumption}}

\uss{Consequently, we have that}
\begin{align}\notag \|G(z_{k,j},\uss{\omega}_{k,j})\|^2  \leq 2\|{S}(z_{k,j},\uss{\omega}_{k,j})\|^2 &+ \tfrac{2}{\eta^2}\|z_{k,j} - x_k\|^2 \leq 2\|{S}(z_{k,j},\uss{\omega}_{k,j})\|^2 + \tfrac{4}{\eta^2}\|z_{k,j}\|^2 + \tfrac{4}{\eta^2} \|x_k\|^2 \\
    \uss{\implies  \mathbb{E}[\|G(z_{k,j},\uss{\omega}_{k,j})\|^2 \mid \mathcal{F}_{k,j}] } &  \notag\overset{\mbox{\scriptsize Assump.}~\ref{assum_sgd_general}}{\leq} (2 \bar{M}^2 + \tfrac{4}{\eta^2}) \|z_{k,j}\|^2  + 2M^2 + \tfrac{4}{\eta^2} \|x_k\|^2 \\
                                                                                            &  \uss{\ =: \ } M_1^2 \|z_{k,j}\|^2+ M_2^2\|x_k\|^2 + M_3^2.\label{bd-sub-G}
\end{align}
{Based on Assumption~\ref{assum_sgd_general} and inspired by a proof technique from~\cite{chombolle10firstorder} amongst others, we derive a rate statement for ({\bf SSG}) (See Appendix for proof).}
{ \begin{proposition}\label{bd-SSG} Consider \eqref{prox-prob} where $F(\uss{\cdot},\uss{\omega})$ is a $\mu$-strongly convex function and ${S}(z,\omega)
\in \partial \uss{\F}(z,\omega)$ \uss{for any $z$}. {Suppose Assumption ~\ref{assum_sgd_general} holds and } $\hat a^2  \triangleq 4+4M_1^2+2M_2^2$ and $\hat b^2 \triangleq (4M_1^2+2M_2^2) [\|x^*\|^2]+M_3^2.$ {Given $x_k$}, consider a sequence generated by \eqref{SA} where {$\tilde \mu = \mu+\tfrac{1}{\eta}$}, $\bar J \triangleq \lceil \tfrac{2M_1^2}{{\tilde \mu}^2}-1\rceil$, and $$\sigma_j \triangleq \begin{cases}
\min \left\{\tfrac{1}{(j+1)\log(j+1)},\tfrac{\tilde \mu}{M_1^2}\right\}, & j<\bar J \\
\tfrac{1}{(j+1)\log(j+1)}. & j\geq \bar J
\end{cases}$$ 
Then  the following holds for $j \geq \bar J$. 
  \begin{align}\label{sub_g_2}
      \mathbb{E}[\|z_{k,j} - z_{k}^*\|^2 \uss{ \ \mid \mathcal{F}_{k}}]  \leq \tfrac{\hat a^2\|x_k-x^*\|^2+\hat b^2}{j}. 
\end{align}
\end{proposition}}

We now show the convergence of {\bf (mVS-PM}) when $\nabla_x F_{\eta}(x)$ is approximated via {\bf (SSG)} (See Appendix for proof).

{\begin{theorem}[{\bf ({\bf mVS-PM}) {under state-dependent bound on subgradients}}]\label{th-rate-itercomp-sc-sgd-general}
Suppose Assumptions~\ref{fistaass2} and~\ref{assum_sgd_general} hold. Consider the iterates
generated by {\bf (VS-PM)} {applied on $F_{\eta}(x)$}, where $\tilde{\kappa} \triangleq 1+\tfrac{1}{\eta \mu}$, $\gamma = \eta$, and $N_k \triangleq \lfloor N_0 \rho^{-k}\rfloor$  for all $k \geq
0$,  {$N_0> \max\{\tfrac{{2}\hat a^2}{(1-q/{2})},\bar J\}$}, $q \triangleq 1-\tfrac{1}{\tilde \kappa}$, {$p_0 \triangleq \tfrac{q}{2}+\tfrac{2\hat{a}^2}{N_0}$}, and
$\bar J \triangleq \lceil \tfrac{2M_1^2}{\bar \mu^2}-1\rceil$. Then the
following hold.\\
\noindent(i) {\bf (Rate).}  For all $k\geq 1$, we have that the following holds.
{\begin{align*}
\mathbb{E}[\|{x_k}-x^*\|^2]
 \leq \mathcal{C} \hat{p}^k \mbox{ where } \mathcal{C}   \triangleq \left( \mathbb{E}[\|x_{0}-x^*\|^2]+\tfrac{\hat b \widehat{D}}{N_0}\right),  \begin{cases}
	\rho \neq p_0, & \hat{p} = \max\{\rho,p_0\}, \widehat{D} \triangleq \tfrac{1}{{1-\tfrac{\min\{\rho,p_0\}}{\max\{\rho,p_0\}}}}\\ 
	\rho = p_0. & \hat{p} \in (p_0,1), \widehat{D} >\tfrac{1}{\ln(p_0/\hat p)^e}
\end{cases}
\end{align*}}
\noindent (ii) {\bf (Iteration complexity).}  The  iteration complexity of {$(${\bf mVS-PM}$)$} in gradient steps {$($of $\nabla_x {F}_{\eta}(x_k))$}  to obtain an \redd{$\epsilon$-accurate} solution is $\mathcal O({\tilde \kappa} \log(\mathcal C/\epsilon))$. 

\noindent (iii) {\bf (Oracle  complexity in {\bf (SSG)} steps)}. {To compute {$x_{K}$} such that $\mathbb{E}[\|{x_K}-x^*\|^2]\leq \epsilon$, the complexity in subgradient steps  is bounded as  $\sum_{k = 1}^K N_k \leq \mathcal O\left(\tilde \kappa \left(\tfrac{\mathcal C}{\epsilon}\right)^{\log_{1/\hat p}(1/\rho)}\right)$ for $\hat{p} \in [p_0,1)$, $\rho  \leq p_0$  and 
$\sum_{k = 1}^K N_k \leq \mathcal O\left(\tilde \kappa \left(\tfrac{\mathcal C}{\epsilon}\right)\right)$ for $\rho > p_0$.}
\end{theorem}
}

{\begin{remark}
We observe that when $\rho > p_0$, we achieve the optimal oracle
complexity in subgradient steps akin to the statement in the regime of bounded
subgradients. Notably, $\tilde{\kappa}$ can be controlled since $\eta$ is any
nonnegative scalar. For instance, if  $\eta=\tfrac{1}{\mu}$, 
$\tilde \kappa=2$. \end{remark}}

\section{Iteratively Smoothed VS-APM for Nonsmooth Convex Problems}\label{sec. convex}
Thus far, we have considered settings where $f$ is a strongly convex
function. However, there are many instances when the function $f$ is
neither smooth nor strongly convex. \usv{In fact, in strongly convex regimes, estimating the strong convexity parameter may often be challenging.}  In such settings,  if the function $f$ is
subdifferentiable, then subgradient methods provide an avenue for resolving
such problems in stochastic regimes but display a significantly poorer rate of convergence. \cite{nesterov2005smooth}
showed that for a subclass of problems, an accelerated gradient scheme may be
applied to a suitably {\em smoothed} problem where the smoothing leads to a
differentiable problem with Lipschitz continuous gradients (with known
Lipschitz constants). If the smoothing parameter is chosen suitably, the
convergence rate to an approximate solution can be improved to
$\mathcal{O}(1/K)$ from $\mathcal{O}(1/\sqrt{K})$ \usv{in terms of expected sub-optimality}. However, since the smoothing parameter is maintained
as fixed, Nesterov's approach can provide approximate solutions at best
but not asymptotically exact solutions. Subsequently,
~\cite{nesterov2005excessive} considered a primal-dual smoothing technique
where the smoothing parameter is reduced at every step while extensions and
generalizations have been considered more
recently by~\cite{tran2018smooth} and ~\cite{van2017smoothing}.  In this section, we
develop an {\em iteratively smoothed variable sample-size accelerated proximal
gradient} scheme that can contend with expectation-valued objectives and is asymptotically convergent. {This can be viewed as a variant of the primal smoothing
scheme introduced by~\cite{nesterov2005smooth} where 
the smoothing parameter} is reduced after every step; this scheme is shown to
admit a rate of $\mathcal{O}(1/K)$, matching the finding by
~\cite{nesterov2005smooth}; however, our scheme is blessed with asymptotic guarantees rather than providing approximate solutions.  
In Section~\ref{sec:4.1}, we  derive rate and complexity statements in Section~\ref{sec:4.2}
for the iteratively smoothed {\bf VS-APM} (or {\bf sVS-APM}), recovering the optimal rate of $\mathcal{O}(1/K^2)$ with the optimal oracle
complexity of $\mathcal{O}(1/\epsilon^2)$ under smoothness. Finally, in
Section~\ref{sec:4.3}, under suitable choices of smoothing sequences, ({\bf
sVS-APM}) produces sequences that converge a.s. to an optimal solution.
\subsection{Smoothing \uvv{T}echniques}\label{sec:4.1}        
{In this section, we consider minimizing  $ F(x)\triangleq \mathbb E[\uss{\F}(x,\omega)]$, where $\uss{\f}(x,\omega)=\uss{\f}(x,\omega)+g(x)$ such that $\afo{f}$ and  $\afo{g}$ are convex and may be nonsmooth while  $g$ has an efficient prox evaluation (or ``proximable'') but $f$ is {\bf not  proximable.} Note that this setting is more general than structured nonsmooth problems, where \uss{the} function $\afo{f}$ is considered to be convex and smooth. In contrast to the previous section, we assume that $\nabla_x \uss{\f}_{\eta_k}(x_k,\omega_k)$ is generated from the
stochastic oracle, where $\eta_k$ is a smoothing parameter at iteration $k$ such that its sequence is diminishing.} \cite{beck12smoothing} define an $(\alpha,\beta)$-smoothable function as follows.
\begin{definition}[{\bf $(\alpha,\beta)$-smoothable~\cite{beck17fom}}] A
convex function $\uss{h}: \mathbb{R}^n \to \mathbb{R}$ is referred to as
$(\alpha,\beta)$-smoothable if \usv{for any $\eta > 0$}, there exists a convex differentiable function
$\uss{h}_{\eta}: \mathbb{R}^n \to \mathbb{R}$ that satisfies the following: (i)
$\uss{h}_{\eta}(x) \leq \uss{h}(x) \leq \uss{h}_{\eta}(x)+\eta \beta$ for all $x$; and (ii)
$\uss{h}_{\eta}$ is $\alpha/\eta$ smooth.  \end{definition}
There are a host of smoothing functions based on the nature of $\uss{h}$. For
instance, when $\uss{h}(x) = \|x\|_2$, then $\uss{h}_\eta(x) = \sqrt{\|x\|_2^2 + \eta^2} -
\eta$, implying that $\uss{h}$ is $(1,1)$-smoothable function. If $\uss{h}(x) =
\max(x_1,x_2, \hdots, x_n)$, then $\uss{h}$ is $(1,\log(n))$-smoothable and $\uss{h}_{\eta}(x) = \eta
\log(\sum_{i=1}^n e^{x_i/\eta})-\eta \log(n).$ ~(see~\cite{beck12smoothing} for
more examples). Recall that when $\uss{h}$ is a proper, closed, and convex function, the Moreau envelope is defined as 
$\uss{h}_\eta(x) \triangleq \min_{u} \ \left\{\uss{h}(u)+{1\over 2\eta}\|u-x\|^2 \right\}.$
{In fact, $\uss{h}$ is $(1,B^2)$-smoothable when $\uss{h}_\eta$ is given by the Moreau
envelope~(see~\cite{beck12smoothing}) and $B$ denotes a uniform
bound on $\|s\|$ in $x$ where $s \in \partial \uss{h}(x)$. There are a range of
other smoothing techniques including Nesterov
smoothing~(see~\cite{nesterov2005smooth}) and inf-conv smoothing (see~\cite{beck17fom}); our approach is agnostic to the choice of smoothing. \uss{In particular, if}
$\uss{\f(\cdot,\omega)}$ is a proper, closed, and convex function in $x$ for every $\omega$,
then $\uss{\f(\cdot,\omega)}$ is $(1,B^2)$-smoothable for every $\omega$ where $\uss{\f_{\eta}(\cdot,\omega)}$ is a suitable smoothing. \uss{In fact, if $\f(\cdot,\omega)$ satisfies the following smoothability assumption, then smoothability of $f$ follows, as shown by Lemma~\ref{smooth_f}. It is worth emphasizing that the smoothing of $f$, denoted by $f_{\eta}$ is defined as 
\begin{align}
f_{\eta}(x) \triangleq \mathbb{E}[\f_{\eta}(x,\omega)],
\end{align}
where $\f_{\eta}(\cdot,\omega)$ is a smoothing of $\f(\cdot,\omega)$.} 

\blu{\begin{assumption}\label{ass-smooth} {The function}  ${\f}(\cdot,\omega)$ is an $(\alpha(\omega),\beta(\omega))$-smoothable function for every $\omega \in \Omega$ where $\mathbb{E}[\alpha(\omega)] \leq \tilde{\alpha}$ and $\mathbb{E}[\beta(\omega)] \leq \tilde{\beta}$ {with $\tilde{\alpha},\tilde{\beta}> 0;$} {i.e. for any $\eta>0$},  there exists a convex differentiable function ${\f}_{\eta}(\cdot,\omega)$ {for every  $\omega \in \Omega$} such that 
    \begin{align*}
    {\f}_{\eta}(x,\omega) \leq {\f}(x,\omega) & \leq {\f}_{\eta}(x,\omega) + {\eta}\beta(\omega),  \quad \mbox{for all } x\\ \mbox{ and } \|\nabla_x {\f}_{\eta}(x,\omega)-\nabla_x {\f}_{\eta}(y,\omega) \| & \leq \tfrac{\alpha(\omega)}{\eta} \|x-y\|, \quad \qquad \mbox{for all } \ x, y  
     \end{align*}
     {where $\mathbb{E}[\alpha(\omega)] \leq \tilde{\alpha}$ and  $\mathbb{E}[\beta(\omega)] \leq \tilde{\beta}$.}
     \end{assumption}}

Based on the following Lemma, we observe that $f$ is $(\tilde{\alpha},\tilde{\beta})$-smoothable \uss{if} $\f(\cdot,\omega)$ satisfies suitable smoothability requirements for almost every $\omega \in \Omega$.   
\blu{\begin{lemma}\label{smooth_f}
        {Suppose  Assumption~6 holds.}  Then {there exist $\tilde\alpha,\tilde\beta>0$ such that} $f$ is $(\tilde{\alpha},\tilde{\beta})$-smoothable {where $f(x) \triangleq \mathbb{E}[{\f}(x,\omega)]$}. 
\end{lemma}}

 We proceed to develop a smoothed variant of ({\bf
VS-APM}), referred to as ({\bf sVS-APM}),  in which $\nabla_x \uss{\f}_{\eta_k}(x_k,\omega_k)$ is generated from the
stochastic oracle and $\eta_k$ is driven to zero at a sufficient rate (See Algorithm~\ref{smooth scheme}). 

 \begin{algorithm}
 \caption{\bf Iteratively smoothed VS-APM ({sVS-APM})}
\em
\label{smooth scheme}
(0) Given budget $M$, $\redd{x_0} \in X$, \redd{$y_0=x_0$} and  positive
sequences $\{\gamma_k,N_k\}$; Set $\lambda_0=0$, $\lambda_1=1$;  $k := 1$. \\
(1) $y_ {k+1}={\bf P}_{\gamma_k,g}\left(x_k- \gamma_k  (\nabla_x f_{\eta_k}(x_k) + {\bar w}_{k,N_k})\right)$;\\
(2) $\lambda_{k+1}=\frac{1+\sqrt{1+4\lambda_k^2}}{2}$;\\
(3) $x_{k+1}=y_ {k+1}+\frac{(\lambda_{k}-1)}{\lambda_{k+1}}\left(y_ {k+1}-y_ {k}\right)$;\\
(4) If $\sum_{j=1}^k N_j > M$, then stop; else $k := k+1$; return
to (1).
\end{algorithm}
\subsection{Rate and Complexity \uvv{A}nalysis}\label{sec:4.2}
{In this subsection, we develop rate and oracle complexity statements for Algorithm~\ref{smooth scheme} when $f$ is $(1,B^2)$ smoothable and then specialize these results to both the deterministic nonsmooth and the stochastic smooth regimes. We begin with a modified assumption.}
\begin{assumption}\label{fistaass3}
(i) The function $g$ is lower semicontinuous and convex with effective domain denoted by $\mbox{dom}(g)$; (ii) \uss{$f$} is proper, closed,  convex, and $(1,B^2)$-smoothable on an open set containing $\mbox{dom}(g)$;
(iii) There exists $C>0$ such that $\mathbb{E} [\Vert \redd{x_0}-x^{\ast }\Vert]\leq C$ {for all $x^* \in X^*$}. 
\end{assumption}
\blu{Note that Assumption~\ref{ass-smooth} represents a set of sufficiency conditions for $f$ to be smoothable; here, we directly assume that $f$ is smoothable to ease the exposition.}
\begin{lemma}\label{vs-apm smooth}
{Suppose Assumption~\ref{fistaass3} holds. Consider the iterates generated by {\bf (sVS-APM)} on $F(x)$. Suppose Assumption~\ref{ass_error2} holds for $f_{\eta_k}(x)$}. If $\{\gamma_k\}$ is a decreasing sequence {and $\gamma_k\leq \eta_k/2$}, then the following holds  
for all $K\geq 2$:
\begin{align*}
\mathbb{E}[F_{\eta_k}(y_K)-F_{\eta_k}({x^*})]\leq  {2\over \gamma_{K-1}(K-1)^2}\sum_{k=1}^{K-1} \gamma_k^2k^2 \frac{\nu^2}{N_k} 
 +  {2C^2\over \gamma_{K-1}(K-1)^2}.
\end{align*}
\end{lemma}
\begin{proof}
{By the update rule in Algorithm}~\ref{smooth scheme}, we have 
\begin{align}\label{fistaobjective func smooth}
y_  {k+1}=\arg \hspace{-0.025in}\min_x \ g(x)+{1\over 2\gamma_k}\|x-x_k\|^2+\left(\nabla_x f_{\eta_k}(x_k)+\bar w_k\right)^Tx.
\end{align}
{From the optimality condition} for \eqref{fistaobjective func smooth}, $0\in \partial g(y_  {k+1}) +{1\over \gamma_k}(y_  {k+1}-x_k)+\nabla_x f_{\eta_k}(x)+\bar w_k$. {By} convexity {of} $g(x)$, we {have} that $g(x)\geq g(y_k)+s^T(x-y_  {k+1})$ for all $s\in \partial g(y_k)$. {Hence}, we obtain the following.
\begin{align*}
g(x)+(\nabla_x f_{\eta_k}(x_k)+\bar w_k)^Tx\geq g(y_  {k+1})+(\nabla_x f_{\eta_k}(x_k)+\bar w_k)^Ty_  {k+1}-{1\over \gamma_k}(x-y_  {k+1})^T(y_  {k+1}-x_k).
\end{align*}
Now by using Lemma \ref{fistalem1}, we obtain that
\begin{align}\label{fistag(x) smooth}
&\qquad \nonumber g(x)+\left(\nabla_x f_{\eta_k}(x_k)+\bar w_k\right)^Tx+{1\over 2\gamma_k}\|x-x_k\|^2\\
&\geq g(y_  {k+1})+\left(\nabla_x f_{\eta_k}(x_k)+\bar w_k\right)^Ty_  {k+1}+{1\over 2\gamma_k}\|x_k-y_  {k+1}\|^2+{1\over 2\gamma_k}\|x-y_  {k+1}\|^2.
\end{align}
{By invoking the} convexity of {$f_{\eta_k}$} and {by} using the Lipschitz continuity {of} $\nabla_x f_{\eta_k}$, we obtain
\begin{align}\label{fistaf(x) smooth}
f_{{\eta_k}}(x)&\nonumber \geq f_{\eta_k}(x_k)+\nabla_x f_{\eta_k}(x_k)^T(x-x_k)\\\nonumber
&\geq f_{\eta_k}(y_  {k+1})+\nabla_x f_{\eta_k}(x_k)^T(x-y_  {k+1})-{1\over 2\eta_k}\|x_k-y_  {k+1}\|^2\\
&=f_{\eta_k}(y_  {k+1})+\left(\nabla_x f_{\eta_k}(x_k)+\bar w_k\right)^T(x-y_  {k+1})-{1\over 2\eta_k}\|x_k-y_  {k+1}\|^2-\bar w_k^T(x-y_  {k+1}),
\end{align}
where {the last equality follows from adding and subtracting} $\bar w_k$. By adding \eqref{fistag(x) smooth} and \eqref{fistaf(x) smooth}, we obtain
\begin{align}\label{fistaF(x) smooth}
F_{\eta_k}(y_  {k+1})-F_{\eta_k}(x)&\nonumber\leq {1\over 2\gamma_k}\|x-x_k\|^2-{1\over 2\gamma_k}\|x-y_  {k+1}\|^2+{1\over 2}\left({1\over \eta_k}-{1\over \gamma_k}\right)\|x_k-y_  {k+1}\|^2-\bar w_k^T(y_  {k+1}-x)\\ 
&= \left({1\over 2\eta_k}-{1\over \gamma_k}\right)\|x_k-y_  {k+1}\|^2+{1\over \gamma_k}(x_k-y_  {k+1})^T(x_k-x)-\bar w_k^T(y_  {k+1}-x),
\end{align}
where {the} last inequality {follows} from Lemma \ref{fistalem1} by choosing $Q=I$, $v_1=x_k$, $v_2=x$, and $v_3=y_k$.  By {setting} $x=y_{k}$ in \eqref{fistaF(x) smooth}, we {have}
\begin{align}\label{fistabound y smooth}
F_{\eta_k}(y_  {k+1})-F_{\eta_k}(y_  {k}) & \nonumber \leq \Big({1 \over 2\eta_k}-{1\over \gamma_k}\Big)\|x_k-y_  {k+1}\|^2+{1\over \gamma_k}(x_k-y_  {k+1})^T(x_k-y_  {k})\\ &-\bar w_{k,N_k}^T(y_  {k+1}-y_  {k}).
\end{align}
Similarly, by letting $x={x^*}$, we can obtain
\begin{align} \label{fistaeq3 smooth}
F_{\eta_k}(y_  {k+1})-F_{\eta_k}({x^*}) &\nonumber \leq \Big({1 \over 2\eta_k}-{1\over \gamma_k}\Big)\|x_k-y_  {k+1}\|^2+{1\over \gamma_k}(x_k-y_  {k+1})^T(x_k-x^*)\\ &-\bar w_{k,N_k}^T(y_  {k+1}-{x^*}).
\end{align}
 {By invoking Lemma \ref{fistalem1} where $v_1=x_k$, $v_2=y_{k+1}$ and $v_3=y_k$, we obtain
\begin{align*}
{1\over \gamma_k}(y_{k+1}-x_k)^T(y_k-x_k) = {1\over 2\gamma_k} \left(\|y_k-x_k\|^2 +\|y_{k+1}-x_k\|^2 - \|y_{k+1}-y_k\|^2\right). 
\end{align*}
Consequently, \eqref{fistabound y smooth} can further bounded as follows:
\begin{align}\label{bound y2}
\notag
&  F_{\eta_k}(y_{k+1})-F_{\eta_k}(y_k)   \leq \Big({1 \over 2\eta_k}-{1\over
		\gamma_k}\Big)\|x_k-y_{k+1}\|^2+{1\over
	\gamma_k}(x_k-y_{k+1})^T(x_k-y_k)
 -\bar w_{k,N_k}^T(y_{k+1}-y_k) \\
\notag 	& = \Big({1 \over 2\eta_k}-{1\over
		\gamma_k}\Big)\|x_k-y_{k+1}\|^2+{1\over 2\gamma_k} \left(\|x_k-y_k\|^2 +\|y_{k+1}-x_k\|^2 -\|y_{k+1}-y_k\|^2\right) -\bar w_{k,N_k}^T(y_{k+1}-y_k) \\
 	& = \Big({1 \over 2\eta_k}-{1\over
		2\gamma_k}\Big)\|x_k-y_{k+1}\|^2+{1\over 2\gamma_k} \left(\|x_k-y_k\|^2 - \|y_{k+1}-y_k\|^2\right)  -\bar w_{k,N_k}^T(y_{k+1}-y_k). 
\end{align}
Similarly, we have that 
\begin{align}\label{bound y3}
\notag
F_{\eta_k}(y_{k+1})-F_{\eta_k}({x^*}) &  \leq  \Big({1 \over 2\eta_k}-{1\over
		2\gamma_k}\Big)\|x_k-y_{k+1}\|^2+{1\over 2\gamma_k} \left(\|x_k-{x^*}\|^2 - \|y_{k+1}-{x^*}\|^2\right) \\ 
	& -\bar w_{k,N_k}^T(y_{k+1}-{x^*}). 
\end{align}
}
By multiplying \eqref{bound y2} by $(\lambda_k-1)$ and adding to \eqref{bound y3}, where $\delta_k\triangleq F_{\eta_k}(y_k)-F_{\eta_k}({x^*})$, we have 
  {\begin{align}
  \label{eqlam}  & \quad \lambda_k \delta_{k+1}-(\lambda_k-1)\delta_k 
 \leq\Big({1 \over 2\eta_k}-{1\over
		2\gamma_k}\Big)\lambda_k\|y_{k+1}-x_k\|^2   \\
\label{eqx}	& + {1\over 2\gamma_k} (\lambda_k - 1) \left(\|x_k-y_k\|^2 - \|y_{k+1}-y_k\|^2\right)  +{1\over 2\gamma_k} \left(\|x_k-{x^*}\|^2 - \|y_{k+1}-{x^*}\|^2\right) \\ 
	\label{eq4}	& + \bar w_{k,N_k}^T\left((\lambda_k-1)y_k+{x^*}-\lambda_ky_{k+1}\right).
\end{align}
Again by using Lemma \ref{fistalem1}, we may express the terms in \eqref{eqx} as follows:
\begin{align*}
	& \  {1\over 2\gamma_k} (\lambda_k - 1) \left(\|x_k-y_k\|^2 - \|y_{k+1}-y_k\|^2\right)  +{1\over 2\gamma_k} \left(\|x_k-{x^*}\|^2 - \|y_{k+1}-{x^*}\|^2\right) \\ 
	& =  {1\over 2\gamma_k} \left( \lambda_k \|x_k-y_k\|^2 - \lambda_k\|y_{k+1}-y_k\|^2  - \|x_k-y_k\|^2 + \|y_{k+1}-y_k\|^2
	 + \|x_k -{x^*}\|^2 - \|y_{k+1}-{x^*}\|^2\right) \\
	& = {1\over 2 \gamma_k} \left( -\lambda_k \|y_{k+1}-x_k\|^2 + 2\lambda_k (y_{k+1}-x_k)^T(y_k-x_k) + \|y_{k+1}-x_k\|^2 - 2(y_{k+1}-x_k)^T (y_k-x_k)\right. \\
	& \left. - \|y_{k+1}-x_k\|^2 + 2(y_{k+1}-x_k)^T({x^*}-x_k)\right)\\
	& = {1\over 2\gamma_k} \left( -\lambda_k \|y_{k+1}-x_k\|^2 + 2(y_{k+1}-x_k)^T((\lambda_{k}-1)y_k -\lambda_k x_k + {x^*})\right).	
\end{align*}
In addition,  
\begin{align*}
& \quad	\bar w_{k,N_k}^T\left((\lambda_k-1)y_k+{x^*}-\lambda_ky_{k+1}\right)  =\bar w_{k,N_k}^T\left((\lambda_k-1)y_k+{x^*}-\lambda_kx_k \right) +   
\bar w_{k,N_k}^T\left(\lambda_kx_k - \lambda_k y_{k+1} \right).
\end{align*}
From the update rule, $\lambda_{k-1}^2=\lambda_k(\lambda_k - 1)
	= \lambda_k^2-\lambda_k$. Now by
		multiplying \eqref{eqlam} by $\lambda_k$, we obtain the following, where $u_k=(\lambda_k-1)y_k-\lambda_kx_k+x^*$:
\begin{align}\label{bound lamk}
&   \quad \lambda_k^2\delta_{k+1}-\lambda_{k-1}^2\delta_k \leq \lambda_k^2 \left( \frac{1}{2\eta_k} - {1\over 2\gamma_k}\right) \|y_{k+1}-x_k\|^2\\ \nonumber
& + {1\over 2\gamma_k}\left(-\|\lambda_k y_{k+1}-\lambda_k x_k\|^2 + 2 (\lambda_k y_{k+1}-\lambda_k x_k)^T ((\lambda_k-1)y_k + x^* -\lambda_k x_k) \right) \\ \nonumber
& - \lambda_k^2 \bar{w}_{k,N_k}^T(x_k-y_{k+1}) -\lambda_k w_k^T u_{k}  = \lambda_k^2 \left( \frac{1}{2\eta_k} - {1\over 2\gamma_k}\right) \|y_{k+1}-x_k\|^2- \lambda_k^2 \bar{w}_{k,N_k}^T(x_k-y_{k+1})\\ \nonumber
& + {1\over 2\gamma_k}\left(\|\lambda_k x_k - (\lambda_k - 1) y_k -{x^*}\|^2 - \|\lambda_k y_{k+1}- (\lambda_k-1) y_k -{x^*}\|^2\right) 
  -\lambda_k w_k^T u_{k} \\ \nonumber
& \leq \frac{\lambda_k^2}{{2 \over \gamma_k} - {2 \over \eta_k}} \|\bar{w}_{k,N_k}\|^2 
 + {1\over 2\gamma_k}\left(\|u_k\|^2 - \|u_{k+1}\|^2\right)
  -\lambda_k w_k^T u_{k}, 
\end{align}
where in the last inequality we used the update rule of algorithm, $x_{k+1}=y_{k+1}+\frac{\lambda_k-1}{\lambda_{k+1}}(y_{k+1}-y_k)$, to obtain the following:
$$u_{k+1}=(\lambda_{k+1}-1)y_{k+1}-\lambda_{k+1}x_{k+1}{+}{x^*}=(\lambda_k-1)y_k-\lambda_ky_{k+1}{+}{x^*}.$$
By multiplying both sides by $\gamma_k$ and assuming $\gamma_k\leq \gamma_{k-1}$, we obtain
\begin{align}\label{moreau a.s.}
&\gamma_k\lambda_k^2\delta_{k+1}-\gamma_{k-1}\lambda_{k-1}^2\delta_{k} \leq \frac{\gamma_k \lambda_k^2}{{2 \over  \gamma_k} - {2 \over \eta_k}} \|\bar{w}_{k,N_k}\|^2 
 + {1\over 2}\left(\|u_k\|^2 - \|u_{k+1}\|^2\right)
  -\gamma_k\lambda_k w_k^T u_{k}.
  \end{align}
  By assuming $\gamma_k\leq {\eta_k\over 2}$, we obtain ${1\over \gamma_k}-{1\over \eta_k}\geq {1\over 2\gamma_k}$,  {implying that}
 \begin{align} \label{ineqk}
 &\gamma_k\lambda_k^2\delta_{k+1}-\gamma_{k-1}\lambda_{k-1}^2\delta_{k} \leq \gamma_k^2 \lambda_k^2\|\bar{w}_{k,N_k}\|^2 
 + {1\over 2}\left(\|u_k\|^2 - \|u_{k+1}\|^2\right)
  -\gamma_k\lambda_k w_k^T u_{k}.
 \end{align} 
  Summing {\eqref{ineqk} from} $k = 1$ to $K-1$, we have the following:
\begin{align*}
&\gamma_{K-1}\lambda_{K-1}^2 \delta_K \leq  \sum_{k=1}^{K-1} \gamma_k^2\lambda_k^2\|\bar{w}_{k,N_k}\|^2 
 +  {1\over 2}\|u_1\|^2
  -\sum_{k=1}^{K-1}\gamma_k \lambda_k w_k^T u_{k} \\
&\implies  \delta_K \leq  {1\over \gamma_{K-1}\lambda_{K-1}^2}\sum_{k=1}^{K-1} \gamma_k^2\lambda_k^2 \|\bar{w}_{k,N_k}\|^2 
 +  {1\over 2\gamma_{K-1}\lambda_{K-1}^2}\|u_1\|^2
 -{1\over \gamma_{K-1}\lambda_{K-1}^2}\sum_{k=1}^{K-1}\gamma_k\lambda_k w_k^T u_{k}. 
\end{align*}
Taking expectations, we note that the last term on the right is zero (under a zero bias assumption), leading to the following:
\begin{align*}
\mathbb{E}[\delta_K] \leq  {1\over \gamma_{K-1}\lambda_{K-1}^2}\sum_{k=1}^{K-1} \gamma_k^2\lambda_k^2 \frac{\nu^2}{N_k} 
 +  {1\over 2\gamma_{K-1}\lambda_{K-1}^2}\mathbb{E}[\|u_1\|^2\|]&\leq  {2\over \gamma_{K-1}(K-1)^2}\sum_{k=1}^{K-1} \gamma_k^2k^2 \frac{\nu^2}{N_k} 
 \\&+  {2C^2\over \gamma_{K-1}(K-1)^2},
\end{align*}}
where in the last inequality we used the fact that  $\|y-x^*\|\leq C$ for all $y\in {\rm dom}(g)$ and  ${k\over 2}\leq \lambda_k\leq k$ which may be shown inductively.
\end{proof}
We are now ready to prove our main rate result and oracle complexity bound for ({\bf sVS-APM}). 
{\begin{theorem}[{\bf Rate Statement and Oracle Complexity Bound for (sVS-APM)}]\label{rate-sVS-APM}
{Suppose Assumption~\ref{fistaass3} holds. Consider the iterates generated by {\bf (sVS-APM)} on $F(x)$. Suppose Assumption~\ref{ass_error2} holds for $f_{\eta_k}$}. {Suppose $\{\lambda_k\}$ is specified in {\bf (sVS-APM)}, 
$\eta_k={1 \slash k}$, $\gamma_k=1/  {2}k$, and $N_k = \lfloor k^a\rfloor$.} \\
\noindent (i) The following holds for
any $K\geq 1$: 
{\begin{align*}
 \mathbb{E}[F(y_{K+1})-F(x^*)]\leq \begin{cases}
			\tfrac{ \left(\tfrac{2\nu^2a}{a-1}+4C^2+B^2\right)}{K}, & a = 1+\delta, \delta \in [\delta_L,\delta_U] \\
			\tfrac{ 2\nu^2(1+\log(K))+4C^2+B^2}{K}, & a = 1 
			\end{cases}
\end{align*}
\noindent (ii) Let $\epsilon \leq \tilde C/2$ and $K$ is such that $\mathbb{E}[F(y_{K+1})-F(x^*)]\leq \epsilon$. Then the following holds.
 $$\sum_{k=1}^K N_k\leq \begin{cases}
		{\mathcal O}\left(\tfrac{1}{\epsilon^{2+\delta_L}}\right), & a = 1+\delta, \delta \in [\delta_L, \delta_U] \\
		 {\mathcal O}\left(\tfrac{1}{\epsilon^2} \log^2(1/\epsilon)\right). & a = 1 
		\end{cases} $$}
\end{theorem}}
\begin{proof}
{\bf (i)} If $N_k=\lfloor k^{a}\rfloor\geq {1\over 2} k^{a}$ and $\gamma_k=1/(2k)$ {is utilized in Lemma~\ref{vs-apm smooth}}, we obtain the following
\begin{align}\label{bound nk}
\mathbb{E}[\delta_{K+1}] \leq  {2\nu^2\over K}\sum_{k=1}^{K}\frac{1}{k^{a}} 
 +  {4C^2\over K}. 
\end{align}
\noindent {(a) {\em $a = 1+ \delta$ where $\delta \in [\delta_L,\delta_U]$}. Consequently, we may derive the next bound.}
\begin{align*}
\sum_{k=1}^{K} k^{-a}=1+\sum_{k=2}^{K} k^{-a}\leq 1+\int_1^{K} k^{-a} dk= 1+ {1-K^{1-a}\over a-1}\leq 
\frac{ {1+\delta_U}}{\delta_L}.
\end{align*}
{By invoking $(1,B^2)$-smoothability of $f$ and $\eta_K=1/K$, we have that 
$F_{{\eta_K}}(y_{K+1}) \leq F(y_{K+1})$ and $-F_{\eta_K}(x^*) \leq -F(x^*) + \eta B^2$.}
Hence, the required bound follows from \eqref{bound nk}  
\begin{align*}
\mathbb{E}[F(y_{K+1})-F(x^*)]&\leq  {2\nu^2 a\over (a-1)K}
 +  {4C^2+ {B^2}\over K}\leq {\bar C \over K}, \mbox{ where } \bar C \triangleq  {2\nu^2 a\over (a-1)}
 +  {4C^2}+{B^2}. 
\end{align*}
\noindent {(b)} {{\em $a = 1$.} Recall that the convergence rate is given by the following: \begin{align*}
 \mathbb{E}[F(y_{K+1})-F(x^*)]\leq \frac{{2\nu^2 (a-K^{1-a})\over (a-1)}
 +  {4C^2}+{B^2}}{K}.
\end{align*}
Taking limits, we obtain that
$$ \lim_{a\to 1} \frac{a-K^{1-a}}{a-1}= \lim_{a \to 1}\frac{1+K^{1-a}\log(K)}{1} =   1+\log(K).$$ Therefore, we have that
\begin{align*}
 \mathbb{E}[F(y_{K+1})-F(x^*)]\leq \frac{{2\nu^2 \log(K)}
 +  {4C^2}+{B^2}}{K} {\ \triangleq \ } \frac{a+b \log(K)}{K}.
\end{align*}}
{\bf (ii)} Consider $y_{K+1}$ satisfying $\mathbb{E}[F(y_{K+1})-F(x^*)]\leq
\epsilon$. {We again consider two cases.} 
{(a) {\em $a = 1+ \delta$ where $\delta \in [\delta_L,\delta_U]$}}. Since 
we have ${\bar C\over K}\leq \epsilon$ which implies that $K=\lceil
\bar C/\epsilon\rceil$. To obtain the optimal oracle complexity we require
$\sum_{k=1}^K N_k$ gradients. Hence, the following holds for sufficiently small
$\epsilon$ such that $2\leq \bar C/\epsilon$: \begin{align*}
\sum_{k=1}^K N_k\leq \sum_{k=1}^K k^{a}=\sum_{k=1}^{1+\bar C/\epsilon} k^{a}\leq \int_0^{2+\bar C/\epsilon} k^a da=\frac{{(2+\bar C/\epsilon)^{1+a}}}{1+a}\leq 
{\left({\bar C\over \epsilon}\right)^{1+a}} \leq \mathcal{O}\left( \frac{1}{\epsilon^{1+a}}\right)\leq 
\mathcal{O}\left( \frac{1}{\epsilon^{{2+\delta_L}}}\right).
\end{align*}
\noindent {(b) {\em $a = 1$.}}  
{To compute $K$ such that $\frac{a+b \log(K)}{K} \leq \epsilon$ is not immediately obvious but may be obtained via the Lambert function\footnote{The Lambert function $W(x)$ is the inverse function of $y e^y = x$ and is denoted by $y = W(x)$. This function has two real branches: an upper branch $W_0(x)$ for $x \in [-\tfrac{1}{e},+\infty]$ and a lower branch $W_{-1}(x)$ for $x \in [-\tfrac{1}{e},0]~$\cite{veberic2010having}.}~\cite{chatzigeorgiou2013bounds}. For purposes of simplicity, suppose $a = 0$ and $b=1$. Then we have the following.
\begin{align*}
        \tfrac{\log(K)}{K} \leq \epsilon &  \Leftrightarrow \tfrac{-\log(K)}{K} \geq -\epsilon \\
                &  \Leftrightarrow W_{-1}\left(\tfrac{-\log(K)}{K}\right) \leq W_{-1}(-\epsilon), \mbox{ since } W_{-1}(\cdot) \mbox{ is decreasing}.
\end{align*}
But $W_{-1}(-\tfrac{\log(x)}{x}) = -\log(x)$ for $x > e$. Consequently, we have that
\begin{align*}
        -\log(K) \leq  W_{-1}(-\epsilon) \Leftrightarrow K \geq e^{-W_{-1}(-\epsilon)}.
\end{align*}
By definition of the Lambert function, we have that $e^{W(x)} = \tfrac{x}{W(x)}$, implying that
\begin{align*}
K \geq e^{-W_{-1}(-\epsilon)} = \tfrac{W_{-1}(-\epsilon)}{\epsilon} \geq \mathcal{O}\left(\tfrac{\log(\epsilon)}{-\epsilon}\right) = \mathcal{O}\left(\tfrac{1}{\epsilon}\log(1/\epsilon)\right).
\end{align*}
where the first inequality follows from (3) in ~\cite{chatzigeorgiou2013bounds}.} 
 Hence, the oracle complexity for $a = 1$ will be
$\mathcal O\left(\frac{\log^2(1/\epsilon)}{\epsilon^2}\right)$, which is near
optimal (where optimal is $\mathcal {O}(1/\epsilon^2))$.\end{proof}

{We now consider two cases of Theorem~\ref{rate-sVS-APM} for which similar rate statements are available.}

\noindent {\bf Case 1. Structured stochastic nonsmooth optimization with $f$ smooth.} 
{Now consider problem \eqref{main problem}, where $f(x)$ is a smooth
function. Recall that we considered such a problem in Section~\ref{sec. strong}
for strongly convex $f$ and in this case, we consider the merely convex case.
When $f$ is deterministic, accelerated gradient methods first proposed by~\cite{nesterov83} and
their proximal generalizations suggested by~\cite{beck2009fast} were characterized by the
optimal rate of convergence of $\mathcal{O}(1/K^2)$. When $f$ is
expectation-valued, ~\cite{ghadimi2016accelerated} presented the
first known accelerated scheme for stochastic convex optimization where the
optimal rate of $1/k^2$ was shown for the expected sub-optimality error.  This
rate required choosing the simulation length $K$ and choosing $N_k=\lfloor k^2
K\rfloor$ which led to the optimal oracle complexity of
$\mathcal{O}(1/\epsilon^2)$. However, this method is somewhat different from ({\bf VS-APM}). In
particular, every step requires two prox evaluations (rather than one for
{\bf VS-APM}).\footnote{While pursuing submission of the present work, we were informed of
related work by ~\cite{jofre2017variance} through a private
communication.} ~\cite{jofre2017variance} developed an accelerated
proximal scheme for convex problems with a similar algorithm but allow for
state dependent noise. The weakening of the noise requirement still allows for
deriving the optimal rate of $\mathcal{O}(1/K^2)$ but necessitates  choosing
$N_k=\lfloor k^3(\ln k)\rfloor$. As a consequence, the oracle complexity is
slightly poorer than the optimal level and is given by $\mathcal
O\left(\epsilon^{-2} \ln^2 (\epsilon^{-0.5})\right)$. We note that ({\bf
VS-APM}) displays the optimal oracle complexity $\mathcal O(\epsilon^{-2})$ by
choosing $N_k=\lfloor k^2 K\rfloor$ while by choosing $N_k=\lfloor k^a\rfloor$
for $a=3+\delta$, then the oracle complexity can be made arbitrarily close to
optimal and is given by $\mathcal O(\epsilon^{-2{-}\delta/2})$. However, ({\bf
VS-APM}) imposes a stronger assumption on noise, as formalized next.}
\begin{corollary} {\bf (Rate and oracle complexity bounds with smooth $f$ for ({\bf VS-APM}))} Suppose {Assumptions~~\ref{smooth-f}, \ref{ass_error2},} and \ref{fistaass3} hold. Suppose $\gamma_k=\gamma\leq 1/2L$ for all $k$. \\
\noindent (i) Let  $N_k=\lfloor k^a\rfloor$ where $a=3+\delta$ and $\widehat C \triangleq {2\nu^2\gamma (a-2)\over a-3}+{4C^2 \over \gamma}$. Then the following holds.
\begin{align*}
\mathbb{E}[F(y_{{K}+1}-F(x^*))] & \leq \frac{\widehat C}{K^2} \mbox{ for all } K \mbox{ and } 
\sum_{k=1}^{K(\epsilon)} N_k  \leq {\mathcal O\left({1\over \epsilon^{2+\delta/2}}\right)}, 
\end{align*}
{ where } $\mathbb{E}[F(y_{{K(\epsilon)}+1})-F(x^*)]\leq \epsilon.$ \\
\noindent (ii) {Given a $K >$ 0}, let  $N_k=\lfloor k^2 K\rfloor$ where $a>3$ and $\tilde C 
\triangleq 2\nu^2\gamma+{4C^2\over \gamma}$. Then the following holds.
\begin{align*}
\mathbb{E}[F(y_{{K}+1}-F(x^*))] \leq  \frac{\tilde C}{K^2} \mbox{ and }  
\sum_{k=1}^{K} N_k  \leq  
\mathcal{O}\left({1\over {\epsilon^{2}}}\right), \mbox{ where } \mathbb{E}[F(y_{{K}+1})-F(x^*)]\leq \epsilon. 
\end{align*}
\end{corollary}
\begin{proof}
{\bf (i)} Similar to the proof of Lemma \ref{vs-apm smooth}, by defining $\delta_k=F(y_k)-F(x^*)$ we can prove:
\begin{align*}
\mathbb{E}[F(y_{K+1})-F(x^*)] \leq  {2\nu^2\gamma\over K^2}\sum_{k=1}^{K}\frac{k^2}{k^{a}} 
 +  {4C^2\over \gamma K^2}.
\end{align*}
Let $N_k=\lfloor k^{a}\rfloor\geq {1\over 2} k^{a}$ and $\gamma_k=\gamma$. {Then}  we have that the following holds where $\widehat C \triangleq {2\nu^2\gamma (a-2)\over a-3}+{4C^2 \over \gamma}.$
\begin{align}\label{bound nk nonsmooth}
\mathbb{E}[F(y_{K+1})-F(x^*)] \leq  {2\nu^2\gamma\over K^2}\sum_{k=1}^{K}\frac{k^2}{k^{a}} 
 +  {4C^2\over \gamma K^2}\leq  {2\nu^2\gamma (a-2)\over (a-3)K^2}
 +  {4C^2\over \gamma K^2}=
{\frac{\widehat C}{K^2}}, 
\end{align}
where the first inequality follows from bounding the summation as follows:
\begin{align*}
\sum_{k=1}^{K} k^{2-a}=1+\sum_{k=2}^{K} k^{2-a}\leq 1+\int_1^{K} {x}^{2-a} d{x}={1\over a-3}-\frac{K^{3-a}}{a-3}+1\leq {1\over a-3}+1={a-2\over a-3}.
\end{align*}
{Suppose} $y_{K+1}$ satisfies $\mathbb{E}[F(y_{K+1})-F(x^*)]\leq \epsilon$,
implying that ${\widehat C\over K^2}\leq \epsilon$ or $K=\lceil\widehat
C^{1/2}~\slash~\epsilon^{1/2}\rceil$. If $\epsilon \leq \widehat C/2$, then the oracle complexity can be bounded as follows:\begin{align*}
\sum_{k=1}^K N_k\leq \sum_{k=1}^K k^{a}=\sum_{k=1}^{1+\sqrt{\widehat C/\epsilon}} k^{a}\leq \int_0^{2+\sqrt{\widehat C/\epsilon}} k^a da=\frac{{(2+\sqrt{\widehat C/\epsilon})^{1+a}}}{1+a}\leq {\left({\sqrt {\hat C}\over 2\sqrt\epsilon}\right)^{1+a} = \mathcal{O}\left(1\over \epsilon^{2+\delta/2}\right)}.
\end{align*}
{(ii) Let $N_k=\lfloor k^2 K\rfloor \geq {1\over 2} k^2 K$. Then similar to part (i), we may bound the expected sub-optimality as follows where $\tilde C \triangleq 2\nu^2\gamma+{4C^2\over \gamma}$.}
\begin{align*}
\mathbb{E}[F(y_{K+1})-F(x^*)] \leq  {2\nu^2\gamma\over K^2}\sum_{k=1}^{K}\frac{k^2}{k^2K}
 +  {4C^2\over \gamma K^2}={2\nu^2\gamma\over K^2} +  {4C^2\over \gamma K^2}\leq {\tilde C \over  K^2}.
\end{align*}
Since $K = \lceil \tilde C^{1/2} \slash \epsilon^{1/2}\rceil$, the oracle complexity may be bounded as follows:
\begin{align*}
\sum_{k=1}^K N_k\leq \sum_{k=1}^K k^{2}K={1\over 6} K^2(K+1)(2K+1)= {1\over 6}K^2(2K^2+3K+1)\leq K^4\leq \mathcal{O} \left({1\over \epsilon^{2}}\right) . 
\end{align*}
\end{proof}
\noindent {\bf Case 2: Deterministic nonsmooth convex optimization.}
{When the function $f$ in \eqref{main problem} is deterministic but possibly
nonsmooth, ~\cite{nesterov2005smooth} showed that by applying an
accelerated scheme to a suitably smoothed problem (with a fixed smoothing
parameter)  leads to a convergence rate of $\mathcal O(1/K)$. In contrast with
Theorem~\ref{rate-sVS-APM}, utilizing a fixed smoothing parameter leads to an
approximate solution at best and such a scheme is not characterized by
asymptotic convergence guarantees. In addition, we observe that the rate
statement for  \usv{the deterministic counterpart of ({\bf sVS-APM}), denoted by ({\bf s-APM}),} is global (valid for all $k$) while \usv{any statement with} constant smoothing
holds for the prescribed $K$. We observe that the rate statements by using an
appropriately chosen smoothing and steplength parameter matches that by using a
selecting a suitable smoothing and steplength sequence.} 
\begin{corollary}\label{cor-detns} {\bf (Iterative vs constant smoothing for deterministic nonsmooth convex optimization)} Consider \eqref{main problem} and assume $f(x)$ is a deterministic function. Suppose Assumption \ref{fistaass3} holds. 
\noindent {\bf (i)}  Iterative smoothing: {Suppose $\gamma_k=1/2k$ and $\eta_k=1/k$. Then, $F(y_{k+1})-F(x^*)\leq  {4C^2+B^2\over k},$ for all $k > 0.$}
\noindent {\bf (ii)} Fixed smoothing: For a given $K > 0$, suppose {$\eta_k=1/K$ and $\gamma_k=1/2K$. Then, $F(y_{K+1})-F(x^*)\leq  {4C^2+B^2 \over K}.$}
\end{corollary}
\blue{\begin{remark} {{By recalling that} $f_\eta(x)\triangleq 
\mathbb E[{\f}_\eta(x,\omega)]$, {by} using Theorem 7.47 in \cite{shapiro09lectures} (interchangeability of the derivative and the expectation), and noting that ${\f_{\eta}(\cdot,\omega)}$ is differentiable in $x$ for every $\omega$, 
 we have $\nabla 
f_\eta(x)=\nabla \mathbb E[{\f}_\eta(x,\omega)]=\mathbb E[\nabla 
{\f}_\eta(x,\omega)]\implies \mathbb E[\nabla f_\eta(x)-\nabla 
{\f}_\eta(x,\omega)]=0.$  
Therefore, such a gradient estimator is unbiased and our assumption holds. {We now derive bounds on the second moments for some common smoothings in Table~\ref{bd-sec-moment}.}} 
\begin{table}[htpb]
\tiny \hspace{-0.3in}
{\begin{tabular}{|c|c|c|c|c|} \hline
$\uss{\f}(x,\omega)$&$\uss{\f}_\eta(x,\omega)$&$\nabla \uss{\f}_\eta(x,\omega)$ & $\mathbb{E}[\|\nabla_x \uss{\f}_{\eta}(x,\omega)-\nabla_x f_{\eta}(x)\|^2]$\\ \hline
$\uss{\f}_{1}(x,\omega)=\lambda(\omega)\|x\|_1$& $\sum_{i=1}^n h_{\uss{\eta}}(x_i,\omega)$, where&  $[\nabla_{x_i}h_{\uss{\eta}}(x_i,\omega)]_{i=1}^n$, where& \\
&
 $  \uss{h_{\eta}}(x_i,\omega)$ = $\left\{\begin{array}{lr}
        \lambda^2(\omega){x_i^2\over 2\eta}, &\hspace{-0.9cm} \lambda(\omega)|x_i|<\eta\\
        \lambda(\omega)|x_i|-\eta/2, &  o.w.
        \end{array}\right\}$&   $ \nabla_{x_i} \uss{h_{\eta}}(x_i,\omega)$ = $\left\{\begin{array}{lr}
        \lambda^2(\omega){x_i\over \eta}, &\hspace{-0.6cm} \lambda(\omega)|x_i|<\eta\\
        \lambda(\omega)x_i/|x_i|, &  o.w.
        \end{array}\right\}$&{$4n\mathbb{E}[\lambda^2(\omega)]$} \\ \hline
$\uss{\f}_{2}(x,\omega)=\lambda(\omega)\|x\|_2$&$\sqrt{\lambda^2(\omega)\|x\|^2+\eta^2}-\eta$&${\lambda^2(\omega)x\over \sqrt{\lambda^2(\omega)\|x\|^2+\eta^2}}$&{$ 4 \mathbb{E}[\lambda^2(\omega)]$}\\ \hline
\tabincell{l}{$\uss{\f}_{3}(x,\omega)$ $=\displaystyle \max_{1 \leq i \leq n}\{h_i(x,\omega)\}$ \\ where 
$h_i(x,\omega) = v_i+s_ic(\omega)^Tx$} &$\eta \log\left(\sum_{i=1}^n
\exp({h_i(x,\omega)/ \eta})\right)$ &
${\sum_{i=1}^n \nabla_{x}
h_i(x,\omega)\exp({h_i(x,\omega)/ \eta})\over \sum_{i=1}^n\exp({h_i(x,\omega)/
\eta})}$&{$ 
4\mathbb{E}\left[\left(\displaystyle \max_{1 \leq i \leq n} \|s_i c(\omega)\|\right)^2\right]$},\\ &&&
\\ \hline
\end{tabular}}
\caption{Bounding the second moments for certain smoothings}
\label{bd-sec-moment}
\end{table}
\end{remark}}
\subsection{Almost-sure \uvv{C}onvergence}\label{sec:4.3}
While the previous subsection focused on providing rate statements for expected sub-optimality, we now consider the open question of whether the sequence of iterates produced by ({\bf sVS-APM}) \usv{converges} a.s. to a solution. Schemes employing a constant smoothing parameter preclude such guarantees. Proving a.s. convergence requires using the following lemma. 
\begin{lemma}[Supermartingale convergence lemma (\cite{Polyak87})] \label{almost sure}
Let $\{v_k\}$ be a sequence of nonnegative random variables, where $\mathbb
E{[v_0]}<\infty$ and let $\{\alpha_k\}$ and $\{\eta_k\}$ be deterministic scalar
sequences such that $0\leq \alpha_k \leq 1$ and $\eta_k\geq 0$ for all $k\geq
0$, $\sum_{k=0}^\infty\alpha_k=\infty$, $\sum_{k=0}^\infty \eta_k<\infty $, and
$\lim_{k\rightarrow \infty}{\eta_k\over \alpha_k}=0$, and $\mathbb
E{[v_{k+1}\mid \mathcal H_k]\leq (1-\alpha_k)v_k+\eta_k}$ a.s. for all $k\geq
0$. Then, $v_k\rightarrow 0$ a.s. as $k\rightarrow \infty$.
\end{lemma}

\begin{proposition}\label{a.s. smoothing}{{\bf (a.s. convergence of {\bf (sVS-APM)})}} Suppose Assumptions~\ref{ass_error2} and \ref{fistaass3} hold and $\{{y}_k\}$ is a sequence generated by ({\bf sVS-APM}). Suppose $\gamma_k=k^{-b}  <  \eta_k$, where $b\in (0,1/2]$, $\{\eta_k\}$ is a decreasing sequence, and $N_k=\lfloor k^{a}\rfloor$ such that $(a+b)>1$. Then $\{{y}_k\}$ converges to {a} solution of~\eqref{main problem} a.s. .  
\end{proposition}
\begin{proof}
From inequality~\eqref{bound lamk}, we have that the following holds. 
\begin{align*}
\gamma_k\delta_{k+1}
&\leq {\lambda_{k-1}^2\over \lambda_k^2}\gamma_k\delta_{k} +{1\over 2 \lambda_k^2}\left( \|u_{k}\|^2-\|u_{k+1}\|^2\right)+\left({\gamma_k\over {2\over \gamma_k}-{2\over \eta_k}}\right)\|\bar w_{k,N_k}\|^2 -{1\over \lambda_k} \bar w_{k,N_k}^T u_{k}\\
&\leq {\lambda_{k-1}^2\over \lambda_k^2}\gamma_{k-1}\delta_{k} +{1\over 2 \lambda_k^2}\left( \|u_{k}\|^2-\|u_{k+1}\|^2\right)+ \left({\gamma_k\over {2\over \gamma_k}-{2\over \eta_k}}\right)\|\bar w_{k,N_k}\|^2 -{1\over \lambda_k} \bar w_{k,N_k}^T u_{k}.
\end{align*}
Dividing both sides of the previous inequality by $\gamma_k$,  we obtain the following relationship.
\begin{align*}\delta_{k+1}+{1\over 2 \gamma_k\lambda_k^2}\|u_{k+1}\|^2&\leq {\lambda_{k-1}^2\over \lambda_k^2 \gamma_k}\gamma_{k-1}\delta_{k}+{1\over 2 \gamma_k \lambda_k^2}\|u_{k}\|^2  + \left({1\over {2\over \gamma_k}-{2\over \eta_k}}\right)\|\bar w_{k,N_k}\|^2-{1\over {\gamma_k} \lambda_k} \bar w_{k,N_k}^T u_{k}\\
&={\lambda_{k-1}^2\gamma_{k-1}\over \lambda_k^2\gamma_k}\left(\delta_k+{\|u_k\|^2\over 2\gamma_{k-1}\lambda_{k-1}^2}\right)+ \left({1\over {2\over \gamma_k}-{2\over \eta_k}}\right)\|\bar w_{k,N_k}\|^2-{1\over {\gamma_k} \lambda_k} \bar w_{k,N_k}^T u_{k}.
\end{align*}
By defining $v_{k+1}\triangleq \delta_{k+1}+{1\over 2\gamma_k\delta_k^2}\|u_{k+1}\|^2$ and $\alpha_k\triangleq 1-  {\lambda_{k-1}^2\gamma_{k-1}\over \lambda_k^2\gamma_k}$, we have the following {recursion}.
\begin{align} \notag
v_{k+1} & \leq (1-\alpha_k)v_k+\left({1\over {2\over \gamma_k}-{2\over \eta_k}} \right)\|\bar w_{k,N_k}\|^2 -{1\over {\gamma_k} \lambda_k} \bar w_{k,N_k}^T u_{k} \iff \\ 
 v_{k+1}+\eta_kB^2 & \leq (1-\alpha_k)(v_k+\eta_{k-1}B^2)+\eta_kB^2-(1-\alpha_k)\eta_{k-1}B^2+\left(\tfrac{1}{ {2\over \gamma_k}-{2\over \eta_k}} \right)\|\bar w_{k,N_k}\|^2
 \label{lyap} -{1\over {\gamma_k} \lambda_k} \bar w_{k,N_k}^T u_{k}.
\end{align}
{Let $\bar v_{k+1} \triangleq v_{k+1}+\eta_kB^2$. From $(1,B^2)$ smoothability and the decreasing nature of $\{\eta_k\}$, 
$$ 0 \leq F(y_{k+1}) - F(x^*) \leq F_{\eta_{k+1}}(y_{k+1}) - F_{\eta_{k+1}}(x^*) + \eta_{k+1} B^2 \leq  F_{\eta_{k+1}}(y_{k+1}) - F_{\eta_{k+1}}(x^*) + \eta_{k} B^2. $$ 
 Then \eqref{lyap} can be rewritten as  follows:
\begin{align*}
\bar v_{k+1}\leq (1-\alpha_k)\bar v_k+\eta_kB^2-(1-\alpha_k)\eta_{k-1}B^2+\left({1\over {2\over \gamma_k}-{2\over \eta_k}} \right)\|\bar w_{k,N_k}\|^2 -{1\over {\gamma_k} \lambda_k} \bar w_{k,N_k}^T u_{k}
\end{align*}}
{Recall by the} definition of $\lambda_k$, we have $\lambda_{k-1}^2={(2\lambda_k-1)^2-1\over 4}$ and ${k\over 2}\leq\lambda_k\leq k$, if $\gamma_k=k^{-b}$, $b\in (0,1/2]$, we obtain the following relationship.
\begin{align}\label{alpha}\alpha_k&\nonumber=1-{\lambda_{k-1}^2\gamma_{k-1}\over \lambda_k^2\gamma_k}=1-{\gamma_{k-1}(4\lambda_k^2-4\lambda_k)\over 4\lambda_k^2 \gamma_k}={\lambda_k^2\gamma_k-\gamma_{k-1}\lambda_k^2+\gamma_{k-1}\lambda_k\over \lambda_k^2\gamma_k}={\gamma_k-\gamma_{k-1}\over \gamma_k}+{\gamma_{k-1}\over\lambda_k\gamma_k}\\
&\geq {k^{-b}-(k-1)^{-b}\over k^{-b}}+{(k-1)^{-b}\over k^{1-b}}={k^{1-b}-(k-1)^{1-b}\over k^{1-b}}\geq {{(1-b)}\over k}, \quad {b\in(0,1/2]},\end{align}
where in the last inequality we use $b\in (0,1/2]$:
{\begin{align*}
& \quad k \left({k^{1-b}-(k-1)^{1-b}\over k^{1-b}}\right) 
  = k - k \left( {{k-1} \over k }\right)^{1-b} 
				 = k - k^b (k-1)^{1-b}= k - (k-1) \left( {k \over k-1}\right)^b  \\ 
				& 
				 = k - (k-1) \left( 1+ {1 \over k-1}\right)^b 
				 = k - (k-1) - {b} - \frac{b(b-1)}{2 ! (k-1)^2} - \frac{b(b-1)(b-2)}{3 ! (k-1)^3} - \hdots   \\
		& = (1-b)  + \frac{b(1-b)}{2! (k-1)^2} \left(1-\frac{ (2-b)}{3 (k-1)}\right) + 
+ \frac{b(1-b)(2-b)(3-b)}{4! (k-1)^4} \left(1-\frac{ (4-b)}{5(k-1)}\right) + \hdots \\
		& \geq (1-b), \mbox{ since } k \geq 2 \geq 1 + \max\left\{ \frac{2}{3},\frac{4}{5}, \frac{6}{7}, \hdots \right\}. 
\end{align*}}
{By taking conditional expectations and recalling that $\eta_k = c \gamma^k$ where $c > 1$, we obtain the following.}
\begin{align*}
\mathbb E[{\bar v_{k+1}}\mid \mathcal H_k]&\leq (1-\alpha_k) {\bar v_k}{+\eta_kB^2-(1-\alpha_k)\eta_{k-1}B^2} +\left({1\over {2\over \gamma_k}-{2\over \eta_k}}\right) {\nu^2\over N_k}\\&
\leq (1-\alpha_k) v_k {+\eta_kB^2-(1-\alpha_k)\eta_{k-1}B^2}+ \left({c \over 2(c-1)} \right) {\gamma_k\nu^2\over N_k}.
\end{align*}
If $\gamma_k=k^{-b}$ {where } $b\in (0,1/2]$ and $N_k=\lfloor k^a\rfloor $ {where} {$a+b>1$}, by Lemma~\ref{fistabound for floor}, {we have that} $\sum_{k=1}^\infty  {\gamma_k\nu^2\over N_k}<\infty$ and the following holds for $\eta_k=ck^{-b}$, $c>1$ and $b\in (0,1/2]$:
{\begin{align*}
\eta_k-(1-\alpha_k)\eta_{k-1}&=\eta_k-{\lambda_{k-1}^2\gamma_{k-1}\over \lambda_k^2\gamma_k}\eta_{k-1}=ck^{-b}-\left(1-{1\over \lambda_k}\right){c(k-1)^{-2b}\over k^{-b}}\\ &\leq ck^{-b}-\left(1-{1\over \lambda_k}\right)ck^{-b}\leq {2c\over k^{1+b}}
\implies \sum_{k=1}^\infty (\eta_kB^2-(1-\alpha_k)\eta_{k-1}B^2)<\infty.
\end{align*}}
  {Furthermore,} from \eqref{alpha}, {it follows that} $\sum_{k=1}^\infty \alpha_k=\infty$ and 
$$\lim_{k\rightarrow \infty} \left({1\over \alpha_k}\right) \left({c \over 2(c-1)} \right)\left(\nu^2\over
k^{a+b}\right) { \ \leq } \lim_{k\rightarrow \infty} \left({c \over 2(c-1)} \right) \left(\nu^2\over {(1-b)}
k^{a+b-1}\right)=0$$ for $b\in (0,1/2]$ and {$a+b>1$}. Additionally, we have the following:
\begin{align*}
& \lim_{k\rightarrow \infty}{\eta_kB^2-(1-\alpha_k)\eta_{k-1}B^2\over \alpha_k}=\lim_{k\rightarrow \infty}{ck^{-b}B^2-c(1-\alpha_k)({k-1})^{-b}B^2\over \alpha_k}\\
&\leq \lim_{k\rightarrow \infty} {ck^{-b}B^2-c(1-\alpha_k){k}^{-b}B^2\over \alpha_k}=\lim_{k\rightarrow \infty} {cB^2\over k^b }=0,
\end{align*}
where  $\eta_kB^2-(1-\alpha_k)\eta_{k-1}B^2\geq0$ can be concluded as follows. For any $b\in (0,1/2]$, we have:
\begin{align*}
{\lambda_{k-1}^2\over \lambda_k^2}=\left(1-{1\over \lambda_k}\right)\leq {k-1\over k}\leq {(k-1)^{2b}\over k^{2b}}& \implies {\lambda_{k-1}^2\over \lambda_k^2} {k^b\over (k-1)^b}\leq {(k-1)^b\over k^b}\implies {\lambda_{k-1}^2\gamma_{k-1}\over \lambda_k^2\gamma_k}\leq {\eta_k\over \eta_{k-1}}\\
&\implies (1-\alpha_k)\leq {\eta_k\over \eta_{k-1}}\implies \eta_k-(1-\alpha_k)\eta_{k-1}\geq 0.
\end{align*}
Therefore, Lemma~\ref{almost sure} can be applied and
{$\bar v_{k} = F_{\eta_k}(x_k)-F_{\eta_k}(x^*)+\eta_kB^2 \rightarrow 0$} a.s.. {By $(1,B^2)$ smoothness of $f$,} 
${\ 0 \ \leq \ } F(x_k)-F(x^*)\leq F_{\eta_k}(x_k)-F_{\eta_k}(x^*)+\eta_k B^2$,
implying that $F(x_k)\rightarrow
F(x^*)$ a.s.  
\end{proof}
{The next proposition provides a similar a.s.  convergence for ({\bf VS-APM})
that can accommodate structured nonsmooth optimization where $f(x)$ is a smooth
merely convex function. The proof of this result is similar to Proposition
\ref{a.s. smoothing}, but $\delta_k$ in this case is defined as $\delta_k
=F(y_k)-F(x^*)$. }  
\begin{proposition}{{\bf (Almost sure convergence theory for {\bf (VS-APM)})}} 
{ Suppose Assumptions~\ref{smooth-f},~\ref{ass_error2}, and \ref{fistaass3}
hold.  Suppose $\{{y}_k\}$ defines a sequence generated by {\bf (VS-APM)}}.
Suppose {$\gamma_k=\gamma\leq 1/(2L)$ and
$N_k=\lfloor k^{a}\rfloor$ for $a>1$}. Then $\{{y}_k\}$ converges to {a} solution of \eqref{main
problem} \usv{almost surely}.   
\end{proposition}
section{Numerical \uvv{R}esults}
We now compare the performance of {\bf (mVS-APM)} and {\bf (sVS-APM)} with existing solvers on Matlab running on a 64-bit macOS 10.13.3 with Intel i7-7Y75 @1.4GHz with 16GB RAM.

{\noindent {\bf 1. {\bf mVS-APM}: Strongly convex and nonsmooth $f$.}

{\bf Example 1.} Consider the following constrained problem.
\begin{align}\label{example2 sc}
\min_{x\in [-1,1]} f(x), \mbox{where} f(x) \triangleq \mathbb E\left[{1\over 2}x^TA(\omega)x+\beta(\omega)^Tx+\lambda(\omega)\|x\|_1\right],
\end{align}
$A(\omega)=\bar A+W\in \mathbb R^{n\times n}$ and the elements of $W$
have an i.i.d. normal distribution with mean zero and {standard deviation (std) $0.1$}.
Similarly, $\beta(\omega)=\bar \beta+w\in \mathbb R^n$, where $w$ is a random
vector. {Since, tractable prox evaluations are not available for
\eqref{example2 sc}, we compute approximate gradients $\nabla_x f_{\eta}$  
using {\bf (SSG)}}. \afo{We set $N_k = \lfloor \rho^{-k}\rfloor$, where $\rho \triangleq \left(1-{1\over
{2}a\sqrt{\tilde \kappa}}\right)$ and $a = 2.01.$} Using a budget of $1e5$ and $10$ replications, we provide results in Table \ref{example_sc_tab_ex2} (L) while 
Figure \ref{nonsmooth_sc_ex2} shows the behavior of {\bf (mVS-APM)}
with different smoothing parameters $\eta$ versus {\bf (SSG)}.   When the strong convexity
modulus $\mu$ is small, {\bf mVS-APM} performs significantly better than 
{\bf (SSG)} and is far more stable. For instance, when $\eta=1$,  {\bf (mVS-APM)}
terminates with an empirical error of approximately $4.8e$-$3$ and $5.5e$-$3$ for
$\mu=1$ and $\mu=1e$-$4$ while corresponding errors for {\bf (SSG)} are
$7.8e$-$3$ to $6.3$. {As one can see, $\eta=1$ for  {\bf (mVS-APM)} seems to be
a reasonable practical choice for different problem settings. Note that in this
table, $\eta^*$ is chosen according to Lemma~\ref{char-eta} {where we note
that as $\mu \ll 1$, the benefit of utilizing $\eta^*$ is muted.} {Next, we
consider the unconstrained variant \eqref{example2 sc}, where $x\in
\mathbb R^n$. Since the subgradient is unbounded, we use
unaccelerated method {\bf (mVS-PM)}. In Table \ref{example_sc_tab_ex2}
(R), the behavior of {\bf (mVS-PM)} is compared with {\bf (SSG)} for different choices of $\mu$. As suggested after Theorem \ref{th-rate-itercomp-sc-sgd-general}, we
set $\eta=\tfrac{1}{\mu}+1e$-$3>\tfrac{1}{\mu}$. }
}
\begin{figure}[htb]
\centering
	\includegraphics[scale=0.12]{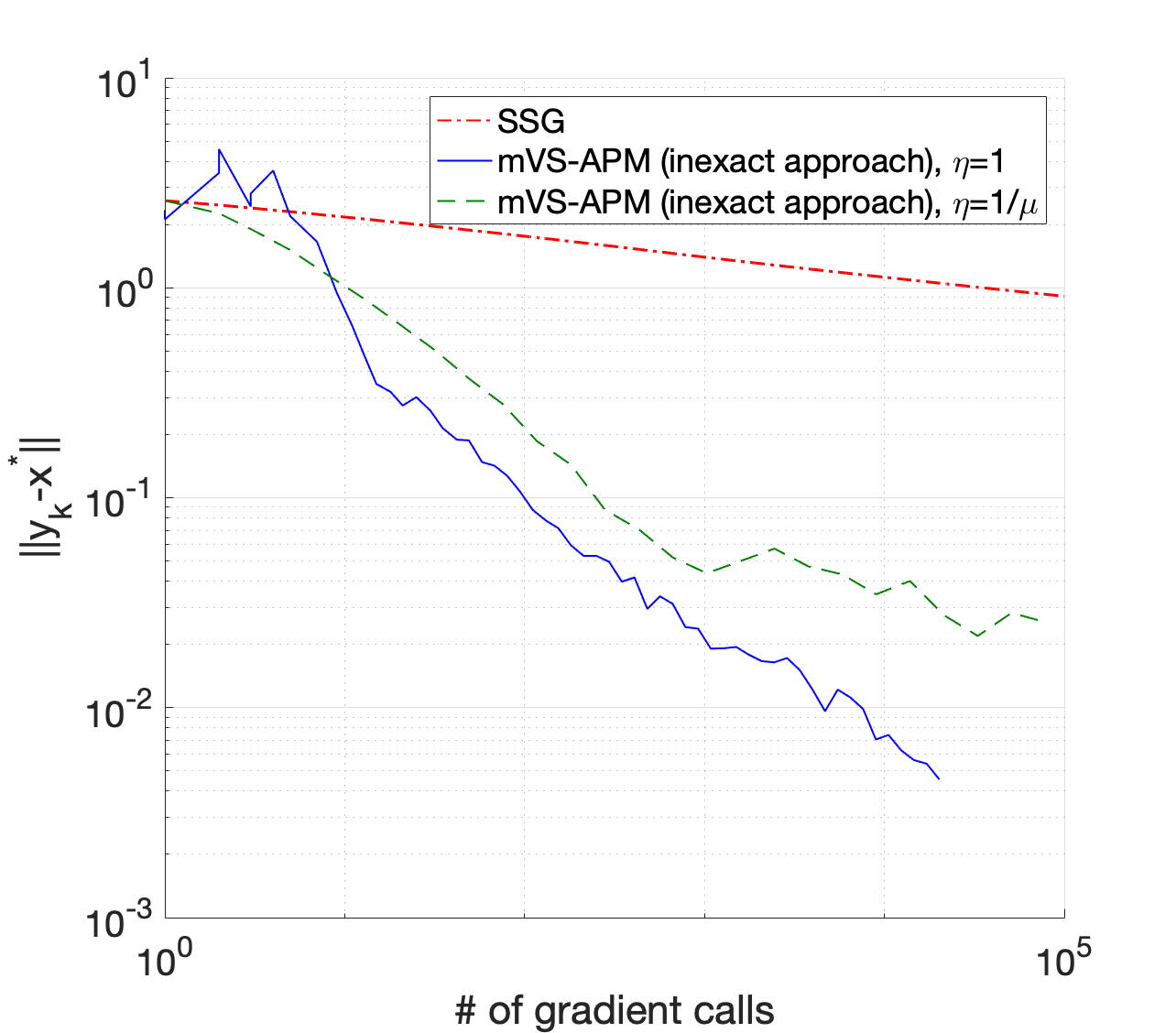}
	\captionof{figure}{Example~1: ({\bf mVS-APM}) vs SSG for $\mu=0.1$}
	\label{nonsmooth_sc_ex2}
\end{figure}
\begin{table}[htbp]
	\centering
\tiny

\begin{tabular}{|c|c||c|c|c|c|} \hline 
&{\bf SSG}&\multicolumn{4}{|c|}{$\|y_k-x^*\|$ for \bf mVS-APM}\\ \hline \hline
$\mu$&$\|y_k-x^*\|$&$\eta=\eta^*$&$\eta=0.1$&$\eta=1$&$\eta=10$\\ \hline
1&7.8609e-4&2.8078e-1&2.2150e-2&4.7893e-3&1.9443e-2\\ \hline
1e-1&9.9114e-1&3.3207e-3&3.7247e-2&5.8973e-3&1.8865e-2\\ \hline
1e-2&3.0611&3.7218e-2&8.3083e-2&7.3432e-3&3.6886e-2\\ \hline
1e-3&4.0682&1.3893&1.7692e-1&4.7901e-3&5.2147e-2\\ \hline
1e-4&6.3783&2.7269&4.7065e-1&5.5248e-3&6.3872e-2\\ \hline

 	\end{tabular}
\begin{tabular}{|c|c||c|} \hline 
&{\bf SSG}& {\bf mVS-PM}\\ \hline \hline
$\mu$&$\|y_k-x^*\|$&$\|y_k-x^*\|$\\ \hline
1&2.0847e-1&3.0971e-2\\ \hline
1e-1&2.4283&9.5149e-2\\ \hline
1e-2&4.2409&1.5115e-1\\ \hline
1e-3&4.4784&1.8033e-1\\ \hline
1e-4&4.5028&1.7261e-1\\ \hline

 	\end{tabular}
	\caption{Example 1: {\bf mVS-APM}vs SSG (L),  {\bf mVS-PM} vs SSG (R)}
\label{example_sc_tab_ex2}
\end{table}
{In Table \ref{sc_tab_ex1}, we compare {\bf (mVS-APM)} with {\bf (SSG)} for
different choices of standard deviation of noise and dimension ($n$). In Table
\ref{sc_tab_ex1} (L), we set $\mu=0.1$ and $n=20$ while in Table \ref{sc_tab_ex1}
(R), we set $\mu=0.1$ and std. dev. is $0.1$. We {run both schemes with} 
total {budget in subgradient evaluations of} $1$e$5$ and 10 replications and observe that {\bf (mVS-APM)}
outperforms {\bf (SSG)} .  }\begin{table}[ht]
	\centering
\tiny

\begin{tabular}{|c|c|c||c|c|c|} \hline 
&\multicolumn{2}{|c||}{\bf SSG}&\multicolumn{3}{|c|}{\bf mVS-APM}\\ \hline \hline
std.&$\|y_k-x^*\|$&time&$\eta$&$\|y_k-x^*\|$&time\\ \hline
1e+1&1.6691&5.8269&1&5.6007e-1&2.9858\\ 
1&9.4759e-1&5.9375&1&5.1574e-2&2.9925\\ 
1e-1&9.1148e-1&5.9096&1&5.8973e-3&3.8961\\ 
1e-2&9.1285e-1&5.9444&1&5.7294e-4&3.0362\\ 
\hline
\end{tabular}
\begin{tabular}{|c|c|c||c|c|c|} \hline 
&\multicolumn{2}{|c||}{\bf SSG}&\multicolumn{3}{|c|}{\bf mVS-APM}\\ \hline \hline
n&$\|y_k-x^*\|$&time&$\eta$&$\|y_k-x^*\|$&time\\ \hline
20&9.1148e-1&5.9096&1&5.8973e-3&3.8961\\
30&1.5326&6.117&1&5.9034e-3&3.2213\\ 
40&8.5934e-1&6.2494&1&6.0096e-3&3.6658\\ 
50&3.6236&6.4209&1&6.3496e-3&3.3903\\ 
\hline
 	\end{tabular}
	\caption{Example 1: Comparing {\bf mVS-APM} vs SSG: different std (L), different n (R)}
\label{sc_tab_ex1}
\end{table}

{\bf Example 2.} We revisit this comparison using a stochastic utility problem.
 \begin{align*}
\min_{\|x\|\leq 1} \mathbb E\left[\phi\left(\sum_{i=1}^n \left({i\over n}+\omega_i\right)x_i\right)\right]+{\mu\over 2}\|x\|^2, \end{align*}
where $\phi(t) \triangleq \max_{1\leq j\leq m}(v_i+s_it)$, 
$\omega_i$ are iid normal random variables with mean zero and variance one and
$v_i, s_i \in (0,1)$. Table \ref{sc_tab_ex2} shows similar behavior as in
Example 1. {In Table \ref{sc_tab_ex3}, we compare {\bf(mVS-APM)} with {\bf (SSG)} for different choices of std. dev. and dimension ($n$). In
Table \ref{sc_tab_ex3} (L), we set $\mu=0.1$ while $n=20$ and in Table
\ref{sc_tab_ex3} (R), we set $\mu=0.1$ and std. dev. is $ 1$. Similar
to {\bf Example 1}, {\bf (mVS-APM)} outperforms {\bf (SSG)} in all cases.  }

\begin{table}[ht]
	\centering
\tiny

\begin{tabular}{|c|c|c||c|c|c|} \hline 
&\multicolumn{2}{|c||}{\bf SSG}&\multicolumn{3}{|c|}{\bf mVS-APM}\\ \hline \hline
$\mu$&$\|y_k-x^*\|$&time&$\eta$&$\|y_k-x^*\|$&time\\ \hline
1&4.4908e-3&4.3883&$1/\mu=1$&5.8314e-3&1.5191\\ \hline
1e-1&2.7134e-1&3.8794&1&1.0102e-2&1.1964\\ 
1e-2&8.7266e-1&3.9742&1&1.8236e-2&1.2065\\ 
1e-3&9.8723e-1&4.0129&1&3.8619e-2&1.1510\\ 
1e-4&9.9872e-1&4.0684&1&7.1652e-2&1.1490\\ 
\hline
 	\end{tabular}
	\caption{Example 2: Comparing {\bf (mVS-APM)} vs {\bf (SSG)}}
\label{sc_tab_ex2}
\vspace{-0.2in}
\end{table}
\begin{table}[htb]
\tiny
\centering
\begin{tabular}{|c|c|c||c|c|c|} \hline 
&\multicolumn{2}{|c||}{\bf SSG}&\multicolumn{3}{|c|}{\bf mVS-APM}\\ \hline \hline
std.&$\|y_k-x^*\|$&time&$\eta$&$\|y_k-x^*\|$&time\\ \hline
1e+1&9.8253e-1&3.8733&1&9.6709e-1&1.1661\\ 
1&2.7134e-1&3.8794&1&1.0102e-2&1.1964\\ 
1e-1&2.1394e-1&3.9304&1&8.6589e-3&1.1083\\ 
1e-2&2.1813e-1&3.9134&1&1.1027e-1&1.1270\\ 
\hline
\end{tabular}
\begin{tabular}{|c|c|c||c|c|c|} \hline 
&\multicolumn{2}{|c||}{SSG}&\multicolumn{3}{|c|}{mVS-APM}\\ \hline \hline
n&$\|y_k-x^*\|$&time&$\eta$&$\|y_k-x^*\|$&time\\ \hline
20&2.7134e-1&3.8794&1&1.0102e-2&1.1964\\ 
30&3.5948e-1&4.0277&1&1.2010e-2&1.2594\\ 
40&5.3537e-1&4.0418&1&7.4431e-3&1.3467\\ 
50&2.6880e-1&4.1198&1&8.2670e-3&1.3452\\ 
\hline
 	\end{tabular}
	\caption{Example 2: Comparing {\bf mVS-APM} vs SSG: different std (L), different n (R)}
\label{sc_tab_ex3}
\end{table}

\noindent {\bf 2. ({\bf sVS-APM}). Convex and smoothable $f$.}

{\bf Example 4.} In this setting, we  compare the performance of ({\bf sVS-APM}) for merely convex problems on Example 2 with $\mu = 0$.  The $\delta$-smoothed approximation of $\phi(t)$ provided by~\cite{beck12smoothing} is given by $\phi_\delta(t)=\delta \log\left(\sum_{i=1}^m e^{(v_i+s_it)/\delta}\right)$.  In Table \ref{determ_smooth}, we generate $20$ replications for ({\bf sVS-APM})
with fixed and diminishing smoothing sequences  with $\eta_k=\delta_k/2$, $N_k=\lfloor
k^{3.001}\rfloor $, and sampling budget is $1$e$6$. In Figure
\ref{moreau_iter}, we compare trajectories for ({\bf sVS-APM}) with those for constant smoothing for $n = 200$.
\begin{table}[ht]
	\centering

\tiny
	\begin{tabular}{|c|c|c|c|c|c|} \hline 
		&  &&sVS-APM&&Fixed smooth.\\ \hline
	$n$& $m$ & $\delta_k$ &$\mathbb E[f(y_k)-f^*]$ &$\delta$ & $\mathbb E[f(y_k)-f^*]$ \\ \hline \hline 
20&10&$1/k$&1.832e-4&$1/K$&3.455e-3\\ 
	&&$1/(2k)$&3.014e-3&$1/(2K)$&2.157e-2\\ 
	&&$1/(3k)$&1.269e-2&$1/(3K)$&6.079e-2\\ \hline 
		100&25&$1/k$&1.944e-3&$1/K$&3.126e-2\\ 
	&&$1/2k$&1.181e-2&$1/2K$&5.130e-2\\
	&&$1/3k$&2.411e-2&$1/3K$&5.817e-2\\ \hline
	200&10&$1/k$&1.067e-4&$1/K$&4.695e-3\\ 
	&&$1/2k$&5.173e-3&$1/2K$&3.957e-2\\
	&&$1/3k$&1.594e-2&$1/3K$&6.929e-2\\ \hline
 	\end{tabular}
\caption{Example 4: Comparing {\bf (sVS-APM)} with fixed smoothing}
\label{determ_smooth}
\end{table}
\begin{figure}[htb]
\begin{minipage}[b]{0.32\linewidth}
\centering
	\includegraphics[scale=0.11]{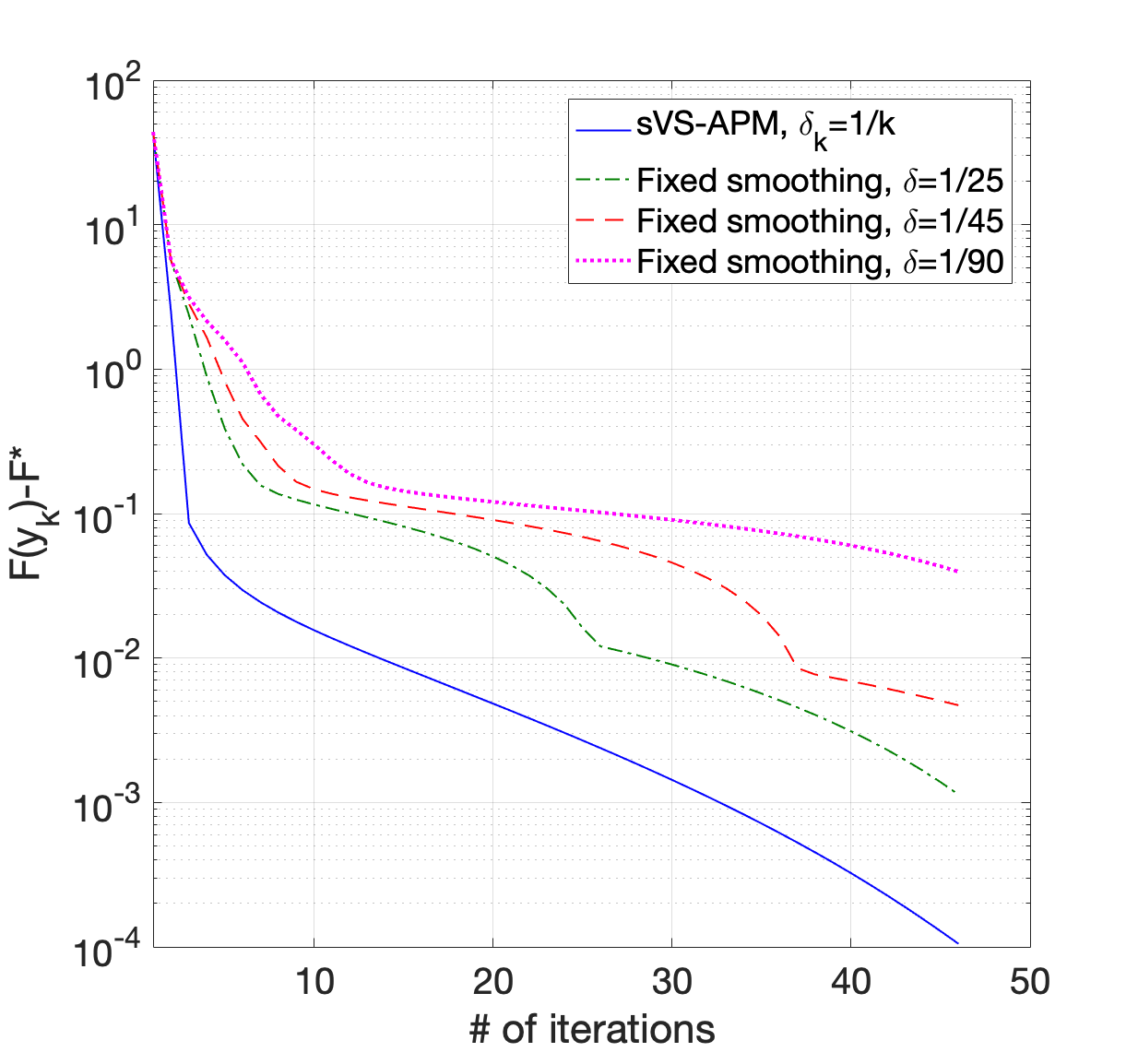}
	\captionof{figure}{Example 4: ({\bf sVS-APM}) vs fixed smoothing; $n=200$}
	\label{moreau_iter}
	\end{minipage}
\begin{minipage}[b]{0.32\linewidth}
	\includegraphics[scale=0.28]{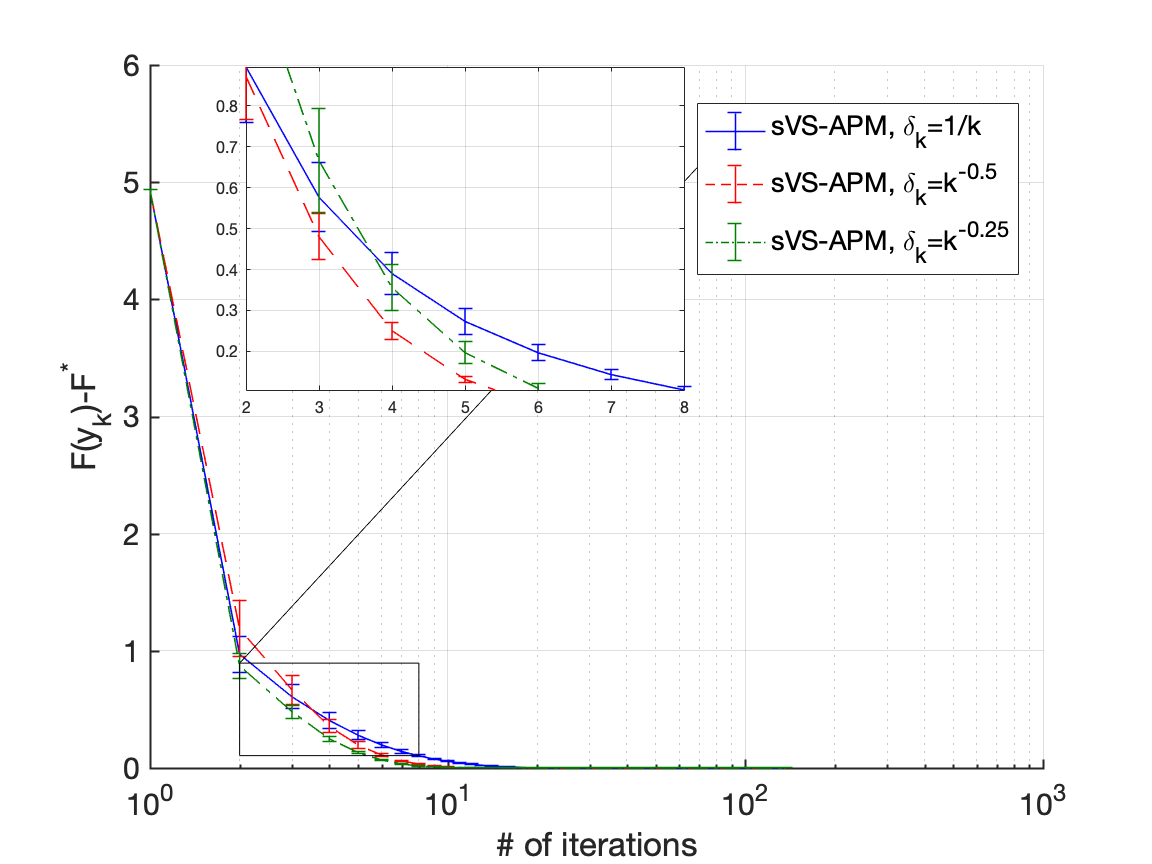}
	\caption{a.s. convergence for {\bf (sVS-APM)},  $N_k=\lfloor k^{3.001}\rfloor$, $\nu^2 =5$.}
	\label{almost sure1}
\end{minipage}
\begin{minipage}[b]{0.32\textwidth}
	\includegraphics[scale=0.28]{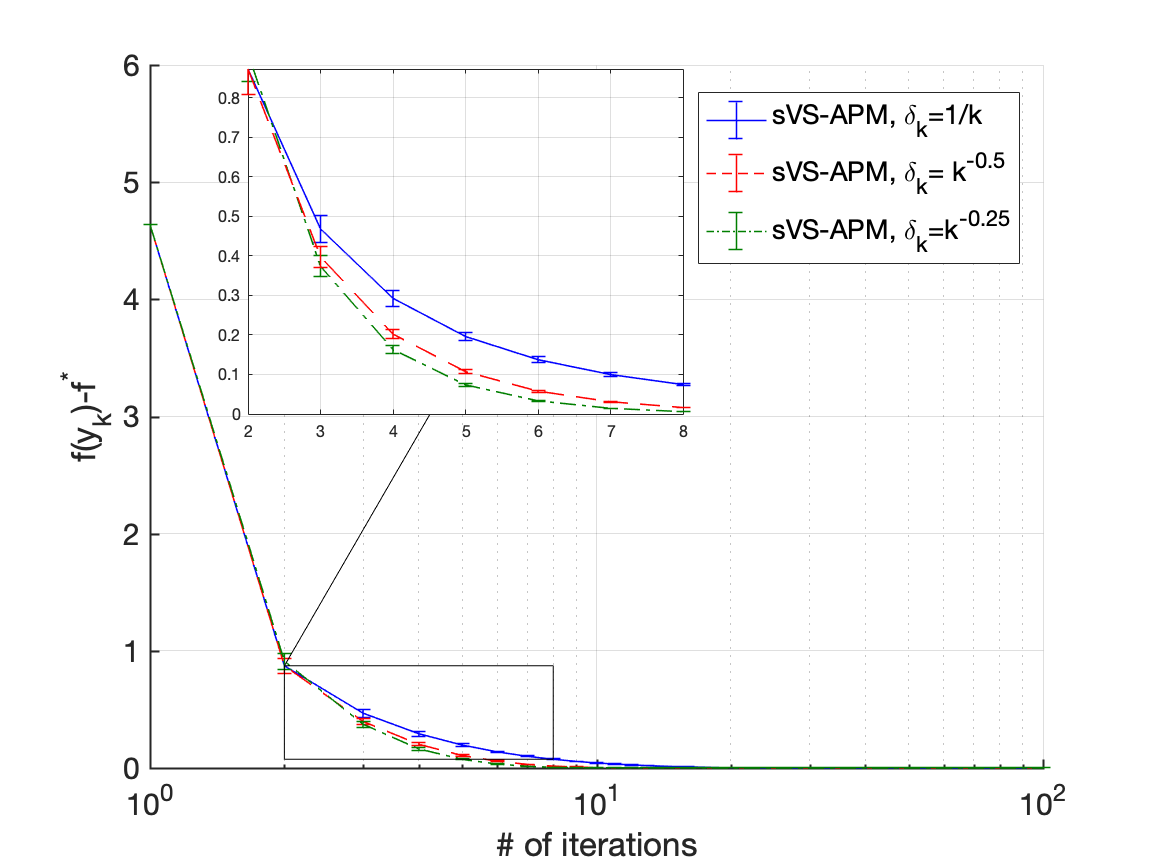}
	\caption{a.s. convergence for {\bf (sVS-APM)},  $N_k=\lfloor k^{3.001}\rfloor$, $\nu^2 =2$.}
	\label{almost sure}
\end{minipage}

\end{figure}

\noindent {\bf Key observations.} The empirical behavior of ({\bf
sVS-APM}) appears to be better on this test problem. One rationale for this may
be drawn from noting that ({\bf sVS-APM}) allows for  larger steplengths early
(since $\eta_k \leq \delta_k$) on while in fixed smoothing technique, $\eta_k \leq
\delta_k$ (where $\delta_k$ may be quite small). This can be seen in the trajectories
where early progress by the iterative smoothing scheme can be observed. A larger $\delta_k$ allows for larger steplengths but leads to a coarser approximation of the original problem while smaller $\delta_k$ leads to poorer progress but better approximations (See Table~\ref{determ_smooth} and Figure~\ref{moreau_iter}). 

\noindent {\bf 4. a.s. convergence. }  Next, we
implemented {\bf sVS-APM} on the stochastic utility problem
 with $n = 20$ and $m = 10$ for different choices of the smoothing sequences. Specifically, we
allow $\delta_k$ to be $\delta_k\in\{1/k,1/\sqrt{k},1/k^{0.25}\}$ (where $\delta_k =
1/k$ is required for convergence in mean and $\delta_k = 1/k^b$ with $b \in
(0,1/2]$ for a.s. convergence). We employ $N_k=\lfloor
k^{3.001}\rfloor$. For each experiment, the mean of 20 replications and their
$95\%$ confidence intervals are plotted in Figure~\ref{almost sure1} and~\ref{almost sure}. It can be
seen that when $\delta_k \to 0$ at a slower rate as mandated by the requirement of
the a.s. convergence result, the confidence bands are tighter, becoming 
more apparent in Figure~\ref{almost sure1}  where the variance is $5$. Furthermore, our numerical studies have revealed that even for less aggressive choices of $N_k$ such as when $N_k = k^a$ and $a >1$, the trajectories show the desired behavior in accordance with Prop.~\ref{a.s. smoothing}. 

\section{Concluding \uvv{R}emarks} Drawing motivation from the generally poor
behavior of {\bf (SSG)} schemes on general (rather than structured) nonsmooth
stochastic convex optimization problems, we develop two sets of accelerated
proximal variance-reduced schemes, both of which rely on a variable sample-size
accelerated proximal method {\bf (VS-APM)} for smooth convex problems. {In
    nonsmooth strongly convex regimes, we \usv{present three sets of schemes}, each of
which produces linearly convergent sequences and is characterized by an overall complexity in
subgradients (or proximal evaluations in the third case) that is optimal (or
near-optimal). First, in compact domains, {we propose} {\bf (mVS-APM)},
{an avenue that} requires applying ({\bf VS-APM}) on the Moreau envelope of
$F(x)$ where increasingly exact gradients are computed via an inner {\bf (SSG)}
scheme. Second, in unbounded domains, we apply an unaccelerated variable
sample-size proximal method {\bf (VS-PM)} which also relies on {\bf (SSG)} for
approximating gradients to increasing accuracy. 
When $\uss{\f}(\uss{\cdot},\omega)$ is smoothable and convex, our smoothed ({\bf VS-APM}) scheme (or
{\bf sVS-APM}) admits optimal rate and oracle complexity. Our findings, when
specialized to the smooth and convex $f$, provide an optimal accelerated rate of
$\mathcal{O}(1/K^2)$ with optimal oracle complexity matching findings by
~\cite{ghadimi2016accelerated} and ~\cite{jofre2017variance}. When $f$ is
deterministic, our rate matches that obtained by ~\cite{nesterov2005smooth} but
does so while providing asymptotically convergent schemes. Preliminary numerics
suggest that the schemes compare well with existing techniques both in terms of
complexity as well as in terms of sensitivity to problem parameters.   

\bibliographystyle{siam} 
\bibliography{demobib,wsc11-v02}
\section{Appendix}
\begin{lemma}\label{fistabound for floor}
For any real number $y\geq 1$ we have that:
$\lfloor y \rfloor\geq \left \lceil {1\over 2} y\right\rceil.$
\end{lemma}\vspace{-0.2cm}
\begin{proof} Let $T=\left \lfloor y \right\rfloor$. If $T$ is an even number. Then, we have $\left\lceil{1\over 2}y \right\rceil= \left\lceil{1\over 2} (T+\epsilon) \right\rceil={T\over 2}+1$.
where $\epsilon \in (0,1) $. Since $T\geq{T\over 2}+1$, so $\lfloor y \rfloor\geq \left \lceil {1\over 2} y\right\rceil.$ If $T$ is an odd number, we have
$\left \lceil{1\over 2} y \right\rceil= \left\lceil{T-1\over 2} +{\epsilon+1\over 2} \right\rceil={T-1\over 2}+1={T+1\over 2}$.
Again since $T\geq{T+1\over 2}$, we have that $\lfloor y \rfloor\geq \left \lceil {1\over 2} y\right\rceil.$
\end{proof}

\begin{lemma}\label{fistalem1}
Given a symmetric positive definite matrix $Q$, then, we have the following for any $\nu_1,\nu_2,\nu_3$:
$(\nu_2-\nu_1)^TQ(\nu_3-\nu_1)={1\over 2}(\|\nu_2-\nu_1\|^2_Q+\|\nu_3-\nu_1\|^2_Q-\|\nu_2-\nu_3\|^2_Q), \mbox{ where } \|\nu\|_Q \triangleq \sqrt{\nu^TQ\nu}.$
\end{lemma}
\begin{lemma}\label{lemma_ac-vssa}
Suppose Assumptions \ref{fistaass2} and \ref{ass_error2}(i)   hold.  Furthermore, $\gamma_k = 1/(2L)$ for all $k$.
	 If $h(x_k) \triangleq {2L}(x_k-{y_{k+1}})$, 
$\nonumber\quad F(x)-{\mu \over 4}\|x-x_k\|^2  \geq F(y_{k+1})+{1\over 4L}\|h(x_k)\|^2+ h(x_k)^T(x-x_k)- {\left({2 \over L}+{1\over \mu}\right) \| \bar{w}_{k,N_k}\|^2}.$
\end{lemma}
\begin{proof} 
Since $y_{k+1} \triangleq {\mbox{arg}}\min_{x}\ \ {1\over 2L}g(x)+{1 \over 2}\left\|x-\left[x_k-{1\over
	2L}(\nabla_x f(x_k)+\bar{w}_{k,N_k})\right]\right\|^2$, we have that 
\begin{align*}
y_{k+1} &=  \mbox{arg}\min_{{x }}\ \ {1\over 2L}g(x)+{1 \over 2}\Big[\|x-x_k\|^2
+ {1 \over L} (x-x_k)^T(\nabla_x f(x_k)+\bar{w}_{k,N_k})+ {1\over
			4L^2} \|\nabla_x f(x_k) + \bar{w}_{k,N_k}\|^2 \Big] \\
		& = \mbox{arg}\min_{{x }}\ \ g(x)+\Big[{L}\|x-x_k\|^2+ f(x_k)
		+ (x-x_k)^T(\nabla_x f(x_k)+\bar{w}_{k,N_k})\Big].
\end{align*}
Let ${\psi_k}(x)\triangleq f(x_k)+\nabla_x f(x_k)^T(x-x_k) +{L}\|x-x_k\|^2+\bar{w}_{k,N_k}^T(x-x_k)$, implying that 
\begin{align}\label{problem y_{k+1}} y_{k+1}  =\mbox{arg}\min_{x } \ \ {\psi_k}(x)+g(x).\end{align}
 Then $\nabla_x {\psi_k}(x)$
	may be expressed as $\nabla_x {\psi_k}(x)=\nabla_x
	f(x_k)+{2L}(x-x_k)+\bar{w}_{k,N_k}$
By the optimality condition of \eqref{problem y_{k+1}}, we have $0\in \partial g(y_{k+1})+\nabla {\psi_k}(y_{k+1})$. Hence, by convexity of function $g(x)$ we obtain
\begin{align}\label{opt cond}
 g(x)\geq g(y_{k+1})-\nabla{\psi_k}(y_{k+1})^T(x-y_{k+1}) \implies \nabla{\psi_k}(y_{k+1})^T(x-y_{k+1})\geq g(y_{k+1})-g(x).
\end{align}
Consequently, by using the definition of ${\psi_k}(x)$ and $h(x)$ we have that
\begin{align} \label{eq-acc1}
\nabla_x f(x_k)^T(x-y_{k+1}) \geq g(y_{k+1})-g(x)+\left(h(x_k)-\bar
	w_{k,N_k}\right)^T(x-y_{k+1}), \quad \ \forall x. \end{align}
{Since $f$ is a $\mu$-strongly convex function},
\begin{align*}
 & \quad f(x)-{\mu \over 2}\|x-x_k\|^2   \geq f(x_k)+\nabla_x
f(x_k)^T(x-x_k)  =f(x_k)+\nabla_x f(x_k)^T(x-x_k+y_{k+1}-y_{k+1})\\
\quad & \overset{\mbox{\tiny (From \eqref{eq-acc1})}}{\geq} f(x_k)+\nabla_x f(x_k)^T(y_{k+1}-x_k)+ (h(x_k)-\bar{w}_{k,N_k})^T(x-y_{k+1})+g(y_{k+1})-g(x)  \\
&= {\psi_k}(y_{k+1})-{L}\|y_{k+1}-x_k\|^2-\bar{w}_{k,N_k}^T(y_{k+1}-x_k) 
+ (h(x_k)-\bar{w}_{k,N_k})^T(x-y_{k+1})+g(y_{k+1})-g(x) \\
&= {\psi_k}(y_{k+1})-{L}\|y_{k+1}-x_k\|^2+\bar{w}_{k,N_k}^T(x_k-x) 
+ h(x_k)^T(x-y_{k+1})+g(y_{k+1})-g(x). 
\end{align*}
From the definition of $h(x_k)$, ${L}\|y_{k+1}-x_k\|^2={1\over {4 L}}\|h(x_k)\|^2$ and inequality \eqref{opt cond}, we have the following: 
\begin{align}
\nonumber & \quad F(x)-{\mu \over 2}\|x-x_k\|^2  \nonumber  \geq {\psi_k}(y_{k+1})-{1\over 4L}\|h(x_k)\|^2+h(x_k)^T(x-y_{k+1})+\bar{w}_{k,N_k}^T(x_k-x) +g(y_{k+1})\\ \nonumber
&= {\psi_k}(y_{k+1})-{1\over 4L}\|h(x_k)\|^2+h(x_k)^T(x-y_{k+1}+x_k-x_k)+\bar{w}_{k,N_k}^T(x_k-x)+g(y_{k+1}) \\ 
&= {\psi_k}(y_{k+1})+{1 \over 4L}\|h(x_k)\|^2+h(x_k)^T(x-x_k)+
\bar{w}_{k,N_k}^T(x_k-x)+g(y_{k+1}),\label{lemma-b1-1} \\
&{\geq {\psi_k}(y_{k+1})+{1 \over 4L}\|h(x_k)\|^2+h(x_k)^T(x-x_k)-
{1\over \mu}\|\bar{w}_{k,N_k}\|-{\mu\over 4}\|x_k-x\|^2+g(y_{k+1})}\label{lemma-b1}
\end{align}
where \eqref{lemma-b1-1} follows from the definition of $h(x_k)$ {and \eqref{lemma-b1} follows by using the fact that $a^Tb\geq -{1\over 2\alpha}\|a\|^2-{\alpha\over 2}\|b\|^2$ {with $\alpha = 2$}.}
From $L$-smoothness of $f$,
{\begin{align} 
\notag {\psi_k}(y_{k+1}) &\nonumber = f(x_k) + \nabla_x f(x_k)^T(y_{k+1}-x_k) + {L} \|x_k-y_{k+1}\|^2 + \bar{w}_{k,N_k}^T(y_{k+1}-x_k) \\ 
& \geq f(y_{k+1}) + \bar{w}_{k,N_k}^T(y_{k+1}-x_k){+}{L \over 2} \|x_k - y_{k+1}\|^2  \geq f(y_{k+1}) - {{2 \over L} \| \bar{w}_{k,N_k}\|^2}, 
\label{lemma-b2}
\end{align}}
{where \eqref{lemma-b2} follows from $2a^Tb + \|a\|^2 \geq - \|b\|^2.$}
By substituting \eqref{lemma-b2} in \eqref{lemma-b1}, the
result follows.
\end{proof} 
It is worth emphasizing that in the proof of Lemma~\ref{lemma_ac-vssa}, we employ a simple bound to ensure that the term $\bar{w}_{k,N_k}^T(y_{k+1}-x_k)$ does not appear in the final bound. Instead, the term $\|\bar{w}_{k,N_k}\|^2$ emerges and this allows for deriving the optimal (rather than sub-optimal) oracle complexity.   {Next, we define a set of parameter sequences that form the basis for updating the iterates.}

\begin{definition}[{\bf Defn. of $v_k, \alpha_k, \tau_k$}] \label{def-ac-param}
Given $v_0$, $\tau_0$, sequences
	$\{v_k,\tau_k, \alpha_k\}$ are defined as follows:
\begin{align}\label{v_k update}
v_{k+1} & :={1\over \tau_{k+1}}\left[ (1-\alpha_k)\tau_k v_k+{1\over 2}\alpha_k \mu x_k 
	 -\alpha_k(h(x_k) \right], \\
\label{upd-alpha} \alpha_k & \ \mbox{\rm solves }  \ (1-\alpha_k) \tau_k +{1\over 2}\alpha_k \mu = {2}\alpha_k^2 L, \\
\label{upd-tau}
\tau_{k+1} & : =(1-\alpha_k)\tau_k+{1\over 2}\alpha_k \mu. 
\end{align}
\end{definition}
{We employ this set of parameters in showing that the update rule (3) in Algorithm~\ref{nonsmooth sc scheme}  can be recast using the parameters $\tau_k, \alpha_k$, and $v_k$. This observation is crucial as we analyze the update.} 
\begin{lemma}[{\bf Equivalence of Update rules}]\label{lemma_b3}
Suppose Assumptions~\ref{fistaass2} and \ref{ass_error2}(i)   hold.  Suppose the sequences $\{v_k\}, \{\alpha_k\}$, and $\{\tau_k\}$ are
prescribed by Definition~\ref{def-ac-param}. Consider the sequence
$\{x_k\}$ generated by the algorithm. Then the following hold:
\begin{align*}
\text{{\it (i)}}   \left[x_{k+1}:=y_{k+1}+\tfrac{\alpha_{K+1}\tau_{k+1}(1-\alpha_k)}{\tau_{k+2}+ \alpha_{k+1}\tau_{k+1}}(y_{k+1}-y_{k}) \right] \equiv 
\left[x_{k+1}:=\tfrac{1}{\tau_{k+1}+ {1\over 2}\alpha_{k+1}\mu}(\alpha_{k+1}\tau_{k+1}v_{k+1}+\tau_{k+2}y_{k+1})\right].
\end{align*}
\noindent (ii)  Suppose $\alpha_k={1\over \lambda_k}$ for all $k$. Then the
update rule (1b) in Algorithm~\ref{nonsmooth sc scheme} with 
$ \sigma_k \triangleq \frac{(\lambda_k-1)\left(1-\frac{\lambda_{k+1}}{{4}\kappa}\right)}{\left(1-\frac{{1}}{{4}\kappa}\right)\lambda_{k+1}}$ for all
$k$ is equivalent to the following:
\begin{align*}
&\left[x_{k+1}:=y_{k+1}+\sigma_k(y_{k+1}-y_k)\right] \equiv \left[
x_{k+1}:={1\over
\tau_{k+1}+{1\over 2}\alpha_{k+1}\mu}(\alpha_{k+1}\tau_{k+1}v_{k+1}+{\tau}_{k+2}y_{k}) \right].
\end{align*}
\end{lemma}
\begin{proof} 
{\bf (i).} The update rule on
the right in (i) can be recast as follows: 
\begin{align}\label{v_k}
x_k  &={1\over \tau_k+\alpha_k \mu}(\alpha_k \tau_k v_k+\tau_{k+1}y_k)  
 {\iff} v_k  = \frac{(\tau_k+{1\over 2}\alpha_k\mu)x_k-\tau_{k+1}y_k}{\alpha_k \tau_k}.
\end{align}
Now by substituting the expression for $v_k$ from \eqref{v_k} in \eqref{v_k update} and {recalling} that $\tau_{k+1}=(1-\alpha_k)\tau_k+{1\over 2}\alpha_k \mu={2}L \alpha_k^2$ and ${h(x_k) = 2L(x_k - y_{k+1})}$, we obtain {the following sequence of equalities.}
\begin{align}
v_{k+1} &  \notag= {1\over \tau_{k+1}} \Big[ (1-\alpha_k)\tau_k v_k+{1\over 2}\alpha_k
\mu x_k-\alpha_k(h(x_k) \Big] \\
		& =   \notag{1\over \tau_{k+1}} \Big[ (1-\alpha_k)\tau_k \frac{(\tau_k+{1\over 2}\alpha_k\mu)x_k-\tau_{k+1}y_k}{\alpha_k \tau_k}+{1\over 2}\alpha_k
\mu x_k-\alpha_k(h(x_k) \Big] \\
&= \notag\frac{(1-\alpha_k)\tau_k+{1\over 2}\alpha_k \mu-{1\over 2}\alpha_k^2\mu}{\tau_{k+1}
	\alpha_k}x_k-\frac{1-\alpha_k}{\alpha_k}y_k+{\alpha_k \mu \over{2}
		\tau_{k+1}}x_k-{\alpha_k\over \tau_{k+1}}(h(x_k))\\
& = \notag \frac{\tau_{k+1}-{1\over 2}\alpha_k^2\mu}{\tau_{k+1}\alpha_k}x_k-\frac{1-\alpha_k}{\alpha_k}y_k+{\alpha_k
	\mu \over {2}\tau_{k+1}}x_k-{\alpha_k\over \tau_{k+1}}h(x_k)\\ 
& = y_k+{1\over \alpha_k}(x_k-y_k)-{\alpha_k\over {2}L\alpha_k^2}({2}L(x_k-y_{k+1}))
= y_k+{1\over \alpha_k}(y_{k+1}-y_k). \label{upd-vk}
\end{align}
We now show that the update rule for $x_{k+1}$  on the left is
equivalent to that on the right in (i).
\begin{align*}
& x_{k+1}={1\over
\tau_{k+1}+{1\over 2}\alpha_{k+1}\mu}(\alpha_{k+1}\tau_{k+1}v_{k+1}+\tau_{k+2}y_{k+1})\\
&\overset{\eqref{upd-vk}}{=} {1\over
\tau_{k+1}+{1\over 2}\alpha_{k+1}\mu}(\alpha_{k+1}\tau_{k+1}y_k+{\alpha_{k+1}\tau_{k+1}\over
	\alpha_k}(y_{k+1}-y_k)+\tau_{k+2}y_{k+1})\\& =
\left(\frac{\tau_{k+2}+\alpha_{k+1}\tau_{k+1}}{\tau_{k+1}+{1\over 2}\alpha_{k+1}\mu)}\right)y_{k+1}+\left({1\over
\alpha_k}-1\right)\left({ {\alpha_{k+1}}\tau_{k+1}\over
	\tau_{k+1}+{1\over 2}\alpha_{k+1}\mu}\right)(y_{k+1}-y_k)\\ 
& =y_{k+1}+\left({1\over
\alpha_k}-1\right)\Big({ {\alpha_{k+1}}\tau_{k+1}\over
	\tau_{k+1}+{1\over 2}\alpha_{k+1}\mu}\Big)(y_{k+1}-y_k)  \\&= {
y_{k+1}+\frac{\alpha_{k+1}\tau_{k+1}(1-\alpha_k)}{\alpha_k(\tau_{k+1}+{1\over 2}\alpha_{k+1}\mu)}(y_{k+1}-y_k)}= {y_{k+1}+\frac{\alpha_{k+1}\tau_{k+1}(1-\alpha_k)}{\alpha_k(\tau_{k+2}+\alpha_{k+1}\tau_{k+1})}(y_{k+1}-y_k), }
\end{align*}
since $\tau_{k+1}=(1-\alpha_k)\tau_k+{1\over 2}\alpha_k \mu$.
	
	\noindent {\bf (ii).} By choosing
		$\tau_{k+1} = {2}\alpha_k^2L $ for $k \geq 0$, satisfying
			\eqref{upd-alpha} and \eqref{upd-tau}, 
\begin{align}\label{rule by alpha}
x_{k+1}&\nonumber=y_{k+1}+\frac{\alpha_{k+1}\tau_{k+1}(1-\alpha_k)}{\alpha_k({\tau}_{k+2}+\alpha_{k+1}\tau_{k+1})}(y_{k+1}-y_k)
 {=y_{k+1}+ \frac{\alpha_{k+1}\alpha_k(1-\alpha_k)}{\alpha_{k+1}^2+\alpha_{k+1}\alpha_k^2}(y_{k+1}-y_k)}\\&
 {=y_{k+1}+ \frac{\alpha_k(1-\alpha_k)}{\alpha_{k+1}+\alpha_k^2}(y_{k+1}-y_k) }.
\end{align}
Now by choosing $\alpha_k={1\over \lambda_k}$, we have the following:
\begin{align}\label{rule by alpha 2}
\frac{\alpha_k(1-\alpha_k)}{\alpha_k^2+\alpha_{k+1}}=\frac{{1\over \lambda_k}(1-{1\over \lambda_k})}{\left({1\over \lambda_k}\right)^2+{1\over \lambda_{k+1}}}=\frac{\lambda_{k+1}(\lambda_k-1)}{\lambda_{k+1}+\lambda_k^2}.
\end{align}
From the update rule for $\lambda_k$, we can obtain:
\begin{align}\label{rule by alpha 3}
&\lambda_{k+1}=\frac{1- \frac{\lambda_k^2}{{4}\kappa}+\sqrt{\left(1-\frac{\lambda_k^2}{{4}\kappa}\right)^2+4\lambda_k^2}}{2}\implies
\lambda_k^2=\frac{\lambda_{k+1}(\lambda_{k+1}-1)}{1-\frac{\lambda_{k+1}}{{4}\kappa}}.
\end{align} 
By substituting \eqref{rule by alpha 3} in \eqref{rule by alpha 2} we obtain
$\frac{\alpha_k(1-\alpha_k)}{\alpha_k^2+\alpha_{k+1}}=\frac{(\lambda_k-1)(1-\frac{\lambda_{k+1}}{{4}\kappa})}{\left(1-\frac{{1}}{{4}\kappa}\right)\lambda_{k+1}}.$
Hence \eqref{rule by alpha} can be written as 
\begin{align*}
& x_{k+1}=y_{k+1}+\sigma_k(y_{k+1}-y_k), \quad \sigma_k=\frac{(\lambda_k-1)\left(1-\frac{\lambda_{k+1}}{{4}\kappa}\right)}{\left(1-\frac{{1}}{{4}\kappa}\right)\lambda_{k+1}}.
\end{align*}\end{proof} 

{We now utilize the previous Lemma in defining an auxiliary function sequence $\{\phi_{k+1}(x)\}$ and a sequence $\{p_k\}$. These sequences form the basis for carrying out the final rate analysis. } 

\begin{lemma}\label{phi}
Suppose Assumptions \ref{fistaass2} and \ref{ass_error2}(i)   hold. Consider the iterates
	generated by Algorithm~\ref{nonsmooth sc scheme} where
	 $\gamma_k =1/(2L)$ while $\{v_k\}, \{\tau_k\}$, and
	$\{\alpha_k\}$ are defined in \eqref{v_k update}--\eqref{upd-tau}. Suppose $\phi_1(x) \triangleq F(\redd{x_0})+{\tau_1\over
			2}\|x-\redd{x_0}\|^2$ and $p_1 = 0$. If $\phi_k(x)$ and $p_k$ are
			defined as follows for $k \geq 1$:
\begin{align} 
 \phi_{k+1}(x) & := 
	 (1-\alpha_k)\phi_k(x)
		 +\alpha_k\Big[
	F(y_{k+1})+{1\over 4 L}\|h(x_k)\|^2+{\mu \over
		{4}}\|x-x_k\|^2+h(x_k)^T(x-x_k) \Big] 
\label{strong_1}
\\
	p_{k+1} & := (1-\alpha_k){\left({2\over L}+{1\over \mu}\right)\|\bar{w}_{k,N_k}\|^2}
		 +(1-\alpha_k)p_k,
 \label{def-pk}
\end{align}
where $h(x_k)= {2L}(x_k-{y_{k+1}})$.
If $\phi_k^* \triangleq \min_{x} \phi_k(x)$, 
then 
$ \phi_k^* \geq F(y_{k}) - p_k, \qquad \mbox{ for  all } k \geq 1. $
\end{lemma}
\begin{proof} 
We begin by showing that $\nabla^2\phi_k(x)=\tau_k I$,
		where $I$ denotes the identity matrix. For $k=1$, $\nabla^2\phi_1(x)=\tau_1 I$. Suppose, this holds for $k$ and we proceed to show that this holds for $k:=k+1:$
\begin{align}
\nabla^2\phi_{k+1}(x)&=(1-\alpha_k)\nabla^2\phi_k(x)+{1\over 2}\alpha_k \mu I
= (1-\alpha_k) \tau_k I + {1\over 2}\alpha_k \mu I.
\end{align}\label{gamma_k+1}
By choosing $\tau_{k+1}  =(1-\alpha_k)\tau_k+{1\over 2}\alpha_k\mu$, the required claim follows.  Next we show that the sequence {$\phi_k(x)$} can be written as follows:
\begin{align}\label{strong_2}
\phi_k(x)=\phi^*_k+{\tau_k\over 2}\|x-v_k\|^2,
\end{align}
where $\phi^*_k=\min_x \phi_k(x)$ and $v_k=\arg \min_x \phi_k(x)$. 
Since $\phi_{k+1}(x)$ is a convex quadratic function by definition, we may represent it as 
$ \phi_{k+1}(x) = a + b^Tx + {1\over 2} x^TQ x. $
First, we note that
$ \nabla^2 \phi_{k+1}(x) = Q = \tau_{k+1} I. $
By noting that $\nabla_x \phi_{k+1}(v_{k+1}) = 0$, implying that 
$ b + \tau_{k+1} v_{k+1} = 0 \implies b = -\tau_{k+1} v_{k+1}.$
Consequently, we have that 
$\phi_{k+1}(v_{k+1}) = \phi_{k+1}^* = a -\tau_{k+1}v_{k+1}^T v_{k+1} + {1\over 2} \tau_{k+1} \|v_{k+1}\|^2 \implies a = \phi_{k+1}^* + {\tau_{k+1} \over 2} \|v_{k+1}\|^2.$
This implies that $\phi_{k+1}(x) = \phi_{k+1}^* + \frac{\tau_{k+1}}{2}
\|x-v_{k+1}\|^2$ and 
\eqref{strong_2} has been shown to be true for all $k$.  Next, we proceed to obtain the recursive rule for $v_{k+1}$ and $\phi^*_{k+1}.$ 
By using the optimality conditions for the unconstrained strongly convex
problem $\min_x \phi_k(x)$, we obtain the following:
\begin{align}\label{v_{k+1}}
 &\nonumber 0 = \nabla_x \phi_{k+1}(x)=(1-\alpha_k)\nabla_x \phi_k(x)+\alpha_k\left[ {1\over 2}\mu(x-x_k)+h(x_k)\right] \\& \nonumber
\overset{\eqref{strong_2}}{=}(1-\alpha_k) \tau_k(x-v_k)+\alpha_k\left[
{1\over 2}\mu(x-x_k)+h(x_k) \right]\\
 & \implies {\nabla_x \phi_{k+1}(x) = \tau_{k+1}(x-v_{k+1})} 
\  \mbox{implying}\  v_{k+1}  ={1\over \tau_{k+1}}\left[ (1-\alpha_k)\tau_k
 v_k+{1\over 2}\alpha_k \mu x_k-\alpha_kh(x_k) \right].
\end{align}
{By using equations \eqref{strong_1} and \eqref{strong_2}, we obtain the following:
\begin{align*}
\phi^*_{k+1}&=\phi_{k+1}(x_k)-{\tau_{k+1}\over
	2}\|x_k-v_{k+1}\|^2\\&=(1-\alpha_k)\Big[ \phi^*_k+{\tau_k \over 2}\|x_k-v_k\|^2 \Big]
+\alpha_k\Big[ {F}(y_{k+1})+{1\over 4L}\|h(x_k)\|^2\Big]
-{\tau_{k+1}\over 2}\|x_k-v_{k+1}\|^2 \\
	& =(1-\alpha_k)\Big[ \phi^*_k+{\tau_k \over 2}\|x_k-v_k\|^2 \Big]
+\alpha_k\Big[ {F}(y_{k+1})+{1\over 4L}\|h(x_k)\|^2 \Big]\\&-{\tau_{k+1}\over 2}\Big\|x_k-{1\over \tau_{k+1}}\Big[ (1-\alpha_k)\tau_k
 v_k+{1\over 2}\alpha_k \mu x_k-\alpha_kh(x_k) \Big]\Big\|^2 \\
 & = (1-\alpha_k)\ \phi^*_k+ \alpha_k {F}(y_{k+1}) +(1-\alpha_k){\tau_k \over 2}\|x_k-v_k\|^2 
+\alpha_k\Big[ {1\over 4L}\|h(x_k)\|^2\Big]\\&-{\tau_{k+1}\over 2}\Big\|x_k-{1\over \tau_{k+1}}\Big[  (1-\alpha_k)\tau_k
 (v_k-x_k+x_k) +{1\over 2} \alpha_k \mu x_k -\alpha_kh(x_k) \Big]\Big\|^2.
\end{align*}
The expression on the right can be further simplified as follows:
\begin{align*}
& \phi_{k+1}^*   = (1-\alpha_k)\ \phi^*_k+ \alpha_k {F}(y_{k+1}) +(1-\alpha_k){\tau_k \over 2}\|x_k-v_k\|^2 
+\alpha_k\Big[ {1\over 4L}\|h(x_k)\|^2\Big]\\
	&-{\tau_{k+1}\over 2}\Big\|{1\over \tau_{k+1}}\Big[ - (1-\alpha_k)\tau_k
 (v_k-x_k)  +\alpha_kh(x_k) \Big]\Big\|^2 \\
& = (1-\alpha_k)\ \phi^*_k+ \alpha_k {F}(y_{k+1}) +(1-\alpha_k){\tau_k \over 2}\|x_k-v_k\|^2 
+\alpha_k\Big[ {1\over 4L}\|h(x_k)\|^2\Big]
	-{(1-\alpha_k)^2 \tau_k^2\over 2\tau_{k+1}}\|v_k-x_k\|^2 \\&-
	\frac{\alpha_k^2 }{2\tau_{k+1}}\|h(x_k)\|^2 +
	\frac{(1-\alpha_k)\alpha_k\tau_k}{\tau_{k+1}} h(x_k)^T(v_k-x_k) \\
& = (1-\alpha_k)\ \phi^*_k+ \alpha_k {F}(y_{k+1}) +(1-\alpha_k){\tau_k \over 2}\|x_k-v_k\|^2 
+\left({\alpha_k \over {4L}} -
		\frac{\alpha_k^2}{2\tau_{k+1}}\right)\|h(x_k)\|^2 		\\&-{(1-\alpha_k)^2 \tau_k^2\over 2\tau_{k+1}}\|v_k-x_k\|^2
+ \frac{(1-\alpha_k)\alpha_k\tau_k}{\tau_{k+1}} h(x_k)^T(v_k-x_k)\\
&\implies \phi_{k+1}^* = (1-\alpha_k)\ \phi^*_k+ \alpha_k {F}(y_{k+1}) +(1-\alpha_k){\tau_k
	\over 2}\left(1-\frac{(1-\alpha_k)\tau_k)}{\tau_{k+1}}\right)\|x_k-v_k\|^2 \\
& +\left({\alpha_k \over 4L} -
		\frac{\alpha_k^2}{2\tau_{k+1}}\right)\|h(x_k)\|^2 + \frac{(1-\alpha_k)\alpha_k\tau_k}{\tau_{k+1}} h(x_k)^T(v_k-x_k)\\
 & = (1-\alpha_k)\ \phi^*_k+ \alpha_k {F}(y_{k+1})
+{(1-\alpha_k)\alpha_k \tau_k (\mu/{2})
	\over 2\tau_{k+1}}\|x_k-v_k\|^2 
+\left({\alpha_k \over 4L} -
		\frac{\alpha_k^2}{2\tau_{k+1}}\right)\|h(x_k)\|^2 \\ &	+ \frac{(1-\alpha_k)\alpha_k\tau_k}{\tau_{k+1}} h(x_k)^T(v_k-x_k) \\
& = (1-\alpha_k)\ \phi^*_k+ \alpha_k {F}(y_{k+1})
+\tfrac{(1-\alpha_k)\alpha_k }{ \tau_{k+1}} \tau_k \left({\mu 
	\over {4}}\|x_k-v_k\|^2 + h(x_k)^T(v_k-x_k)\right) 
+\left(\tfrac{\alpha_k}{4L} -
		\tfrac{\alpha_k^2}{2\tau_{k+1}}\right)\|h(x_k)\|^2.
\end{align*}}
Next, we inductively prove that $\phi^*_k\geq {F}(y_k)-p_k$ {where $p_k$ is defined in \eqref{def-pk}}. This
holds for $k=1$ where $p_1=0$. Assuming, it is true for $k$, we 
prove it holds for $k+1$ by invoking
Lemma \ref{lemma_ac-vssa} for $x=y_k$:
\begin{align*}
\phi^*_{k+1}&\geq (1-\alpha_k)({F}(y_k)-p_k)+\alpha_k {F}(y_{k+1})+\Big( {\alpha_k \over 4L}-{\alpha_k^2\over 2\tau_{k+1}} \Big)\|h(x_k)\|^2\\&+{\alpha_k (1-\alpha_k)\tau_k\over \tau_{k+1}}\Big( {\mu \over
		{4}}\|x_k-v_k\|^2+h(x_k)^T(v_k-x_k) \Big) \qquad \mbox{\scriptsize (Since $\phi_k^* \geq {F}(y_k)-p_k$)} \\\
		& \geq (1-\alpha_k)  ({F}(y_{k+1}) + h(x_k)^T(y_k - x_k) + {1\over
				4L} \|h(x_k)\|^2 + {\mu \over {4}} \|y_k - x_k\|^2 \\
		& - \left({2\over L}+{1\over \mu}\right)\|\bar{w}_{k,N_k}\|^2) - (1-\alpha_k) p_k  +\alpha_k {F}(y_{k+1})+\Big( {\alpha_k \over {4L}}-{\alpha_k^2\over 2\tau_{k+1}} \Big)\|h(x_k)\|^2+{\alpha_k (1-\alpha_k)\tau_k\over \tau_{k+1}}\\& {\times} \Big( {\mu \over
		{4}}\|x_k-v_k\|^2+h(x_k)^T(v_k-x_k) \Big)\\
& = {F}(y_{k+1}) +\Big( {1 \over 4L}-{\alpha_k^2\over
		2\tau_{k+1}} \Big)\|h(x_k)\|^2 + (1-\alpha_k)
h(x_k)^T\left(\frac{\alpha_k \tau_k}{\tau_{k+1}} (v_k-x_k) + (y_k
			- x_k)\right)\\& -(1-\alpha_k)p_k - { (1-\alpha_k)\left({2\over L}+{1\over \mu}\right) \|\bar{w}_{k,N_k}\|^2}) +(1-\alpha_k) {\mu \over {4}} \|y_k -
x_k\|^2 + {\alpha_k (1-\alpha_k)\tau_k\over \tau_{k+1}}{\mu \over
		{4}}\|x_k-v_k\|^2\\
&\geq {F}(y_{k+1}) +(1-\alpha_k)h(x_k)^T\overbrace{\Big({\alpha_k\tau_k
		\over \tau_{k+1}}(v_k-x_k)+(y_k-x_k)\Big)}^{\text{{\tiny Term
	(a)}}} +\overbrace{\Big({1\over 4L}-{\alpha_k^2\over
		2\tau_{k+1}}\Big)}^{\text{{\tiny Term
	(b)}}}\|h(x_k)\|^2\\
& 	-(1-\alpha_k){\left({2\over L}+{1\over \mu}\right) \|\bar{w}_{k,N_k}\|^2}-(1-\alpha_k)p_k
=  {F}(y_{k+1})-(1-\alpha_k){\left({2\over L}+{1\over \mu}\right) \|\bar{w}_{k,N_k}\|^2}-(1-\alpha_k)p_k,
\end{align*}
where the last inequality follows noting that terms (a) and (b) are zero from recalling that
${2}L\alpha_k^2=\tau_{k+1}$ and $x_k={1\over \tau_k+{1\over 2}\alpha_k
	\mu}(\alpha_k \tau_k v_k+\tau_{k+1}y_k)$ (by Lemma~\ref{lemma_b3}). 
By choosing
$p_{k+1}=(1-\alpha_k){\left({2\over L}+{1\over \mu}\right)\|\bar{w}_{k,N_k}\|^2}+(1-\alpha_k)p_k,$ we
have
that~\footnote{Update rule for $x_k$, according to Lemma \ref{lemma_b3}, is equivalent to that in the algorithm. Also, compared with the approach by Nesterov, we employ inexact (rather than exact) gradients, the key difference in the proof is term(c)}
$\phi_{k+1}^*\geq {F}(y_{k+1})-\overbrace{p_{k+1}}^{\text{{\tiny Term
	(c)}}}.$	
\end{proof} 
{Before analyzing the rate of convergence, we proceed to examine the limiting behavior of the sequence  $\{\lambda_k\}$ and show that $\lambda_k \to \sqrt{\kappa}$, where $\kappa$ denotes the condition number of the problem. } 

\begin{lemma}[{\bf Properties of $\{\lambda_k\}$}] \label{bound lambda}
Suppose sequence $\{\lambda_k\}_{k\geq 1}$ {is defined by the recursion} 
\begin{align} \label{def-lambda}
\lambda_{k+1}:=\frac{1-\frac{\lambda_k^2}{{4}\kappa}+\sqrt{\left(1-\frac{\lambda_k^2}{{4}\kappa}\right)^2+4\lambda_k^2}}{2},.
\end{align}
where  $\lambda_1\in (1, {2}{\sqrt \kappa}].$
 Then $\{\lambda_k\}$ is an increasing {and bounded} sequence, such that 
$\lim_{k\rightarrow \infty} \lambda_{k}={2}\sqrt \kappa.$
\end{lemma}
\begin{proof} 
First by induction we show that sequence $\{\lambda_k\}$ is bounded above by {${2}\sqrt \kappa$}. By assumption, $\lambda_1\leq {{2}\sqrt \kappa}$, we assume $\lambda_k\leq {{2}\sqrt \kappa}$  and {proceed to show that}  $\lambda_{k+1}\leq {2}\sqrt \kappa$:
\begin{align*}
\lambda_{k+1}&=\frac{1-\frac{\lambda_k^2}{{4}\kappa}+\sqrt{\left(1-\frac{\lambda_k^2}{{4}\kappa}\right)^2+4\lambda_k^2}}{2} \Leftrightarrow \lambda_k^2 = \frac{\lambda_{k+1}(\lambda_{k+1}-1)}{1-\frac{\lambda_{k+1}}{{4}\kappa}}\\
&\implies \lambda_k\leq {2}\sqrt \kappa \Leftrightarrow \frac{\lambda_{k+1}(\lambda_{k+1}-1)}{1-\frac{\lambda_{k+1}}{{4}\kappa}}\leq {4}\kappa \Leftrightarrow \lambda_{k+1}^2\leq {4}\kappa \Leftrightarrow \lambda_{k+1}\leq {2}\sqrt \kappa.
\end{align*}
Since the sequence is increasing and bounded above, {its} limit exists. Suppose, $\lim_{k\rightarrow \infty} \lambda_{k+1}=\lambda$, implying 
$\lambda=\frac{1-\frac{\lambda^2}{{4}\kappa}+\sqrt{\left(1-\frac{\lambda^2}{{4}\kappa}\right)^2+4\lambda^2}}{2} \implies \lambda={2}\sqrt \kappa.$
Second we show that sequence $\{\lambda_k\}$ is increasing, i.e. $\lambda_{k+1}\geq \lambda_k$, which can be written equivalently by replacing {the recursive} rule $\lambda_{k+1}$ as follows
\begin{align*}
\tfrac{1-\frac{\lambda_k^2}{{4}\kappa}+\sqrt{(1-\frac{\lambda_k^2}{{4}\kappa})^2+4\lambda_k^2}}{2}\geq \lambda_k& \Leftrightarrow \left(1-\frac{\lambda_k^2}{{4}\kappa}\right)^2+4\lambda_k^2 \geq \left(\frac{\lambda_k^2}{{4}\kappa}-1+2\lambda_k\right)^2 \Leftrightarrow 4\lambda_k\left(1-\frac{\lambda_k^2}{{4}\kappa}\right) \leq 0 \Leftrightarrow \lambda_k\leq{2} \sqrt{\kappa}. 
\end{align*}
\end{proof} 
{We are now in a position to provide our main proposition that provides a bridge towards deriving rate statements and oracle complexity bounds. }

{\bf Proof of \sc  Lemma \ref{prop-err-bd}}.  \proof{} We have that:
\begin{align*}
& \quad  \mathbb{E}[\phi_{k+1}(x)]   \overset{\eqref{strong_1}}{=}(1-\alpha_k)\mathbb{E}[\phi_k(x)] +\alpha_k\mathbb{E}\Big[
	F(y_{k+1})+{1\over {4}L}\|h(x_k)\|^2+{\mu \over
		{4}}\|x-x_k\|^2+h(x_k)^T(x-x_k)\Big] \\&\leq(1-\alpha_k)\mathbb{E}[\phi_k(x)]+ \alpha_k \mathbb{E}[F(x)] 
	 + \alpha_k {\left({2\over L}+{1\over \mu}\right)\mathbb{E}[\|\bar{w}_{k,N_k}\|^2]}.
\end{align*}
By rearranging terms and setting $x=x^*$ in the inequality above, we obtain 
\begin{align*}
  & \mathbb{E}[\phi_{k+1}(x^*)-F(x^*)] \leq (1-\alpha_k) \mathbb{E}[\phi_k(x^*)-F(x^*)]  +{\left({2\over L}+{1\over \mu}\right)\mathbb{E}[\|\bar{w}_{k,N_k}\|^2]} \\
& \leq (1-\alpha_k)(1-\alpha_{k-1})\mathbb{E}[\phi_{k-1}(x^*)-F(x^*)]+\alpha_k{\left({2\over L}+{1\over \mu}\right) \mathbb{E}[\|\bar{w}_{k,N_k}\|^2]}
+\alpha_k(1-\alpha_{k-1}){\left({2\over L}+{1\over \mu}\right)\mathbb{E}[\|\bar{w}_{k-1,N_{k-1}}\|^2]]}\\
&\leq \left(\prod_{i=1}^{k}(1-\alpha_i)\right)\mathbb{E}[\phi_1(x^*)-F(x^*)]+\alpha_k\sum_{i=0}^{k-1} \left(\prod_{j=0}^{i-1} (1-\alpha_{k-j})\right) {\left({2\over L}+{1\over \mu}\right)\mathbb{E}[\|\bar{w}_{k-i,N_{k-i}}\|^2]}.
\end{align*}
From Lemma~\ref{bound lambda}, $\alpha_k={1\over \lambda_k}\in {[ \bar\alpha,1)}$ where {$\bar \alpha={1\over 2\sqrt \kappa}$}, and by recalling that $ \mathbb{E}[\|\bar{w}_{k-i,N_{k-i}}\|^2 \mid \mathcal{H}_{k-i}] \leq \nu^2/N_{k-i}$, we obtain the following sequence of inequalities:
\begin{align}\label{phi-k+1}
 \mathbb{E}[\phi_{k+1}(x^*)-F(x^*)]\nonumber&\leq \left(\prod_{i=1}^{k}(1-\alpha_i)\right)\mathbb{E}[\phi_1(x^*)-F(x^*)]+\sum_{i=0}^{k-1} \left( (1-\bar{\alpha})^i\right) {\left({2\over L}+{1\over \mu}\right)\mathbb{E}[\mathbb{E}[\|\bar{w}_{k-i,N_{k-i}}\|^2 \mid \mathcal{H}_{k-i}]]}\\
&\leq \left(\prod_{i=1}^{k}(1-\alpha_i)\right)\mathbb{E}[\phi_1(x^*)-F(x^*)]  +\sum_{i=0}^{k-1}  \left({2\over L}+{1\over \mu}\right){\frac{\nu^2(1-\bar\alpha)^i}{N_{k-i}}}.
\end{align}
By using Lemma \ref{phi} and \eqref{phi-k+1}, we may obtain
\begin{align}\label{E[p]}
 F(y_k)-F(x^*)&\nonumber
 \leq \mathbb{E}[\phi^*_k+p_k]-{F}(x^*) \leq \mathbb{E}[\phi_k(x^*)-{F}(x^*)]+\mathbb{E}[p_k]\\ \nonumber
& \leq  \left(\prod_{i=1}^{k-1}(1-\alpha_i)\right)\mathbb{E}[\phi_1(x^*)-F(x^*)]+\sum_{i=0}^{k-2}\left({2\over L}+{1\over \mu}\right) {\frac{\nu^2(1-\bar\alpha)^i }{{N_{k-1-i}}}}+\mathbb{E}[p_{k}]\\ \nonumber
&= \left(\prod_{i=1}^{k-1}(1-\alpha_i)\right)\mathbb{E}[F(\redd{x_0})-F(x^*)+{\tau_1\over
	2}\|x^*-\redd{x_0}\|^2]+\sum_{i=0}^{k-2}\left({2\over L}+{1\over \mu}\right) {\frac{\nu^2(1-\bar\alpha)^i }{{N_{k-1-i}}}}+\mathbb{E}[p_{k}]\nonumber 
\\ &  \leq (1-\bar \alpha)^{k-1}(D+{\mu\over 2}C^2) 
 +\sum_{i=0}^{k-2}\left({2\over L}+{1\over \mu}\right)  {\frac{\nu^2(1-\bar\alpha)^i}{{N_{k-1-i}}}}+\mathbb{E}[p_{k}],
\end{align}
where we used the fact that $\tau_1=\mu$ and {$\alpha_k \in [\bar\alpha,1)$}. Next, we derive a bound on  $\mathbb{E}[p_k]$. By definition, we have
$p_k = \left(1-\bar
\alpha\right){\left({2\over L}+{1\over \mu}\right)\|\bar{w}_{k-1,N_{k-1}}\|^2}+\left(1-\bar
\alpha\right)p_{k-1}$, implying that
\begin{align*}
p_k&=\left(1-\bar \alpha\right){\left({2\over L}+{1\over \mu}\right)\|\bar{w}_{k-1,N_{k-1}}\|^2}+\left(1-\bar \alpha\right)^2{\left({2\over L}+{1\over \mu}\right)\|\bar{w}_{k-2,N_{k-2}}\|^2}+\left(1-\bar \alpha\right)^2p_{k-2}\\
&=\hdots=\sum_{i=0}^{k-2}\left(1-\bar
\alpha\right)^{i+1}{\left({2\over L}+{1\over \mu}\right)\|\bar{w}_{k-i-1,N_{k-i-1}}\|^2}.
\end{align*}
By taking expectations and invoking
Assumptions~\ref{fistaass2} and \ref{ass_error2}(i), 
\begin{align} \label{bd-rate}
 & \quad \mathbb{E}\left[p_k\right] 
\leq \sum_{i=0}^{k-2}\left(1-\bar \alpha\right)^{i+1}{\left({2\over L}+{1\over \mu}\right)
  \mathbb{E}[\mathbb{E}[\|\bar{w}_{k-i-1,N_{k-i-1}}\|^2 \mid \mathcal{H}_{k-i-1}]]} \leq \sum_{i=0}^{k-2}{\left({2\over L}+{1\over \mu}\right){\nu^2 \left(1-\bar \alpha\right)^{i+1}\over N_{k-i-1}}}.
\end{align}
By substituting \eqref{bd-rate} in \eqref{E[p]}, we obtain the desired result. \qed

{\bf Proof of \sc Theorem \ref{th-rate-itercomp-sc}.}  
\proof{} 
\noindent {\bf (i).} From \eqref{bound_ac-VSSA_strong} and by the definition of $\theta$, we may claim the following:
\begin{align}\label{bound N_k}
& \mathbb{E}[F(y_K)-F^*]\nonumber\leq \left(D+{\mu\over 2}C^2\right)\theta^{K-1}+ \sum_{j=0}^{K-2} \theta^{j} \left({2\over L}+{1\over \mu}\right){\nu^2\over
	 {N_{K-j-1}}}+\sum_{j=0}^{K-2}\theta^{j+1}\left({2\over L}+{1\over \mu}\right){\nu^2 \over  {N_{K-j-1}}} \notag \\ 
&= \left(D+{\mu\over 2}C^2\right)\theta^{K-1}+ \left({2\over L}+{1\over \mu}\right)\theta\sum_{j=0}^{K-2} \theta^j {4\nu^2\over
	{N_{K-j-1}}} 
 \leq \left(D+{\mu\over 2}C^2\right)\theta^{K-1}+ \sum_{j=0}^{K-2} \theta^{j} \left({2\over L}+{1\over \mu}\right){2\nu^2\over {N_{K-j-1}}},
\end{align}
where in the last inequality we used the fact that $\bar \alpha+2\theta=2-\bar \alpha \leq2$. If $ N_{K-j-1}=\lfloor\rho^{-(K-j-1)}\rfloor$, by using Lemma \ref{fistabound for floor}, we have the following:
\begin{align}\label{b_ac2}
\nonumber& \quad  \sum_{i=0}^{K-2}\left({2\over L}+{1\over \mu}\right){2\theta^{j} \nu^2 \over  {\lfloor \rho^{-(K-j-1)}\rfloor}}
 \leq \sum_{i=0}^{K-2}\left({2\over L}+{1\over \mu}\right){\theta^{i} \nu^2 \over  { \rho^{-(K-i-1)}}}  \leq \left({2\over L}+{1\over \mu}\right){ \nu^2 } {\rho^{K-1}} \sum_{i=0}^{K-2}{ \left(\frac{\theta}{\rho}\right)^{i} } 
\\& \leq \left({2\over L}+{1\over \mu}\right)\left({ \nu^2 \rho \over  (\rho-\theta)}\right){\rho^{K-1}}.
\end{align}
By substituting \eqref{b_ac2} in \eqref{bound N_k}, the bound in terms
of $K$ is provided next where $\tilde C$ is defined in \eqref{bd-K}:
\begin{align}\label{bound for K}
& \quad \mathbb{E}[F(y_K)-F^*] 
  \leq \left(D+{\mu\over 2}C^2\right)\theta^{K-1}  +  
\left({2\over L}+{1\over \mu}\right)2 \nu^2 \sqrt{\kappa}{\rho^{K-1}} 
 \leq \tilde C \rho^{K-1} \\ 
\notag \mbox{ where }	\tilde C  & = \left(D+ \frac{\mu C^2}{2} \right) + \left({2\over L}+{1\over \mu}\right)2 \nu^2 \sqrt{\kappa}
		 \leq \left(D+ \frac{\mu C^2}{2} \right) + \frac{4 \nu^2 }{{ \mu}}+{2\nu^2\sqrt\kappa\over \mu} 
\end{align}
Furthermore, {we may derive the number of steps $K$ to obtain an \redd{$\epsilon$-optimal} solution:} 
\begin{align}
	\frac{1}{\rho} & = \frac{1}{(1-\frac{1}{{2}a \sqrt{\kappa}})} 
			 = \frac{{2}a \sqrt{\kappa}}{({2}a\sqrt{\kappa}-1)} \implies K \geq \frac{ \log(\tilde C) - \log(\epsilon)} {\log(1/\rho)} \approx \mathcal{O} (\sqrt{\kappa}) \log({\sqrt \kappa}/\epsilon).
\end{align}
\noindent {\bf (ii)} {To compute a vector $y_{K+1}$} satisfying $\mathbb{E}[F(y_{{K+1}})-F^*]\leq \epsilon$,we have $\tilde C{\rho}^{K}\leq \epsilon$, implying that  
$K = \lceil \log_{(1/  {\rho})}(\tilde C/\epsilon)\rceil.$
To obtain the optimal oracle complexity, we require $\sum_{k=1}^{K} N_k$ gradients. If $N_k=\lfloor \rho^{-k}\rfloor\leq \rho^{-k}$, we obtain the following since $(1-\rho) = (1 \slash (a \sqrt{\kappa}))$.
\begin{align*}
 \quad \sum_{k=1}^{K} \rho^{-k} 
 &\leq \frac{1}{\left(\frac{1}{{\rho}} -1\right)}\left({1\over \rho}\right)^{2+K}  \leq \frac{1}{\left(\frac{1}{{\rho}} -1\right)}\left({1\over \rho}\right)^{3+\log_{(1/  {\rho})}\left(\tilde C/\epsilon\right)}  \leq \left( \tilde C \over \epsilon\right)\frac{1}{\rho^2(1-{\rho})} 
 = \frac{ a \sqrt{\kappa} \tilde C}{\rho^2\epsilon}. \\
  \rho=1-{1\over {2}a\sqrt\kappa} \implies \rho^2 & = 1-2/({2}a\sqrt \kappa)+1/({4}a^2\kappa)= {{4}a^2\kappa-{4}a\sqrt \kappa+1\over {4}a^2\kappa}\geq{ {4}a^2\kappa-{8}a\kappa\over {4}a^2\kappa}={(a^2-2a)\kappa\over a^2\kappa}\\
  \implies & {\sqrt\kappa\over \rho^2}\leq {a^2\kappa \sqrt \kappa\over (a^2-2a)\kappa}=\left(a\over a-2\right)\sqrt\kappa
 \implies   \sum_{k=1}^{\log_{(1/{\rho})}\left(\tilde C/\epsilon\right)+1} \rho^{-k} \leq {{2}a^2\sqrt\kappa \tilde C\over (a-2)\epsilon}\\
& ={ \left(\left(D+ \frac{\mu C^2}{2} \right)+ \frac{4 \nu^2 }{{ \mu}}+{2\nu^2\sqrt\kappa\over \mu} \right)} \mathcal O\left({\sqrt \kappa\over \epsilon}\right).\qed
\end{align*}

{\bf Proof of \sc Lemma \ref{char-eta}. } 

\noindent (i) $\lim_{\eta \to 0} {\widehat{C}}(\eta) = +\infty$ and $\lim_{\eta \to +\infty} {\widehat{C}}(\eta) = +\infty$ since $\lim_{\eta \to 0} \tilde{\kappa}(\eta) = +\infty$ and $\lim_{\eta \to +\infty} \tilde{\kappa}(\eta) = 1. $ In other words, $\bar{C}(\eta)$ is a coercive function on the set $\{\eta: \eta \geq 0\}$.

\noindent (ii) We observe that for $\eta > 0$,  
$$ \tilde{\kappa}(\eta) = 1+ \tfrac{1}{\eta \mu} > 0,\qquad \tilde{\kappa}(\eta)' = -\tfrac{1}{\eta^2 \mu} < 0, \qquad \tilde{\kappa}''(\eta) = \tfrac{2}{\eta^3 \mu} > 0. $$
Furthermore, $Q(\eta) = \max\{\eta^2M^2, 4 \Delta^2\}$ and $\bar{\eta} \triangleq \tfrac{2\Delta}{M}$. Therefore, we have that $Q(\eta)$ is a.e. twice  differentiable and its  Clarke generalized gradient and Hessian are defined as follows. 
\begin{align} \partial_{\eta} Q(\eta) = \begin{cases} 
			\{2\eta M^2\}, & \eta > \bar{\eta} \\
			[0,2\bar{\eta} M^2], & \eta =  \bar{\eta}\\
			\{0\},  &  \eta <  \bar{\eta}
	\end{cases} \mbox{ and }  
 \partial^2_{\eta}Q(\eta) = \begin{cases} 
			2 M^2, & \eta > \bar{\eta} \\
			\left\{ 2\alpha M^2 \mid \alpha \in [0,1]\right\} & \eta = \bar{\eta}, \\
			0.  &  \eta <  \bar{\eta}
	\end{cases} 
\end{align} 
From ~\cite[Prop.~7.1.9]{facchinei02finite} and by recalling that $\tilde{\kappa}(\eta)$ is continuously differentiable in $\eta$, we may define $\partial \widehat{C}(\eta)$ as follows.   
\begin{align}\notag
 \partial_{\eta} \widehat{C}(\eta) & = \partial [2D \eta \tilde{\kappa}] + \partial [{8 \tilde{\kappa}(\eta)^{5/2}Q(\eta)a}] 
		 =   2D\eta \tilde{\kappa}' + 2D \tilde{\kappa} + 20  \tilde{\kappa}^{3/2} \tilde{\kappa}' Q(\eta) a + 8 \tilde{\kappa}^{5/2} a \partial Q(\eta) \\
		& = \begin{cases}
	\left\{2D\eta \tilde{\kappa}' + 2D \tilde{\kappa} + 20  \tilde{\kappa}^{3/2} \tilde{\kappa}' Q(\eta) a + 8 \tilde{\kappa}^{5/2} a Q'(\eta)\right\}, & \eta > \bar{\eta} \\
	\left\{2D\bar{\eta} \tilde{\kappa}' + 2D \tilde{\kappa} + 20  \tilde{\kappa}^{3/2} \tilde{\kappa}' Q(\bar{\eta}) a + 8 \tilde{\kappa}^{5/2} a (2\alpha \bar{\eta} M^2) \mid \alpha \in [0,1]\right\}, & \eta = \bar{\eta} \\
	\left\{2D\eta \tilde{\kappa}' + 2D \tilde{\kappa} + 20  \tilde{\kappa}^{3/2} \tilde{\kappa}' Q(\eta) a\right\}, & \eta < \bar{\eta} 
\end{cases}
\end{align}
We may then define the Clarke generalized Hessian of $\widehat{C}$ as follows. 
 \begin{align}
\partial^2_{\eta} \widehat{C}(\eta) & = \begin{cases} 
\left\{\begin{aligned}	& \left\{  4D \tilde{\kappa}' + 2D \eta \tilde{\kappa}'' + 30  \tilde{\kappa}^{1/2} (\tilde{\kappa}')^2 Q(\eta) a +  20  \tilde{\kappa}^{3/2} \tilde{\kappa}'' Q(\eta) a +  20  \tilde{\kappa}^{3/2} \tilde{\kappa}' (2\eta M^2) a \right. \\
	& \left. + 20 \tilde{\kappa}^{3/2} \tilde{\kappa}' (2\eta M^2) a+  8  \tilde{\kappa}^{5/2} (2M^2) a \right\} \end{aligned} \right\},  & \eta > \bar{\eta} \\
\left\{\begin{aligned}	& \left\{  4D \tilde{\kappa}' + 2D \eta \tilde{\kappa}'' + 30  \tilde{\kappa}^{1/2} (\tilde{\kappa}')^2 Q(\eta) a +  20  \tilde{\kappa}^{3/2} \tilde{\kappa}'' Q(\eta) a +  20  \tilde{\kappa}^{3/2} \tilde{\kappa}' (2\alpha \eta M^2) a \right. \\
	& \left. + 20 \tilde{\kappa}^{3/2} \tilde{\kappa}' (2\alpha \bar{\eta} M^2) a+  8  \tilde{\kappa}^{5/2} (2\alpha M^2) a \mid \alpha \in [0,1]\right\} \end{aligned} \right\},  & \eta = \bar{\eta} \\
 \left\{  4D \tilde{\kappa}' + 2D \eta \tilde{\kappa}'' + 30
\tilde{\kappa}^{1/2} (\tilde{\kappa}')^2 Q(\eta) a +  20  \tilde{\kappa}^{3/2}
\tilde{\kappa}'' Q(\eta) a\right\}. & \eta < \bar{\eta} \\
\end{cases}\notag
\end{align} 
We now proceed to show that $H \succ 0$ for all $H \in \partial^2 \widehat{C}(\eta)$ and for all $\eta > 0$. 
  
\noindent Case 1: $0 < \eta < \bar{\eta}$. In this setting, $Q'(\eta) = Q''(\eta) = 0$. It follows that $\partial^2 \widehat{C}(\eta)$ is a singleton given by the scalar $H$ and it suffices to show that $H > 0$. This follows as shown next.  
\begin{align*}
 H & =  4D \tilde{\kappa}' + 2D \eta \tilde{\kappa}'' + 30  \tilde{\kappa}^{1/2} (\tilde{\kappa}')^2 Q(\eta) a +  20  \tilde{\kappa}^{3/2} \tilde{\kappa}'' Q(\eta) a \\
	& = 2D (\tfrac{2}{\eta^2 \mu} - \tfrac{2}{\eta^2 \mu}) + \underbrace{30  \tilde{\kappa}^{1/2} (\tilde{\kappa}')^2 Q(\eta) a +  20  \tilde{\kappa}^{3/2} \tilde{\kappa}'' Q(\eta) a}_{\ > \ 0} \ >    0.  
\end{align*}

\noindent Case 2: $\eta > \bar{\eta}$.  Since  $Q'(\eta) = 2\eta M^2$ and $Q''(\eta) =2 M^2$ for $\eta > \bar{\eta}$, we have that $\partial^2 \widehat{C}(\eta) = \{H\}$, where it suffices to show that $H > 0$. This follows as shown next.  
\begin{align*}
 H  & =  4D \tilde{\kappa}' + 2D \eta \tilde{\kappa}'' + 30  \tilde{\kappa}^{1/2} (\tilde{\kappa}')^2 Q(\eta) a +  20  \tilde{\kappa}^{3/2} \tilde{\kappa}'' Q(\eta) a +  40  \tilde{\kappa}^{3/2} \tilde{\kappa}' Q'(\eta) a
	 +  8  \tilde{\kappa}^{5/2}  Q''(\eta) a \\
		& = 2D (\tfrac{2}{\eta^2 \mu} - \tfrac{2}{\eta^2 \mu}) + 30  \tilde{\kappa}^{1/2} (\tilde{\kappa}')^2 Q(\eta) a +  8  \tilde{\kappa}^{5/2}  Q''(\eta) a
		+  \tilde{\kappa}^{3/2}  ( 20\tilde{\kappa}'' Q(\eta) +  40 \tilde{\kappa}' Q'(\eta))a \\
		& \geq 30  \tilde{\kappa}^{1/2} (\tilde{\kappa}')^2 Q(\eta) a +  8  \tilde{\kappa}^{5/2}   Q''(\eta) a
		+  \tilde{\kappa}^{3/2}  \left(\tfrac{40\eta^2 M^2}{\eta^3 \mu}\right)a - \tilde{\kappa}^{3/2} \left(\tfrac{80 \eta M^2}{\eta^2 \mu}\right)a \\
		& \geq 30  \tilde{\kappa}^{1/2} \tfrac{M^2}{\eta^2\mu^2}a +  16  \tilde{\kappa}^{5/2}  M^2a
		+  \tilde{\kappa}^{1/2} (1+\tfrac{1}{\eta \mu})  \left(\tfrac{40 M^2}{\eta \mu}\right) a -  \tilde{\kappa}^{1/2} \left(\tfrac{80 M^2}{\eta^2 \mu^2}\right)a. 
		\end{align*}
where the first term follows from $Q(\eta) = 2\eta^2 M^2$ and $\tilde{\kappa}' = -\tfrac{1}{\eta^2\mu}$ and the last term follows from $-\tilde{\kappa}^{3/2} \left( \tfrac{80 \eta M^2}{\eta^2 \mu}\right)a \leq -\tilde{\kappa}^{1/2} \left( \tfrac{80 M^2}{\eta^2\mu^2}\right)a$ since $-\tilde{\kappa}^{3/2} = -\tilde{\kappa}^{1/2} (1+\tfrac{1}{\eta \mu}) \leq -\tfrac{\tilde{\kappa}^{1/2}}{\eta \mu}.$   
\begin{align*}
 & \quad   \tilde{\kappa}^{1/2} \left(\tfrac{30M^2}{\eta^2\mu^2}\right)a +  16  \tilde{\kappa}^{5/2}M^2a +  \tilde{\kappa}^{1/2} (1+(1+\tfrac{1}{\eta \mu}))  \left(\tfrac{40 M^2}{\eta \mu}\right)a -  \tilde{\kappa}^{1/2} \left(\tfrac{80 M^2}{\eta^2 \mu^2}\right)a\\
& \geq \tilde{\kappa}^{1/2} \left(\tfrac{30M^2}{\eta^2\mu^2}\right)a +  16  \tilde{\kappa}^{1/2} (1+\tfrac{2}{\eta \mu}+\tfrac{1}{\eta^2\mu^2})M^2a +  \tilde{\kappa}^{1/2} \left( \tfrac{40 M^2}{\eta^2 \mu^2}\right)a -  \tilde{\kappa}^{1/2}\left( \tfrac{80 M^2}{\eta^2 \mu^2}\right)a\\
& \geq \tilde{\kappa}^{1/2} \left(\tfrac{30M^2}{\eta^2\mu^2}\right)a +  \tilde{\kappa}^{1/2}\left( \tfrac{16 M^2 }{\eta^2 \mu^2}\right)a +  \tilde{\kappa}^{1/2} \left( \tfrac{40 M^2}{\eta^2 \mu^2}\right)a -  \tilde{\kappa}^{1/2}\left( \tfrac{80 M^2}{\eta^2 \mu^2}\right)a\\
& =  \tilde{\kappa}^{1/2} \left( \tfrac{6M^2}{\eta^2\mu^2}\right)a > 0. 
\end{align*} 
\noindent Case 3: $\eta = \bar{\eta}$.  Suppose $Q'(\bar{\eta}) \in \partial \widehat{C}(\bar{\eta})$ and $H \in \partial^2 \widehat{C}(\bar{\eta})$,  where $Q'(\bar{\eta}) = 2\alpha \bar{\eta} M^2$ and $H =2\alpha  M^2$ and $\alpha \in [0,1]$. It suffices to show that $H > 0$ for $\alpha \in [0,1]$, as we proceed to do next.
\begin{align*}
 H  & =  4D \tilde{\kappa}' + 2D \eta \tilde{\kappa}'' + 30  \tilde{\kappa}^{1/2} (\tilde{\kappa}')^2 Q(\bar{\eta}) a +  20  \tilde{\kappa}^{3/2} \tilde{\kappa}'' Q(\eta) a +  40  \tilde{\kappa}^{3/2} \tilde{\kappa}' Q'(\eta) a
	 +  8  \tilde{\kappa}^{5/2}  Q''(\bar{\eta}) a \\
		& = 2D (\tfrac{2}{\bar{\eta}^2 \mu} - \tfrac{2}{\bar{\eta}^2 \mu}) + 30  \tilde{\kappa}^{1/2} (\tilde{\kappa}')^2 Q(\eta) a +  8  \tilde{\kappa}^{5/2}  Q''(\bar{\eta}) a
		+  \tilde{\kappa}^{3/2}  ( 20\tilde{\kappa}'' Q(\bar{\eta}) +  40 \tilde{\kappa}' Q'(\bar{\eta}))a \\
		& \geq 30  \tilde{\kappa}^{1/2} (\tilde{\kappa}')^2 Q(\bar{\eta}) a +  8  \tilde{\kappa}^{5/2}   Q''(\bar{\eta}) a
		+  \tilde{\kappa}^{3/2}  \left(\tfrac{40\bar{\eta}^2 M^2}{\eta^3 \mu}\right)a - \tilde{\kappa}^{3/2} \left(\tfrac{80 \alpha \eta M^2}{\eta^2 \mu}\right)a \\
		& \geq \tilde{\kappa}^{1/2} \left(\tfrac{30M^2}{\bar{\eta}^2\mu^2}\right)a +  16  \tilde{\kappa}^{5/2}  \alpha M^2a
		+  \tilde{\kappa}^{1/2} (1+\tfrac{1}{\bar{\eta} \mu})  \left(\tfrac{40 M^2}{\bar{\eta} \mu}\right) a -  \tilde{\kappa}^{1/2} \left(\tfrac{80 \alpha M^2}{\bar{\eta}^2 \mu^2}\right)a \\ 
 & \geq \tilde{\kappa}^{1/2} \left(\tfrac{30M^2}{\bar{\eta}^2\mu^2}\right)a +  16  \tilde{\kappa}^{5/2} \alpha M^2a +  \tilde{\kappa}^{1/2} \left( \tfrac{40 M^2}{\bar{\eta}^2 \mu^2}\right)a -  \tilde{\kappa}^{1/2}\left( \tfrac{80 M^2}{\bar{\eta}^2 \mu^2}\right)a\\
& \geq \tilde{\kappa}^{1/2} \left(\tfrac{30M^2}{\eta^2\mu^2}\right)a +  16  \tilde{\kappa}^{1/2} (1+\tfrac{2}{\eta \mu}+\tfrac{1}{\eta^2 \mu^2}) \alpha M^2a +  \tilde{\kappa}^{1/2} \left( \tfrac{40 M^2}{\bar{\eta}^2 \mu^2}\right)a -  \tilde{\kappa}^{1/2}\left( \tfrac{80 \alpha M^2}{\bar{\eta}^2 \mu^2}\right)a\\
& \geq \tilde{\kappa}^{1/2} \left(\tfrac{30M^2}{\bar{\eta}^2\mu^2}\right)a +  \tilde{\kappa}^{1/2}\left( \tfrac{16 \alpha M^2 }{\bar{\eta}^2 \mu^2}\right)a +  \tilde{\kappa}^{1/2} \left( \tfrac{40 M^2}{\bar{\eta}^2 \mu^2}\right)a -  \tilde{\kappa}^{1/2}\left( \tfrac{80 \alpha M^2}{\bar{\eta}^2 \mu^2}\right)a\\
& = \tilde{\kappa}^{1/2} \left(\tfrac{30M^2}{\bar{\eta}^2\mu^2}\right)a +  \tilde{\kappa}^{1/2} \left( \tfrac{40 M^2}{\bar{\eta}^2 \mu^2}\right)a -  \tilde{\kappa}^{1/2}\left( \tfrac{64 \alpha M^2}{\bar{\eta}^2 \mu^2}\right)a\\
& \overset{\alpha \leq 1}{\geq} \tilde{\kappa}^{1/2} \left(\tfrac{30M^2}{\bar{\eta}^2\mu^2}\right)a +  \tilde{\kappa}^{1/2} \left( \tfrac{40 M^2}{\bar{\eta}^2 \mu^2}\right)a -  \tilde{\kappa}^{1/2}\left( \tfrac{64  M^2}{\bar{\eta}^2 \mu^2}\right)a\\
& =  \tilde{\kappa}^{1/2} \left( \tfrac{6M^2}{\bar{\eta}^2\mu^2}\right)a  > 0. 
\end{align*}
Consequently, we have that $H > 0$ for $H \in \partial^2 \widehat{C}(\eta)$ and $\eta > 0$. It follows that $\widehat{C}$ is strictly convex for $\eta > 0$ (cf.~\cite[Ex.~2.2.]{hu84generalized}). Since $\widehat{C}(0) = +\infty$, we may then conclude from the definition of convexity that $\widehat{C}$ is a strictly convex function on $\{\eta \mid \eta \geq 0\}$. 

\noindent (iii) By part (i), a minimizer of $\widehat{C}(\eta)$ exists in $\{\eta: \eta \geq 0\}$. By part (ii), this minimizer is necessarily unique since $\widehat{C}$ is strictly convex. Therefore $\widehat{C}$ has a unique minimizer on $\{\eta \mid \eta \geq 0\}$.     \qed

{\bf Proof of \sc Proposition \ref{bd-SSG}. } 
{\begin{proof} 
    \noindent (a).  Since $\mathbb{E}[\uss{\F}(\uss{\bullet},\uss{\omega})+\tfrac{1}{2\eta}\|x_k-\uss{\bullet}\|^2]$ is $\tilde\mu$-strongly convex, where $\tilde\mu=\mu+\tfrac{1}{\eta}$ \uss{and $x_k$ is $\mathcal{F}_k$-measurable}, we may  utilize the proof technique in~\cite[{Section 5.9.1}]{shapiro09lectures} to obtain the following {for $j \geq 0$.}  
\begin{align}\label{sub_g_1}
    \mathbb E[\|z_{k,j+1}-z_{k}^*\|^2 \uss{ \ \mid \ \uss{\mathcal{F}_k} }]\nonumber&\leq (1-2\sigma_j\tilde\mu)\mathbb E[\|z_{k,j}-z_{k}^*\|^2\uss{ \ \mid \ \uss{\mathcal{F}_k} }]+\gamma_j^2(M_1^2\uss{\mathbb{E}\left[\|z_{k,j}\|^2 \mid \uss{\mathcal{F}_k}\right]} + M_2^2 \|x_k\|^2 + M_3^2)\\
                                                                                    &\overset{\eqref{bd-sub-G}}\leq (1-2\sigma_j\tilde\mu+{2}\sigma_j^2 M_1^2)\mathbb E[\|z_{k,j}- z_{k}^*\|^2 \uss{\ \mid  \uss{\mathcal{F}_k}} ]\notag \\
                                                                                    & +\sigma_j^2(2M_1^2\uss{\mathbb{E}[\| z_{k}^*\|^2 \mid \uss{\mathcal{F}_k}]}+ M_2^2\|x_k\|^2 + M_3^2). 
\end{align}
If $e_j \triangleq E[\|z_{k,j}- z_{k}^*\|^2 \uss{ \ \mid \uss{\mathcal{F}_k}} ]$ and $d_k \triangleq 2M_1^2\uss{\mathbb{E}[\| z_{k}^*\|^2 \mid \uss{\mathcal{F}_k}]}+ M_2^2\|x_k\|^2 + M_3^2$, {for} any $t_j>0$, we have that 
\begin{align}\label{bound ej}
e_{j+1}\leq (1-2\sigma_j\tilde\mu+2\sigma_j^2 M_1^2)e_j+\sigma_j^2d_k\implies t_{j+1}e_{j+1}&\leq t_{j+1}(1-2\sigma_j\tilde\mu+2\sigma_j^2 M_1^2)e_j+t_{j+1}\sigma_j^2d_k.
\end{align}
We intend to show that $t_{j+1}(1-2\sigma_j\tilde\mu+2\sigma_j^2 M_1^2)e_j\leq  t_je_j$. Let $\bar J, t_j$, and $\sigma_j$ be defined as 
\begin{align}\label{def t sig}
\bar J\triangleq \lceil \tfrac{2M_1^2}{\tilde \mu^2}-1\rceil, 
t_j   \triangleq \left. \begin{cases}
\left(1-\tfrac{\tilde \mu^2}{2M_1^2}\right)^{-j}, & j<\bar J \\
j, & j\geq \bar J
\end{cases} \right\}, \mbox{and} \left.
\sigma_j \triangleq \begin{cases}
\min\left\{\tfrac{1}{(j+1)\log(j+1)},\tfrac{\tilde \mu}{M_1^2} \right\}, & j<\bar J \\
\tfrac{1}{(j+1)\log(j+1)}, & j\geq \bar J
\end{cases}\right\}
\end{align}
For $j {\ \geq \ } \bar J$, we have the following. 
{\begin{align}
\label{bound t}
& \quad t_{j+1}(1-2\sigma_j\tilde\mu+2\sigma_j^2 M_1^2)\leq  t_j\qquad   
\Leftrightarrow \quad  (1-2\sigma_j\tilde\mu+2\sigma_j^2 M_1^2)\leq  \tfrac{t_j}{t_{j+1}}\\
\notag
\Leftrightarrow & \quad (1-\tfrac{t_j}{t_{j+1}}-2\sigma_j\tilde\mu+2\sigma_j^2 M_1^2)\leq 0  
\quad \Leftrightarrow  \quad \sigma_j \leq \tfrac{\tilde{\mu} + \sqrt{ \tilde \mu^2 - 2M_1^2 \left(1-\tfrac{t_j}{t_{j+1}}\right)}}{2M_1^2}.
\end{align}}
{From} \eqref{def t sig}, {we have that $\tfrac{t_j}{t_{j+1}} = (1-\tfrac{1}{j+1}) $ for $j \geq \bar J$.} {Consequently, $$ 2M_1^2(1-\tfrac{t_j}{t_{j+1}}) =\tfrac{ 2M_1^2}{j+1} \leq \tfrac{2M_1^2}{\big \lceil \tfrac{2M_1^2}{\tilde \mu^2}-1\big \rceil+1} \leq \tilde{\mu}^2 \implies 
\tilde \mu^2 - 2M_1^2 \left(1-\tfrac{t_j}{t_{j+1}}\right) \geq 0. $$ }  
Using \eqref{bound t}, we may show that \eqref{bound ej} is bounded as follows for $j\geq\bar J$:
\begin{align}\label{t1}
t_{j+1}e_{j+1}\nonumber&\leq t_{j+1}(1-2\sigma_j\tilde\mu+2\sigma_j^2 M_1^2)e_j+t_{j+1}\sigma_j^2d_k\leq   t_je_j+t_{j+1}\sigma_j^2d_k\leq t_0e_0+\overbrace{\sum_{{\ell}=0}^{\bar J-1}\sigma_{\ell}^2t_{{\ell}+1}d_k}^{\leq c_{\bar J} d_k}+{\sum_{\ell=\bar J}^j\sigma_{\ell}^2t_{{\ell}+1}d_k}\\ \notag
	& \leq t_0e_0 + c_{\bar J} d_k + \sum_{\ell = \bar J}^{j} \tfrac{\ell}{(\ell+1)^2 \log^2 (\ell+1)} d_k \leq t_0e_0 + c_{\bar J} d_k + \sum_{\ell = \bar J}^{j}\tfrac{1}{(\ell+1) \log (\ell+1)} d_k  \notag\\    
& \leq t_0e_0+ (c_{\bar J}+3) d_k \triangleq t_0 e_0+\bar d_k,
\end{align}
where \eqref{t1} follows from $\sum_{j=1}^\infty
\tfrac{1}{(j+1)\log(j+1)}\leq 3$. {Next, we derive a bound on $e_0 =
    \mathbb{E}[\|z_{k,0} - z_{k}^*\|^2  \ \mid \uss{\mathcal{F}_k} ]$.}
\begin{align*}
    \mathbb{E}[\|z_{k,0} - z_{k}^*\|^2  \ \mid \uss{\mathcal{F}_k}]
        &  = \mathbb{E}[\|x_k - z_{k}^*\|^2  \ \mid \uss{\mathcal{F}_k}]  \leq 2\uss{\|x_k - x^*\|^2} + 2\mathbb{E}[\|x^*-z^*_{k}\|^2 \ \mid \uss{\mathcal{F}_k}] \\
        & = 2\uss{\|x_k - x^*\|^2} + 2\mathbb{E}[\|\mbox{prox}_{\eta F} (x^*)- \mbox{prox}_{\eta F} (x_k)\|^2  \ \mid \uss{\mathcal{F}_k}] \leq 4\uss{\|x_k - x^*\|^2}, 
 \end{align*}
 \uss{where the last inequality is a result of $x_k$ being $\mathcal{F}_k$-measurable and non-expansivity of the prox. operator.} Similarly, $d_k$ can be bounded as follows.
\begin{align*}
    d_k & = (2M_1^2 \mathbb{E}[\|z^*_{k}\|^2  \ \mid \uss{\mathcal{F}_k}] + M_2^2 \uss{\|x_k\|^2} + M_3^2) \\
        & \leq 4M_1^2\mathbb{E}[\|z^*_{k}-x^*\|^2  \mid \uss{\mathcal{F}_k}] + 4M_1^2 [\|x^*\|^2] + 2M_2^2 \uss{\|x_k-x^*\|^2} + 2M_2^2 \|x^*\|^2 + M_3^2 \\
				 & 	 \leq (4M_1^2 + 2M_2^2) \uss{\|x_{k}-x^*\|^2} + (4M_1^2+2M_2^2) \|x^*\|^2+M_3^2, 
\end{align*}
where the last inequality follows from $\| z_k^*-x^*\| = \| \mbox{prox}_{\eta F}(x_k) - \mbox{prox}_{\eta F}(x^*) \| \leq \|x_k - x^*\|.$
Therefore, using \eqref{t1}, we may claim that $\mathbb E[\|z_{k,j}-z^*_k\|^2\uss{ \ \mid \uss{\mathcal{F}_k}}]\leq \tfrac{\hat a^2\|x_k-x^*\|^2+\hat b^2}{j}$, where $\hat a^2=4+4M_1^2+2M_2^2$ and $\hat b^2= (4M_1^2+2M_2^2) \|x^*\|^2+M_3^2.$
 \end{proof}}

{\bf Proof of \sc Theorem \ref{th-rate-itercomp-sc-sgd-general}. } 
{\begin{proof} (i) By using Theorem 3.10 in \cite{bubeck2015convex} to bound $\|\bar x_{k+1}-x^*\|^2\leq q\|x_k-x^*\|^2$, where $\tilde \kappa= \tfrac{\eta\mu+1}{\eta\mu}$, $q=1-\tfrac{1}{\tilde \kappa} = \tfrac{1}{\eta\mu+1} \in (0,1)$ if $\eta > 0$,  and $\gamma_k=\eta$, we may obtain the following \uss{where $(1+\delta) < \afo{\tfrac{1}{2q}+\tfrac{1}{2}}$.}
\blue{\begin{align}\label{vs-pm-1}
\notag
\mathbb{E}[\|x_{k+1}-x^*\|^2{ \ \mid \mathcal{F}_k}] & \leq {(1+\tfrac{1}{\delta})}\mathbb{E}[\|x_{k+1}-\bar{x}_{k+1}\|^2{ \ \mid \mathcal{F}_k}] + {(1+\delta)}\mathbb{E}[\|\bar{x}_{k+1}-x^*\|^2{ \ \mid \mathcal{F}_k}]  \\ \nonumber
    & \leq  {(1+\tfrac{1}{\delta})}\mathbb{E}[\|x_{k+1}-\bar{x}_{k+1}\|^2] + {(1+\delta)q}\mathbb{E}[\|{x}_{k}-x^*\|^2] \\ \nonumber
    & = {{(1+\tfrac{1}{\delta})}}\mathbb{E}[\|\tfrac{{\gamma_k}}{\eta}(x_{k}-z_{k,N_k})-\tfrac{{\gamma_k}}{\eta}(x_k-z^*_{k})\|^2 { \ \mid \mathcal{F}_k}] + {(1+\delta)q}{\|{x}_{k}-x^*\|^2} \\
\notag		& =   {(1+\tfrac{1}{\delta})}\tfrac{\gamma_k^2}{\eta^2}\mathbb{E}[\|(z_{k,N_k}-z^*_{k})\|^2 { \ \mid \mathcal{F}_k}] + {(1+\delta)q}{\|{x}_{k}-x^*\|^2} \\
            & = {(1+\tfrac{1}{\delta})}\mathbb{E}[\|(z_{k,N_k}-z^*_{k})\|^2 { \ \mid \mathcal{F}_k}] + {(1+\delta)q}\|{x}_{k}-x^*\|^2,
\end{align}}
 where \eqref{vs-pm-1} follows from ${\gamma_k}=\eta$. By Prop.~\ref{bd-SSG} , the first term on the right can be bounded as
\begin{align*}
    \mathbb{E}[\|(z_{k,N_k}-z^*_{k})\|^2 \uss{ \ \mid \mathcal{F}_k}] \leq \tfrac{\hat{a}^2 \uss{\|x_k-x^*\|^2} + \hat{b}^2}{N_k}, 
\end{align*}
where $N_k$ denotes the number of stochastic subgradient steps taken at major iteration $k$. \uss{Then by taking unconditional expectations, we have} 
\begin{align*}
    \mathbb{E}[\|x_{k+1}-x^*\|^2]  \leq \left(\uss{(1+\delta)q} + \tfrac{\uss{(1+1/\delta)} \hat{a}^2}{ N_k}\right) \mathbb{E}[\|{x}_{k}-x^*\|^2] + \tfrac{\afo{(1+1/\delta)}\hat{b}^2}{N_k}.
\end{align*}
Let $p_k \triangleq \uss{(1+\delta)q} + \tfrac{\uss{(1+1/\delta)}\hat{a}^2}{ N_k}$  and $N_k=\lfloor N_0 \rho^{-k}\rfloor$ for $k \geq 0$, where \afo{$N_0>\tfrac{(1+1/\delta)\hat a^2}{1-(1+\delta) q}$}. Note that $p_0<1$ and $\{p_k\}$ is a decreasing sequence based on the choice of $N_0$ and $\{N_k\}$. {We consider two cases.}

\noindent {\bf Case (a).} Let $\rho \neq p_0$ and $\rho \in (0,1)$. {In this instance, we obtain the following result.} 
 \begin{align*}
& \quad \mathbb{E}[\|x_{k+1}-x^*\|^2] \leq\mathbb{E}[\|x_{0}-x^*\|^2]{\prod_{i=0}^{k}}p_i  +\sum_{i=0}^{k}\left(\tfrac{\afo{(1+1/\delta)}{\hat{b}^2}{\prod_{j=0}^{i-1}}p_{k-j}}{N_{k-i}}\right)\\
&\leq p_0^{{k+1}}\mathbb{E}[\|x_{0}-x^*\|^2]+\tfrac{\rho^k{\afo{(1+1/\delta)}\hat{b}^2}}{N_0}\sum_{i=0}^{k}\left(\tfrac{p_0}{\rho}\right)^i\leq \mathcal{C} \left(\max\{\rho,p_0\}\right)^{{k+1}},\end{align*}
 where $\mathcal{C} \triangleq \left(\mathbb{E}[\|x_{0}-x^*\|^2]+\tfrac{\afo{(1+1/\delta)}{\hat{b}^2}/N_0}{1-\tfrac{\min\{\rho,p_0\}}{\max\{\rho,p_0\}}}\right).$

\noindent {\bf Case (b)}. Let $\rho = p_0$. { Consequently, we obtain the following result. 
 \begin{align*}
\mathbb{E}[\|x_{k+1}-x^*\|^2]  & \leq p_0^{k+1}\mathbb{E}[\|x_{0}-x^*\|^2]+\tfrac{p_0^{k+1}\afo{(1+1/\delta)}{\hat{b}^2}}{N_0} (k+1) = ap_0^{k+1} + b(k+1)p_0^{k+1}.   
\end{align*}
It can be shown that, {there exists $\hat p$ such that} $p_0 < \hat{p} < 1$. By analyzing $\displaystyle \max_{z \geq 0} z\left(\tfrac{p_0}{\hat p}\right)^z$, we may claim that 
$ kp_0^k < D \hat{p}^k$ for $k \geq 0$ and $\widehat D > \tfrac{1}{\ln(p_0/\hat{p})^e}.$
Consequently, for $\hat{p} \in (p_0,1)$ and $\widehat D > \tfrac{1}{\ln(p_0/\hat{p})^e}$, 
\begin{align*}
\mathbb{E}[\|x_{k+1}-x^*\|^2]  &\leq \mathcal{C}\hat{p}^{k+1}, \mbox{ where } \mathcal{C} \triangleq \left(\mathbb{E}[\|x_{0}-x^*\|^2]+ \tfrac{\afo{(1+1/\delta)}\hat{b}^2 \widehat D}{N_0}\right). 
\end{align*}}
(ii) Suppose $\rho = p_0$ and $\hat{p} \in (p_0, 1)$ and   to compute a vector $x_{K}$ satisfying $\mathbb{E}[\|x_{K}-x^*\|^2]\leq \epsilon$, we have $\mathcal{C}\hat{p}^{K}\leq \epsilon$ where $\mathcal{C}$ depends on $\hat{p}$. This  implies that  
$K = \lceil \log_{(1/  \hat{p})}(\mathcal{C}/\epsilon)\rceil.$ From the definition of $\hat{p}, p_0$, $q$ and by choosing \afo{$N_0=\tfrac{2(1+1/\delta)\hat a^2}{1-(1+\delta) q}$}, we obtain that 
{\begin{align*}
        \tfrac{1}{\log(1/\hat{p})} & = \tfrac{\log(1/p_0)}{\log(1/\hat{p})} \tfrac{1}{\log(1/p_0)} \leq  \tfrac{\log(1/p_0)}{\log(1/\hat{p})} \tfrac{1}{(1-p_0)} = \tfrac{\log(1/p_0)}{\log(1/\hat{p})}\tfrac{1}{1-\left(\uss{(1+\delta)}q+\tfrac{\uss{(1+1/\delta)} \hat{a}^2}{N_0}\right)} \\ 
                                   &  \uss{ \ \leq \ } \tfrac{\log(1/p_0)}{\log(1/\hat{p})}\afo{\left(\tfrac{1}{1-\left((1+\delta)q+\tfrac{1-(1+\delta)q}{2}\right)}\right) } = \tfrac{\log(1/p_0)}{\log(1/\hat{p})}\uss{\left(\tfrac{1}{\tfrac{1}{2}- \tfrac{(1+\delta)q}{2}}\right)}\leq \tfrac{\log(1/p_0)}{\log(1/\hat{p})}\uss{\left(\tfrac{1}{\tfrac{1}{4}-\tfrac{q}{4}}\right) } = \tfrac{4\log(1/p_0)}{\log(1/\hat{p})}\afo{\tilde \kappa},
\end{align*}
\uss{where the last inequality follows from $\tfrac{(1+\delta)q}{2} \leq \tfrac{1}{4}+\tfrac{q}{4}.$} }
Therefore, the iteration complexity is bounded  as ${\log(\mathcal{C}/\epsilon)/\log(1/p_0)} \leq  \left(\tfrac{4\log(1/p_0)}{\log(1/\hat{p})}\right)\tilde \kappa {\log} (\mathcal{C}/\epsilon)$. 
{Similarly, if $\rho \neq p_0$, since $\mathcal{C} \max\{\rho,p_0\}^k \leq \epsilon$, the iteration complexity is $\mathcal{O}(\tilde{\kappa}\log(\mathcal{C}/\epsilon))$.}

(iii) {Suppose $\rho = p_0$ and $\hat{p} \in (p_0,1)$}. To obtain the oracle complexity, we require $\sum_{k=1}^{{K}} N_k$ gradients where $K = \lceil \log_{(1/  \hat{p})}(\mathcal{C}/\epsilon)\rceil$. 
{\begin{align*}
 & \quad N_0\sum_{k=1}^{{K}} \rho^{-k} 
 \leq \tfrac{N_0}{\left(\frac{1}{{{\rho}}} -1\right)}\left(\tfrac{1}{{\rho}}\right)^{2+{K}}  \leq \tfrac{N_0}{\left(\frac{1}{{{\rho}}} -1\right)}\left(\tfrac{1}{{\rho}}\right)^{3+\log_{(1/  {\hat{p}})}\left(\mathcal C/\epsilon\right)}  \leq \tfrac{N_0}{{\rho}^2(1-{{\rho}})} \left(\tfrac{1}{\rho}\right)^{\log_{1/\hat{p}} (\mathcal C/\epsilon) } \\
	& =  \tfrac{N_0}{{\rho}^2(1-{{\rho}})} \left(\tfrac{1}{\rho}\right)^{\log_{1/\rho} (\mathcal C/\epsilon) \log_{1/\hat{p}} (1/\rho)} =
\tfrac{N_0}{{\rho}^2(1-{{\rho}})} \left(\tfrac{\mathcal C}{\epsilon}\right)^{\log_{1/\hat{p}}(1/\rho)} =  \left(\tfrac{p_0^2}{\rho^2}\right)\tfrac{N_0}{{p_0}^2(1-{{\rho}})} \left(\tfrac{\mathcal C}{\epsilon}\right)^{\log_{1/\hat{p}}(1/\rho)} \\
& \leq  
   \left(\tfrac{p_0^2}{\rho^2}\right)\afo{\tfrac{16(1+1/\delta) \hat{a}^2}{(1-q)^2}} \left(\tfrac{\mathcal C}{\epsilon}\right)^{\log_{\hat{p}}(1/\rho)}. 
 \end{align*}

It follows that the oracle complexity is $\mathcal{O}\left(\afo{\tilde{\kappa}^3}\left(\tfrac{\mathcal C}{\epsilon}\right)^{\log_{1/\hat{p}}(1/\rho)}\right).$ {Similarly, it can be shown that when $\rho > p_0$ (or $\rho < p_0$), the oracle complexity is $\mathcal{O}\left(\tfrac{\afo{\tilde{\kappa}^3} \mathcal{C}}{\epsilon}\right)$ (or $\mathcal{O}\left(\afo{\tilde{\kappa}^3}\left(\tfrac{ \mathcal{C}}{\epsilon}\right)^{\log_{1/p_0}(1/\rho)}\right)$)}.}
\end{proof} }

\uss{{\bf Proof of } {\sc Lemma~\ref{smooth_f}}.
    Since $\uss{\f}_{\eta}(x,\omega) \leq \uss{\f}(x,\omega) \leq \uss{\f}_{\eta}(x,\omega) + \uss{\eta}\beta(\omega)$ for any $x$, by taking expectations on both sides and recalling that $\mathbb{E}[\beta(\omega)] \leq \tilde{\beta}$, we have that 
\begin{align*}
    \mathbb{E}[\uss{\f}_{\eta}(x,\omega)] \leq\mathbb{E}[ \uss{\f}(x,\omega)] \leq\mathbb{E}[ \uss{\f}_{\eta}(x,\omega)] + \uss{\eta}\mathbb{E}[\beta(\omega)] \qquad \forall x. 
\end{align*}
    Suppose $f_{\eta}$ is defined as 
\begin{align}
    f_{\eta}(x) \triangleq \mathbb{E}[\uss{\f}_{\us{\eta}}(x,\omega)],
\end{align}
implying that $f_{\eta}(x) \leq f(x) \leq f_{\eta}(x) + \uss{\eta} \tilde{\beta}.$ In addition, since
$\|\nabla_x \usv{\f}_{\eta}(x\usv{,\omega})-\nabla_x \usv{\f}_{\eta}(y\usv{,\omega}) \| \leq \tfrac{\alpha(\omega)}{\eta} \|x-y\|,$ for all  $x, y$,  
by taking expectations on both sides and invoking Jensen's inequality, we have that  
\begin{align*}
\|\nabla_x f_{\eta}(x)-\nabla_x f_{\eta}(y) \| 
& = \|\nabla_x \mathbb{E}[\uss{\f}_{\eta}(x,\omega)]-\nabla_x \mathbb{E}[\uss{\f}_{\eta}(y,\omega)] \|\\ 
    {\scriptsize (\mbox{Jensen's inequality})} \qquad & \leq \mathbb{E}\left[\|\nabla_x\uss{\f}_{\eta}(x,\omega)-\nabla_x \uss{\f}_{\eta}(y,\omega) \| \right]\\
    {\scriptsize (\mbox{$\afo{\f_{\eta}}(\cdot,\omega)$ is $\tfrac{\alpha(\omega)}{\eta}$-smooth})}\qquad                                & \leq \mathbb{E}\left[\frac{\alpha(\omega)}{\eta}\right] \|x-y\| \\
                                              & \leq \tfrac{\tilde{\alpha}}{\eta} \|x-y\| \quad \forall x,y, 
\end{align*}
\blu{where in the first inequality, {we use} Theorem 7.47 in \cite{shapiro09lectures} (interchangeability of the derivative and the expectation)}. It follows that $f_{\eta}$ is $\tilde{\alpha}/\eta$-smooth. We may conclude that \us{$(\tilde{\alpha},\tilde{\beta})$-smoothability of $f$ follows}.\qed
}



\end{document}